\documentclass[11pt,reqno]{amsart}
\usepackage{amssymb}
\usepackage{enumerate}
\usepackage{amsmath,amssymb,amsfonts,amsthm,
	latexsym, amscd, epsfig, epic,color}
\usepackage{epstopdf}
\usepackage{mathtools}
\usepackage{youngtab}
\usepackage{extarrows}
\usepackage{enumitem}
\usepackage{enumerate}
\usepackage{mathrsfs}
\usepackage{eucal}
\usepackage{eufrak}
\usepackage{xspace}
\usepackage{a4wide}
\usepackage{colordvi}
\usepackage{comment}
\usepackage{hyperref,eucal}
\usepackage{tabularx}
\usepackage{tikz}\usetikzlibrary{matrix,arrows}
\usepackage{ytableau} \usepackage{chemfig}
\usepackage{float}
\usepackage{youngtab}
\usepackage{mathrsfs}
\usepackage{caption}
\usepackage{epsfig}
\usepackage{stfloats}
\usepackage{booktabs}
\usepackage{changepage}
\usepackage{url}
\usepackage{graphicx}
\theoremstyle{plain}

\theoremstyle{plain}

\newtheorem{theorem}{Theorem}[section]
\newtheorem{corollary}{Corollary}[section]

\newtheorem{lemma}{Lemma}[section]
\newtheorem{proposition}{Proposition}[section]

\newtheorem{definition}{Definition}[section]

\theoremstyle{example}
\newtheorem{example}{Example}[section]

\theoremstyle{openproblem}

\def\T{\CMcal{T}}
\def\M{\CMcal{M}}

\def\A{\mathcal{A}}

\def\S{\CMcal{S}}
\def\Q{\CMcal{Q}}
\def\inv{\mathrm{inv}}
\def\coinv{\mathrm{coinv}}
\def\quinv{\mathrm{quinv}}

\def\maj{\mathrm{maj}}

\def\Des{\mathrm{Des}}
\def\des{\mathrm{des}}
\def\South{\mathrm{South}}
 \def\arm{\mathrm{arm}} \def\leg{\mathrm{leg}}   

 \def\Z{\mathbb{Z}}
\def\dg{\mathrm{dg}}
\def\minv{\mathrm{mixinv}}
\def\w{\mathbf{w}}

\def\quadcoinv{\mathrm{quadcoinv}}
\def\quadinv{\mathrm{quadinv}}

\def\quadinv{\mathrm{quadinv}}  \def\rev{\mathrm{rev}}

\makeatletter
\@namedef{subjclassname@2020}{\textup{2020} Mathematics Subject Classification}
\makeatother

\theoremstyle{remark}

\newtheorem{remark}{Remark}
\numberwithin{equation}{section}
\numberwithin{figure}{section}

\title[Combinatorial formulas by superizations]{Combinatorial formulas for Macdonald polynomials by superizations}

\author{Emma Yu Jin}
\address{School of Mathematical Sciences, Xiamen University, Xiamen 361005, China}
\email{yjin@xmu.edu.cn}

\author{Xiaowei Lin}
\address{School of Mathematical Sciences, Xiamen University, Xiamen 361005, China}
\email{linxiaoweiqing@126.com}



\begin{document}
	
	\begin{abstract}
		In this paper, we derive new combinatorial formulas for symmetric Macdonald polynomials $P_{\lambda}(X;q,t)$ and non-symmetric Macdonald polynomials $E_{\gamma}(X;q,t)$, in terms of several new statistics and the major index, for a partition $\lambda$ and a weak composition $\gamma$.
		
		Compared to previous formulas, these new formulas contain the fewest terms and lead to explicit $(q,t)$-formulas for the coefficients in the monomial expansion of $P_{\lambda}(X;q,t)$. In particular, the combinatorial formula for $E_{\gamma}(X;q,t)$ extends the one for $E_{\lambda}(X;q,t)$ indexed by a partition $\lambda$, due to Corteel, Mandelshtam and Williams (2022).
		Three existing formulas for $P_{\lambda}(X;q,t)$ established by Corteel, Mandelshtam and Williams (2022), by Corteel, Haglund, Mandelshtam, Mason and Williams (2022), and by Mandelshtam (2025) are recovered.
		
		Our proof relies on two new statistics on super fillings, employing the superization formulas of Haglund--Haiman--Loehr (2005) and Ayyer--Mandelshtam--Martin (2023), together with our recent approach to modified Macdonald polynomials.
		
		
	\end{abstract}	
	
	\maketitle
	
	\section{Introduction and main results}\label{S:intro}
	{\em \color{blue} Symmetric Macdonald polynomials} $P_{\lambda}(X;q,t)$ \cite{Mac88,Mac95} indexed by partitions $\lambda$ are symmetric and orthogonal polynomials in infinitely many variables $X=(x_1,x_2,\ldots)$ with coefficients in the field $\mathbb{Q}(q,t)$ of rational functions in two variables $q$ and $t$. Several important families of symmetric polynomials are specializations of Macdonald polynomials, such as Schur functions $s_{\lambda}(X)=P_{\lambda}(X;t,t)$, Hall-Littlewood polynomials $P_{\lambda}(X;0,t)$ and $q$-Whittaker polynomials $P_{\lambda}(X;q,0)$.
	
	\begin{definition}[{\em \color{blue} Symmetric Macdonald polynomials}]\label{Def:P}
		For a partition $\lambda$, the Macdonald polynomial $P_{\lambda}(X;q,t)$ is defined as the unique symmetric function over the field $\mathbb{Q}(q,t)$ such that
		\begin{align}\label{E:P_mono}
			P_{\lambda}(X;q,t)=m_{\lambda}(X)+\sum_{\mu<\lambda}c_{\lambda\mu}(q,t)m_{\mu}(X),
		\end{align}
		for some $c_{\lambda\mu}(q,t)\in\mathbb{Q}(q,t)$, and $\langle P_{\lambda},P_{\mu}\rangle=0$
		if $\lambda\ne \mu$.
	\end{definition}
	Here $m_{\lambda}(X)$ is the monomial symmetric function and $<$ denotes the dominance order on weak compositions. A weak composition $\alpha=(\alpha_1,\ldots,\alpha_n)$ is a finite sequence of non-negative integers, and we write $|\alpha|=\alpha_1+\cdots +\alpha_n$. The partial order $\le$ is defined on weak compositions $\alpha$ and $\beta$ satisfying $|\alpha|=|\beta|$ so that
	\begin{align}\label{E:domi}
		\alpha\le \beta \quad\,\mbox{if and only if}\quad\,\alpha_1+\cdots+\alpha_k\le \beta_1+\cdots+\beta_k
	\end{align}
	for all $1\le k\le n$. The inner product $\langle\cdot,\cdot\rangle$ in Definition \ref{Def:P} is defined in terms of power-sum symmetric functions $p_{\lambda}(X)$, namely,
	\begin{align*}
		\langle p_{\lambda},p_{\mu}\rangle=\delta_{\lambda\mu}\,z_{\lambda}\,\prod_{i\ge 1}\frac{1-q^{\lambda_i}}{1-t^{\lambda_i}},
	\end{align*}
	where $z_{\lambda}=\prod_{i\ge 1}i^{m_i}m_i!$ and $m_i$ counts the multiplicity of part $i$ in the partition $\lambda$.
	\begin{example}\label{Example:p}
		For $\lambda\in \{(1,1),(2)\}$ and $X=(x_1,x_2,\ldots)$,
		\begin{align*}
			P_{11}(X;q,t)=m_{11}(X) \,\,\mbox{ and }\,\,
			P_{2}(X;q,t)=m_{2}(X)+\frac{(1+q)(1-t)}{1-qt}m_{11}(X).
		\end{align*}
	\end{example}
	Subsequently, non-symmetric Macdonald polynomials $E_{\gamma}(X;q,t)$, indexed by weak compositions $\gamma\in \mathbb{N}^n$ in $n$ variables $X=(x_1,\ldots,x_n)$, were introduced and developed by Opdam, Cherednik, Macdonald and Haiman \cite{Cherednik:95,Haiman:06,Mac94,Opdam:95}.
	
	Let $\gamma^+$ be the partition formed by the positive parts of a weak composition $\gamma$. The symmetric Macdonald polynomial $P_{\lambda}(X;q,t)$ is the unique symmetric polynomial in the space $\mathbb{Q}(q,t)\{E_{\gamma}(X;q,t):\gamma^+=\lambda\}$ such that $[x^{\lambda}]P_{\lambda}(X;q,t)=1$. We adopt the notations from \cite{HHL08} to define non-symmetric Macdonald polynomials. Let
	\begin{align*}
		\Delta=\prod_{i<j}\frac{(x_ix_j^{-1};q)_{\infty}(qx_jx_i^{-1};q)_{\infty}}
		{(tx_ix_j^{-1};q)_{\infty}(tqx_jx_i^{-1};q)_{\infty}}=\prod_{i<j}\prod_{k\ge 0}
		\frac{(1-q^kx_ix_j^{-1})(1-q^{k+1}x_jx_i^{-1})}{(1-tq^kx_ix_j^{-1})(1-tq^{k+1}x_jx_i^{-1})},
	\end{align*}
	and let $\Delta_1=\Delta/([x^0]\Delta)$. The {\em \color{blue} Cherednik inner product} is defined by
	\begin{align*}
		\langle f,g \rangle_{q,t}=[x^0](f(X;q,t)g(X^{-1};q^{-1},t^{-1})\Delta_1),
	\end{align*}
	where $X^{-1}=(x_1^{-1},\ldots,x_n^{-1})$.
	We define a partial order $\preceq$ on weak compositions $\gamma$ and $\nu$ with $|\gamma|=|\nu|$, so that
	{
		\[
		\nu\prec\gamma
		\quad\Longleftrightarrow\quad
		\nu^+<\gamma^+
		\quad\text{or}\quad
		(\nu^+=\gamma^+\ \text{and}\ \gamma<\nu),
		\]}
	where $<$ is the dominance order defined in (\ref{E:domi}).
	\begin{definition}[{\em \color{blue} Non-symmetric Macdonald polynomials}]\label{Def: nonsym}
		{\em Non-symmetric Macdonald polynomials} $E_{\gamma}(X;q,t)$, indexed by weak compositions $\gamma$ in $n$ variables $X=(x_1,\ldots,x_n)$, are the unique polynomials satisfying
		\begin{align}\label{E:mon1}
			E_{\gamma}(X;q,t)=x^{\gamma}+\sum_{\nu\prec \gamma}d_{\nu\gamma}(q,t)x^{\nu},
		\end{align}
		for some $d_{\nu\gamma}(q,t)\in \mathbb{Q}(q,t)$, and $	\langle E_{\gamma}, E_{\nu}\rangle_{q,t}=0$ for $\gamma\ne \nu$.
	\end{definition}
	\begin{example}\label{Example:e}
		For $\gamma\in \{(1,1,0),(0,1,1),(1,0,1)\}$ and $X=(x_1,x_2,x_3)$, we have
		\begin{align*}
			E_{110}(X;q,t)&=x_1x_2,\quad
			E_{011}(X;q,t)=x_2x_3+\frac{1-t}{1-qt}(x_1x_2+x_1x_3),\\
			E_{101}(X;q,t)&=x_1x_3+\frac{1-t}{1-qt^2}x_1x_2.
		\end{align*}
	\end{example}
	Considerable work has been devoted to finding formulas for Macdonald polynomials. However, to the best of our knowledge, no formula has been known to be both combinatorial and explicit for the coefficients $c_{\lambda\mu}(q,t)=[m_{\mu}]P_{\lambda}$. At present, there are primarily three combinatorial models related to both symmetric and non-symmetric Macdonald polynomials.
	\begin{enumerate}
		\item Young tableaux model: Haglund, Haiman and Loehr proved the first combinatorial formulas for Macdonald polynomials $P_{\lambda}(X;q,t)$ and $E_{\gamma}(X;q,t)$, in terms of tableaux statistics inversions and the major index \cite{HHL04,HHL08}.
		Mandelshtam \cite{O:24} introduced probabilistic operators acting on non-attacking tableaux to prove a new compact formula for $P_{\lambda}(X;q,t)$.\\
		\item Alcove walk model: Ram and Yip \cite{RY:11} established uniform combinatorial formulas for Macdonald polynomials for all Lie types using the alcove walk model developed by Schwer \cite{Schwer} and Ram \cite{Ram}. Lenart improved the Ram--Yip formula for $P_{\lambda}(X;q,t)$ by grouping terms into equivalence classes and summing their contributions, under the assumption that all parts of $\lambda$ are distinct \cite{Lenart}. \\
		\item Multiline queues model: Corteel, Mandelshtam and Williams introduced a weighted generating function for multiline queues to prove a new combinatorial formula for $P_{\lambda}(X;q,t)$ \cite{CMW22}. In particular they established the multiline queue formula for $E_{\gamma}(X;q,t)$ when $\gamma$ is a partition. Subsequently, Corteel, Haglund, Mandelshtam, Mason and Williams \cite{CHO22} derived a tableau formula for $P_{\lambda}(X;q,t)$ by applying the multiline queue formula and
		its relation to permuted-basement Macdonald polynomials \cite{Alex:16,Fer:11} as stated in \cite{CMW22}.
	\end{enumerate}
	On the other hand, non-combinatorial formulas for Macdonald polynomials include a matrix product formula for $P_{\lambda}(X;q,t)$ by Cantini, de Gier and Wheeler \cite{CdGW:15}, and sooner after, this yields a sum formula for $P_{\lambda}(X;q,t)$ containing the action of Hecke generators on the ring of polynomials \cite{dGW:15}. Meanwhile, Borodin and Wheeler \cite[Theorem 1.3]{BW:22} constructed the partition functions of integrable vertex models to deduce a formula for $E_{\gamma}(X;q,t)$, which turns out to be equivalent to the formula given by Haglund, Haiman and Loehr \cite[Theorem 3.5.1]{HHL08}.
	
	The objective in this paper is to develop new effective formulas for $P_{\lambda}(X;q,t)$ and $E_{\gamma}(X;q,t)$. Our main results are new combinatorial formulas for both of them, which give rise to computable formulas for the coefficients $c_{\lambda\mu}(q,t)=[m_{\mu}]P_{\lambda}$. Before stating the main results, we present the starting point of our research line on Macdonald polynomials.
	
	Our journey into Macdonald polynomials begins with a combinatorial formula for modified Macdonald polynomials $\widetilde{H}_{\lambda}(X;q,t)$, a variant of symmetric Macdonald polynomials introduced by Garsia and Haiman \cite{GH93}. Haglund, Haiman and Loehr \cite{HHL04} remarkably related these polynomials to the combinatorial statistics on fillings of a Young diagram.
	\begin{theorem}\cite[Theorem 2.2]{HHL04}
		For a partition $\lambda$, let $\T(\lambda)$ be the set of positive fillings of the Young diagram $\dg(\lambda)$ of $\lambda$. Then,
		\begin{align}\label{E:mmp1}
			\widetilde{H}_{\lambda}(X;q,t)=\sum_{\tau\in \T(\lambda)}x^{\tau}t^{\maj(\tau)}q^{\inv(\tau)}.
		\end{align}
	\end{theorem}
	The definitions of combinatorial statistics such as inversion and the major index are postponed to Section \ref{S:2}. Very recently, both authors introduced a set $\A$ of sixteen statistics on positive fillings and established new combinatorial formulas for modified Macdonald polynomials $\widetilde{H}_{\lambda}(X;q,t)$; see Theorem \ref{T:eta}. Two fillings (or tableaux) $\sigma$ and $\tau$ of a given Young diagram are called {\em \color{blue}row-equivalent}, denoted by $\sigma\sim\tau$, if for every $i$, the multisets of entries in the $i$th rows of $\sigma$ and $\tau$ coincide. Let $[\sigma]$ denote the row-equivalence class of $\sigma$.
	\begin{theorem}\cite[Theorem 1.1]{JL24}\cite[Theorem 2]{JL25}\label{T:eta}
		For a partition $\lambda$, let $\sigma$ be a positive filling of the Young diagram of $\lambda$. Then, for any statistic $\eta\in \A$, where $\A$ denotes the set of sixteen statistics in Definition \ref{Def:eta3}, the following identity holds:
		\begin{align}\label{eqthm1}
			\sum_{\tau\in[\sigma]}t^{\maj(\tau)}q^{\eta(\tau)}=\sum_{\tau\in[\sigma]}t^{\maj(\tau)}q^{\inv(\tau)}
			=\sum_{\tau\in[\sigma]}t^{\maj(\tau)}q^{\quinv(\tau)}.
		\end{align}
		Consequently,
		\begin{align}\label{E:mmp21}
			\widetilde{H}_{\lambda}(X;q,t)=\sum_{\tau\in \T(\lambda)}x^{\tau}t^{\maj(\tau)}q^{\quinv(\tau)}
			=\sum_{\tau\in \T(\lambda)}x^{\tau}t^{\maj(\tau)}q^{\eta(\tau)}.
		\end{align}
	\end{theorem}
	Inspired by Martin's multiline-queue formula for the stationary distribution of multitype asymmetric simple exclusion processes \cite[Theorem 3.4]{Martin:20}, Corteel, Haglund, Mandelshtam, Mason and Williams \cite{CHO22} introduced the queue inversion statistic ($\quinv$). The first equality of (\ref{E:mmp21}) was first proved by Ayyer, Mandelshtam and Martin \cite[Equation (2)]{AMM23}.
	
	Our first main result (Theorems \ref{T:main1} and \ref{cor:1}) is two new combinatorial formulas for $P_{\lambda}(X;q,t)$ via superization, building on (\ref{E:mmp21}). For a partition $\lambda=(\lambda_1,\ldots,\lambda_k)$, let $\lambda'$ be the transpose of $\lambda$. Let
	\begin{align}\label{E:nlambda}
		n(\lambda)&=\sum_{1\le i\le k}(i-1)\lambda_i=\sum_{1\le i\le \lambda_1}\binom{\lambda'_i}{2},
	\end{align}
	Given a statistic $\vartheta$, define its complement $\overline{\vartheta}$ by
	\begin{align}\label{E:bar_eta}
		\overline{\vartheta}(\sigma)=n(\lambda)-\vartheta(\sigma).
	\end{align}
	Furthermore, let $\alpha(\sigma)$ be the number of quadruples $(z,w,u,v)$ with repetitions that are counted by $\bar{\eta}(\sigma)$, and set
	\begin{align}\label{E:eta2}
		\eta^{\circ}(\sigma)=\bar{\eta}(\sigma)-\alpha(\sigma).
	\end{align}
	Given a filling $\tau$, let
	\begin{align}
		\label{E:cqt}c_{\tau}(q,t)=\prod_{\tau(u)\ne\tau(\South(u)) \atop \textrm{ and}\, u \,\not\in \,\textrm{row }1}  \frac{1-t}{1-q^{\leg(u)+1} t^{\overline{\arm}(u)+1}}.
	\end{align}
	\begin{theorem}\label{T:main1}
		For a partition $\lambda$ and for any statistic $\eta\in \A^+$, where $\A^+\subseteq \A$ denotes the set of eight statistics in Definition \ref{Def:eta1} and Lemma \ref{L:setSi},
		we have
		\begin{align}\label{eqthm3.3}
			P_{\lambda}(X;q,t)=\sum_{\sigma}
			q^{\maj(\sigma)}t^{\eta^{\circ}(\sigma)}c_{\sigma}(q,t)\,x^{\sigma}, 
		\end{align}
		where the sum is taken over all non-attacking and top-row increasing tableaux $\sigma$ of the Young diagram $\dg(\lambda')$ of $\lambda$.
	\end{theorem}
	By applying the compact formulas for modified Macdonald polynomials \cite[Theorem 16]{JL25}, we find a compressed version of (\ref{eqthm3.3}); see also Examples \ref{Example:mainT} and \ref{Example:main2} for a comparison.
	\begin{theorem}\label{cor:1}
		For any statistic $\eta\in \A^+$ such that $(z>w>u>v)\in\S$, we have
		\begin{align}\label{E:main_comp}
			P_{\lambda}(X;q,t)=\sum_{\sigma}q^{\maj(\sigma)}t^{\eta^{\circ}(\sigma)}d_{\sigma}(t)c_{\sigma}(q,t)\,x^{\sigma},
		\end{align}
		which is summed over all sorted non-attacking tableaux of the Young diagram $\dg(\lambda')$ of $\lambda$, and $d_{\sigma}(t)$ defined in (\ref{E:dsigma}) is a product of $t$-multinomials. Consequently, when $X=(x_1,\ldots,x_n)$ and $\ell(\lambda)\le n$, we have
		\begin{align}\label{E:vsP}
			c_{\lambda\mu}(q,t)=[m_{\mu}]P_{\lambda}
			&=\sum_{\{\nu \}}q^{n(\lambda')}t^{n(\lambda)-\chi(\nu)-\sum_j\binom{\lambda_j'-\lambda_{j+1}'}{2}} \notag\\
			&\qquad \times \prod_{1\leq i< j\leq \lambda_1}
			\phi_{\nu_{i+1,j}| \nu_{i,j}}(q^{-(j-i)}t^{-(\lambda_i'-\lambda_j')},t^{-1},\lambda_j'-\lambda_{j+1}')
		\end{align}
		with the sum running over sequences $\nu_{i,j}=(\nu_{i,j}^{1} \le \cdots \le \nu_{i,j}^{n})$ of nonnegative integers, $1\le i\le j\le \lambda_1$, subject to the conditions
		(\ref{eq5})--(\ref{eq8}), {where
			\begin{align*}
				\chi(\nu)&=\sum_{k=1}^{n}\sum_{1\le i\le j<\ell\le\lambda_1} (\nu_{i,j}^{k}-\nu_{i,j}^{k-1})(\nu_{i,\ell}^{k-1}-\nu_{i+1,\ell}^{k}).
		\end{align*} }
	\end{theorem}
	Our proof of Theorem \ref{T:main1} is not a direct application of superization to the statistics $\eta\in\A$, primarily for the following reasons. Unlike the inversion $(\inv)$ and queue inversion ($\quinv$) statistics in \eqref{eqthm1}, each statistic $\eta$ counts a mixture of quadruples and triples under specific conditions. Consequently, it is difficult to extend these statistics to super fillings or to find a suitable reading order for fillings that simultaneously tracks the $\eta$ and $\maj$ statistics after standardization.
	
	To overcome this difficulty, we introduce a new statistic on super fillings called {\em \color{blue} quadruple coinversion} ($\quadcoinv$), whose complement $\overline{\quadcoinv}$ reduces to the $\coinv''$ statistic on non-attacking fillings defined by Corteel, Mandelshtam and Williams \cite[Remark 5.17]{CMW22}. An advantage of the statistic $\quadcoinv$ is that the flip operator associated with it preserves the non-attacking property, which enables us to apply the superization-compression approach.
	
	Our second main result (Theorem \ref{thmE}) is a compact formula for non-symmetric Macdonald polynomials $E_{\gamma}(X;q,t)$ indexed by a weak composition $\gamma$, extending the multiline-queue formula for $E_{\lambda}(X;q,t)$ for a partition $\lambda$ (see Corollary \ref{cor:10.1}), which is due to Corteel, Mandelshtam and Williams \cite[Proposition 1.10 and Theorem 5.9]{CMW22}.
	\begin{theorem}\label{thmE}
		For $X=(x_1,\ldots,x_n)$ and a weak composition $\gamma\in \mathbb{N}^n$, let $\vartheta\in \{\overline{\quadinv},\eta^{\circ}\}$ for each $\eta\in \A^+$. Then,
		\begin{align}\label{E:quadinv}
			E_{\gamma}(X;q,t)&=\prod_{u\in\dg'(\gamma)}(1-q^{\leg(u)+1} t^{\widetilde{\arm}(u)+1})^{-1}
			\sum_{\sigma:\gamma\rightarrow [n] \atop \widehat{\sigma}: \,\text{non-attacking}}x^{\sigma} q^{\maj(\widehat{\sigma})}  t^{\vartheta(\widehat{\sigma})} \notag\\
			&\qquad \times\prod_{u\in\dg'(\gamma)\atop \widehat{\sigma}(u)\ne\widehat{\sigma}(\South(u))}  (1-t)
			\prod_{u\in\dg'(\gamma)\atop \widehat{\sigma}(u)=\widehat{\sigma}(\South(u))} (1-q^{\leg(u)+1} t^{\arm''(u)+1}).
		\end{align}
	\end{theorem}
	The proof of Theorem \ref{thmE} starts from the combinatorial formula for $E_{\gamma}(X;q,t)$ by Haglund, Haiman and Loehr \cite[Theorem 3.5.1]{HHL08}, and the compression of this formula is achieved by making use of the second new statistic {\em \color{blue} quadruple inversions} ($\quadinv$) and its relation to $\eta$.
	
	Moreover, using superization and \eqref{eqthm1} of Theorem \ref{T:eta},
	we provide alternative proofs for two additional combinatorial formulas of $P_{\lambda}(X;q,t)$ that  appear in \cite{CHO22,O:24}.
	\begin{theorem}\cite[Equation (4.10)]{CHO22}\cite[Theorem 1.1]{O:24}\label{thm:twoP}
		For a partition $\lambda$,
		\begin{align}
			P_{\lambda}(X;q,t)
			&=\sum_{\substack{\sigma\, \quinv\textrm{-non-attacking}\\ \textrm{ and bottom-row increasing}}}
			q^{\maj(\sigma)}t^{\coinv^*(\sigma)}x^{\sigma} \notag\\
			&\qquad\qquad \times \prod_{\sigma(u)\ne\sigma(\South(u)) \atop \textrm{ and }u \,\not\in\, \textrm{row }1}  \frac{1-t}{1-q^{\leg(u)+1} t^{\widetilde{\arm}(u)+1}}.\label{macdonald3}\\
			&=\sum_{\substack{\sigma\textrm{ $\quinv$-non-attacking}\\ \textrm{ and top-row increasing}}}
			q^{\maj(\sigma)}t^{\overline{\mathrm{quinv}}(\sigma)}x^{\sigma}\notag \\
			&\qquad\qquad\times \prod_{\sigma(u)\ne\sigma(\South(u)) \atop \textrm{ and }u\, \not\in \,\textrm{row }1}  \frac{1-t}{1-q^{\leg(u)+1} t^{\arm(\South(u))+1}}.\label{macdonald2}
		\end{align}
	\end{theorem}
	Equation (\ref{macdonald3}) is the tableau formula for $P_{\lambda}(X;q,t)$ mentioned in the combinatorial model (3). In parallel Mandelshtam \cite{O:24} introduced probabilistic operators that act on queue inversion non-attacking fillings to establish (\ref{macdonald2}). In this paper, we provide unified proofs of (\ref{macdonald3}) and (\ref{macdonald2}) as applications of the second equality in (\ref{eqthm1}).
	
	When all parts of partition $\lambda$ are distinct, the combinatorial formulas of symmetric Macdonald polynomials discussed above -- namely,
	\eqref{eqthm3.3}, \eqref{eqthm2.1}, \eqref{eqthm2.2}, \eqref{macdonald3} and \eqref{macdonald2} reduce to the formula given by Lenart \cite{Lenart}. For partitions $\lambda$ with repeated parts, \eqref{E:main_comp} contains the fewest terms, as the sorted non-attacking condition is more restrictive than the $\quinv$-non-attacking condition.
	
	The remainder of this paper is organized as follows: Section \ref{S:2} provides the necessary preliminaries on combinatorial statistics, super fillings and superization. In Section \ref{S:3}, we define various flip operators that act on pairs of adjacent columns of equal height. Section \ref{S:reduced} introduces reduced expressions of a permutation and connects them to the action of flip operators. The proofs of Theorems \ref{T:main1} and \ref{cor:1} are presented in Sections \ref{S:roadmap} -- \ref{S:7}. Finally, we present an inversion analogue of Theorem \ref{T:main1} and prove Theorems \ref{thmE} and \ref{thm:twoP} in Sections \ref{S:remark} -- \ref{S:8}, respectively. Finally, we make concluding remarks in Section \ref{S:finalremark}.

	{\color{blue} Declaration on the use of AI}: All mathematical proofs presented in this article originate from the authors' own work. The authors used ChatGPT (OpenAI, GPT-6 Astra) for proofreading and English language polishing. The authors take full responsibility for accuracy and content of the manuscript.
	

	\section{Preliminaries on combinatorial statistics}\label{S:2}

	A {\em \color{blue} partition} $\lambda=(\lambda_1\ge \cdots\ge \lambda_n>0)$ of a positive integer $N$ is a weakly decreasing sequence of positive integers satisfying $\lambda_i\ge \lambda_{i+1}$ for all $1\le i<n$ and $|\lambda|=\lambda_1+\cdots+\lambda_n=N$. Each integer $\lambda_i$ is called the $i$th part of $\lambda$. The number of parts is the length of $\lambda$, denoted by $\ell(\lambda)$.
	
	The {\em \color{blue}Young diagram} of $\lambda$, denoted by $\dg(\lambda)$, is a left-justified array of boxes where the $i$th row contains $\lambda_i$ boxes. We adopt the Cartesian coordinate system where a box has coordinates $(i,j)$ if it lies in the $i$th row from the bottom and the $j$th column from the left.
	The {\em \color{blue}conjugate} (or {\em \color{blue}transpose}) of $\lambda$, denoted by $\lambda'$, is the partition whose Young diagram is obtained by reflecting $\dg(\lambda)$ across the main diagonal (the boxes with coordinates $(i,i)$); see Figure \ref{F:f2}. We simply write $\dg'(\lambda)=\dg(\lambda')$.
	\begin{figure}[ht]
		\centering
		\includegraphics[scale=0.9]{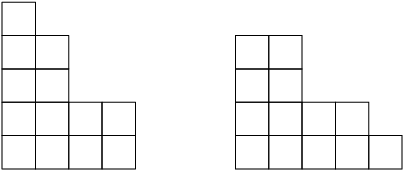}
		\caption{\label{F:f2} The Young diagram $\dg(\lambda)$ of $\lambda=(4,4,2,2,1)$ and its conjugate Young diagram $\dg'(\lambda)$ from left to right.}
	\end{figure}
	
	
	Define the alphabet $\mathcal{A}=\mathbb{Z}_{+}\cup \mathbb{Z}_{-}=\{1,\overline{1},2,\overline{2},\ldots \}$ of positive letters $i$ and negative letters $\bar{i}$. A {\em \color{blue} super filling} of $\dg'(\lambda)$ is a function $\sigma:\dg'(\lambda)\rightarrow \mathcal{A}$, and a {\em \color{blue} positive filling} of $\dg'(\lambda)$ is a function $\sigma:\dg'(\lambda)\rightarrow \mathbb{Z}_{+}$. In both cases, $\sigma$ may be referred to as a filling or a tableau. Let
	$p(\sigma)$ and $m(\sigma)$ denote the numbers of positive and negative entries in $\sigma$, respectively. Define $|\sigma|$ to be the filling obtained by taking the absolute values of the entries in $\sigma$, and define
	\begin{align*}
		x^{|\sigma|}=\prod_{z\in \dg'(\lambda)}x_{|\sigma(z)|}
	\end{align*}
	as the monomial corresponding to the content of $|\sigma|$. Fix the total ordering $<$ of $\mathcal{A}\cup \{ 0,\infty\}$ to be
	\begin{align*}
		(\A\cup \{ 0,\infty\},<)=\{0<1< \overline{1}<2<\overline{2}<\cdots<\infty\}.
	\end{align*}
	For $a,b\in \mathcal{A}\cup \{ 0,\infty\}$, we write
	\begin{align*}
		I(a,b)=\begin{cases}
			1,  & \text{ if\ } a>b  \text{ or\ } a=b\in\mathbb{Z_{-}},\\
			0,  & \text{ if\ } a<b \text{ or\ } a=b\in\mathbb{Z_{+}} .
		\end{cases}
	\end{align*}
	For each box $u$ of $\dg'(\lambda)$, let $\South(u)$ denote the box directly below it. For any pair of entries
	$(\sigma(z),\sigma(\South(z)))$ in $\sigma$, it is a {\em \color{blue} descent} if $I(\sigma(z),\sigma(\South(z)))=1$; and {\em \color{blue} non-descent} otherwise. We define
	\begin{align*}
		\Des(\sigma)=\{z\in \dg'(\lambda): I(\sigma(z),\sigma(\South(z)))=1\}
	\end{align*}
	to be the {\em \color{blue} descent set} of $\sigma$, and $\des(\sigma)=|\Des(\sigma)|$. A non-descent pair $(\sigma(z),\sigma(\South(z)))$ is an {\em \color{blue} ascent} if $\sigma(z)<\sigma(\South(z))$.
	Let $\leg(u)$ be the number of boxes strictly above box $u$ in the diagram $\dg'(\lambda)$. Then, we call
	\begin{align*}
		\maj(\sigma)=\sum_{u\in\Des(\sigma)}({\mathrm{leg}}(u)+1)
	\end{align*}
	{\em \color{blue} the major index} of $\sigma$. Each Young diagram $\dg'(\lambda)$ is regarded as a concatenation of
	maximal rectangles in a way that the heights of rectangles are strictly decreasing from left to right; see for instance \cite{CHO22,JL24,JL25}. For a super filling $\sigma$, let $\sigma_j$ be the filling of the rectangle of $\dg'(\lambda)$ with height $j$. We write
	\begin{align}\label{E:sigmadec}
		\sigma=\sigma_{\lambda_1} \sqcup \cdots \sqcup \sigma_1,
	\end{align}
	where some $\sigma_j$ may be empty for $1\le j\le \lambda_1$. For example,  in Figure \ref{F:f2}, every filling $\sigma$ of the Young diagram $\dg'(\lambda)$ has the decomposition $\sigma=\sigma_{4}\sqcup \sigma_2\sqcup \sigma_1$.
	
	\begin{definition}[{\em \color{blue}Arm-lengths}]\label{Def:arm-lengths}
		Given a super filling $\sigma$,
		a box $u$ in $\dg'(\lambda)$ is called {restricted} if $|\sigma(u)|=|\sigma(\South(u))|$, and {unrestricted} otherwise.
		For a box $u=(i,j)\in \dg'(\lambda)$, define
		\begin{align*}
			\arm(u)&=|\{(i,k)\in \dg'(\lambda): k>j \}|,\\
			\widehat{\arm}(u)&=|\{(i-1,k)\in \dg'(\lambda): k>j \}| \\
			&\qquad+ |\{(i,k)\in\dg'(\lambda): k<j,\,\lambda_{k}=\lambda_j \textrm{ and } (i,k)\  \textrm{is unrestricted} \}|,\\
			\overline{\arm}(u)&=|\{(i-1,k)\in \dg'(\lambda): k>j\, \mbox{ and }\,\lambda_{k}<\lambda_j \}| \\
			&\qquad+ |\{(i,k)\in\dg'(\lambda): k<j,\, \lambda_{k}=\lambda_j \textrm{ and } (i,k)\  \textrm{is unrestricted} \}|,\\
			\widetilde{\arm}(u)&=|\{(i-1,k)\in \dg'(\lambda): k>j\, \mbox{ and }\,\lambda_{k}<\lambda_j \}|\\
			&\qquad + |\{(i,k)\in\dg'(\lambda): k<j\, \mbox{ and }\,\lambda_{k}\le \lambda_j  \}|.
		\end{align*}
	\end{definition}
	\begin{remark}\label{R:1}
		Consider the decomposition $\sigma=\sigma_{\lambda_1} \sqcup \cdots \sqcup \sigma_1$. Suppose that in the filling $\sigma_j$, row $i+1$ contains $m$ unrestricted boxes. Then, from left to right, the values $\overline{\arm}$ for these boxes are $s, s+1,\ldots,s+m-1$ where $s=\lambda_i'-\lambda_j'$. Consequently, the set $\{\overline{\arm}(u):\sigma(u)\ne \sigma(\South(u))\}$ is independent of the specific positions of unrestricted boxes.
	\end{remark}
	With the definitions of arm-length and leg-length, we are now ready to define the integral Macdonald polynomials $J_{\lambda}(X;q,t)$ and modified Macdonald polynomials $\widetilde{H}_{\lambda}(X;q,t)$.
	\begin{definition}[{\em \color{blue} Integral and modified Macdonald polynomials}]
		For a partition $\lambda$, let
		\begin{align}\label{E:bqt}
			b_{\lambda}(q,t)&=\prod_{s\in \dg(\lambda)}(1-q^{\arm(s)}t^{\leg(s)+1}).
		\end{align}
		Then the integral and modified Macdonald polynomials are
		\begin{align}\label{E:JP}
			J_{\lambda}(X;q,t)&=b_{\lambda}(q,t)P_{\lambda}(X;q,t),\\
			\widetilde{H}_{\lambda}(X;q,t)&=t^{n(\lambda)}J_{\lambda}\left[\frac{X}{1-t^{-1}};q,t^{-1}\right],\label{E:HtoJ}
		\end{align}
		respectively, where $f[X]$ is the plethystic substitution of $X$ into the symmetric function $f$.
	\end{definition}
	\begin{example} For $\lambda=(2,2)$,
		\begin{align*}
			P_{22}(X;q,t)&=m_{22}(X)+\frac{(1+q)(1-t)}{1-qt}m_{211}(X)\\
			&\qquad+\frac{(2+t+3q+q^2+3qt+2q^2t)(1-t)^2}{(1-qt)(1-qt^2)}
			m_{1111}(X),\\
			J_{22}(X;q,t)&=(1-qt)(1-t)(1-qt^2)(1-t^2)P_{22}(X;q,t)\\
			&=t^2 s_4[X(1-t)]+(qt^2+t+qt)s_{31}[X(1-t)]+(1+q^2t^2)s_{22}[X(1-t)]\\
			&\quad\quad+(q+q^2t+qt)s_{211}[X(1-t)]+q^2s_{1111}[X(1-t)],
		\end{align*}
		where $s_{\lambda}(X)$ is the Schur function. By (\ref{E:HtoJ}), the second equation yields
		\begin{align*}
			\widetilde{H}_{22}(X;q,t)&=s_4(X)+(q+t+qt)s_{31}(X)+(t^2+q^2)s_{22}(X)\\
			&\qquad+(qt^2+q^2t+qt)s_{211}(X)+q^2t^2s_{1111}(X).
		\end{align*}
	\end{example}
	We now define the combinatorial statistics queue inversion $(\quinv)$, inversion $(\inv)$ and the new statistic {quadruple coinversion} ($\quadcoinv$) on super fillings.
	\begin{definition}[{\em \color{blue}Inversion and queue inversion on super fillings}]\label{Def:inv}
		For a super filling $\sigma$ of $\dg'(\lambda)$, let $\hat{\sigma}$ be the augmented super filling obtained from $\sigma$ by adjoining a box with entry $0$ to the top of each column and a
		box with entry $\infty$ to the bottom of each column.
		
		For any triple $(a,b,c)$ of entries in $\hat{\sigma}$,
		define $\CMcal{Q}(a,b,c)=1$ if  exactly one of the following is true:
		$$\{I(a,b)=1,  I(c,b)=0,   I(a,c)=0  \}.$$
		Otherwise, $\CMcal{Q}(a,b,c)=0$.
		
		A {\color{blue} queue inversion triple} of $\sigma$ is a triple $(a,b,c)$ of entries in $\hat{\sigma}$ such that
		\begin{enumerate}
			\item $b$ and $c$ are in the same row, with $c$ to the right of $b$;
			\item $a$ and $b$ are in the same column, with $b$ immediately below $a$;
			\item $\CMcal{Q}(a,b,c)=1$, with $c\ne 0$.
		\end{enumerate}
		
		An {\color{blue} inversion triple} of $\sigma$ is a triple $(a,b,c)$ of entries in $\hat{\sigma}$ satisfying (2), (3) and
		\begin{enumerate}
			\setcounter{enumi}{3}
			\item $a$ and $c$ are in the same row, with $c$ to the right of $a$.
		\end{enumerate}
		These are illustrated in the diagram below. Let $\quinv(\sigma)$ and $\inv(\sigma)$ be the numbers of queue inversion triples and inversion triples of $\sigma$, respectively.
		\begin{figure}[H]
			\centering
			\begin{minipage}[H]{0.4\linewidth}
				\centering
				\begin{ytableau}
					a  &\none[]
					\\
					b &\none[\dots] & c\\
				\end{ytableau}
				\caption*{\small{queue inversion triple}}
			\end{minipage}
			\begin{minipage}[H]{0.4\linewidth}
				\centering
				\begin{ytableau}
					a  &\none[\dots] & c
					\\
					b
					&  \none[]\\
				\end{ytableau}
				\caption*{\small{inversion triple}}
			\end{minipage}
		\end{figure}
	\end{definition}

	
	
	
	\begin{definition}[{\em \color{blue} Quadruple coinversion on super fillings}]\label{Def2.2}
		For a super filling $\sigma$, consider the quadruple $(z,w,u,v)$ of entries in the augmented super filling $\hat{\sigma}$ (Definition \ref{Def:inv}), arranged as in Figure \ref{F:f1}, with the columns containing $(z,u)$ and $(w,v)$ having equal height in $\sigma$. {
			\begin{enumerate}
				\item If  $|z|=|u|$, $|w|=|v|$ and $|z|\ne|w|$, then $(z,w,u,v)$ (resp.$(z,w,v,u)$) is a quadruple coinversion if and only if $\min\{z,w\}\in\mathbb{Z_{+}}$ (resp. $\min\{z,w\}\in\mathbb{Z_{-}}$).
				\item If  $|w|=|v|$ and $|z|,|u|,|w|$ are distinct, then $(z,w,u,v)$(resp. $(z,w,v,u)$) is a quadruple coinversion if and only if $w\in \mathbb{Z_{+}}$ (resp.$w\in \mathbb{Z_{-}}$).
		\end{enumerate} }
		Otherwise, either the column containing $(z,u)$ is longer than the column containing $(w,v)$, or neither condition (1) nor condition (2) is satisfied. In this case, $(z,w,u,v)$ is a quadruple coinversion if and only if $(z,u,v)$ is a queue inversion triple; that is, if and only if exactly one of the following is true:
		$$\{I(z,u)=1,  I(v,u)=0,   I(z,v)=0  \}.$$
		We write $\Q(z,w,u,v)=1$ if $(z,w,u,v)$ is a quadruple coinversion; otherwise $\Q(z,w,u,v)=0$.
		Let $\quadcoinv(\sigma)$ count the number of quadruple coinversions of $\sigma$.
	\end{definition}
	\begin{figure}[H]
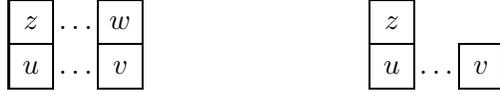

		\centering
		\begin{minipage}[H]{0.3\linewidth}
			\centering
			\begin{ytableau}
				z  &\none[\dots] & w
				\\
				u &\none[\dots] & v\\
			\end{ytableau}
		\end{minipage}
		\begin{minipage}[H]{0.3\linewidth}
			\centering
			\begin{ytableau}
				z  &\none[]
				\\
				u  &\none[\dots] & v
			\end{ytableau}
		\end{minipage}
		\caption{A quadruple $(z,w,u,v)$ and a queue inversion triple $(z,u,v)$} \label{F:f1}
	\end{figure}
	{Note that for any entries $a,b,c$ with $|a|=|b|\ne|c|$, we have
		\begin{align}\label{eq:a=b}
			\Q(a,b,c)=1  \,  \text{if and only if} \, a\in\Z_+.
	\end{align} }
	We next define three types of non-attacking tableaux, top-row (resp. bottom-row) increasing or decreasing tableaux, and reverse tableaux.
	
	\begin{definition}[{\em \color{blue} Non-attacking, $\quinv$- and $\inv$-non-attacking fillings}]\label{Def:non-attacking}
		Given the Young diagram $\dg'(\lambda)$, two boxes $u,v\in \dg'(\lambda)$ are said to be {\em $\quinv$-attacking} each other if either
		\begin{enumerate}
			\item[$\mathrm{(I)}$] $u=(i,j)$, $v=(i,k)$, where $j\ne k$; or
			\item[$\mathrm{(II)}$] $u=(i+1,j)$, $v=(i,k)$, where $k>j$.
		\end{enumerate}
		Two boxes $u,v\in \dg'(\lambda)$ are said to be  {\em $\inv$-attacking} each other if either $\mathrm{(I)}$ or the condition obtained from $\mathrm{(II)}$ by substituting $k>j$ by $k<j$ is satisfied.
		A $\quinv$(resp. $\inv$)-non-attacking filling $\sigma$ is a filling without two $\quinv$(resp. $\inv$)-attacking boxes $u,v$ such that $|\sigma(u)|=|\sigma(v)|$.
		
		Two boxes $u,v\in \dg'(\lambda)$ are said to be {attacking} each other if either $\mathrm{(I)}$ or $\mathrm{(II)}$ or $\mathrm{(III)}$ is satisfied.
		\begin{enumerate}
			\item[$\mathrm{(III)}$] $u=(i,j)$, $v=(i-1,k)$ where $k<j$ such that $\lambda_j=\lambda_{k}$.
		\end{enumerate}
		A non-attacking filling $\sigma$ is a filling without two attacking boxes $u,v$ satisfying $|\sigma(u)|=|\sigma(v)|$. The relative positions of boxes satisfying conditions $\mathrm{(I)}$, $\mathrm{(II)}$ and $\mathrm{(III)}$ are illustrated in the diagram below.
		\begin{figure}[H]
			\centering
			\begin{minipage}[H]{0.3\linewidth}
				\centering
				\begin{ytableau}
					u &\none[\dots] & v\\
					\none[] & \none[]\\
				\end{ytableau}
				\caption*{\small{Case (I)}}
			\end{minipage}
			\begin{minipage}[H]{0.3\linewidth}
				\centering
				\begin{ytableau}
					u	 &\none[\dots]
					\\
					\none[]
					&  \none[\dots] & v\\
				\end{ytableau}
				\caption*{\small{Case (II)}}
			\end{minipage}
			\begin{minipage}[H]{0.3\linewidth}
				\centering
				\begin{ytableau}
					\none[]  &\none[\dots] & u
					\\
					v
					&  \none[\dots]\\
				\end{ytableau}
				\caption*{\small{Case (III)}}
			\end{minipage}
		\end{figure}
	\end{definition}
	\begin{definition}[{\em \color{blue} Top-row and bottom-row increasing/decreasing tableaux}]\label{Def:quad_sorted}
		A positive filling $\sigma=\sigma_{\lambda_1}\sqcup\cdots\sqcup \sigma_{1}$ of the Young diagram $\dg'(\lambda)$ is called top-row (resp. bottom-row) increasing or decreasing if, for each $1\le j\le \lambda_1$, the entries of the top row (resp. bottom-row) in $\sigma_j$ are weakly increasing or weakly decreasing from left to right.
	\end{definition}
	\begin{definition}[{\em \color{blue}Reverse tableaux}]\cite[Definition 3.1]{JL24}
		Given a filling $\sigma=\sigma_{\lambda_1} \sqcup \cdots \sqcup \sigma_1$, let $\rev(\sigma)=\rev(\sigma_{\lambda_1})\sqcup\cdots\sqcup \rev(\sigma_1)$ be the reverse of $\sigma$, where $\rev(\sigma_i)$ is obtained from $\sigma_i$ by reversing the sequence of entries in each row. 
	\end{definition}
	Finally, we present superization formulas in terms of inversion and queue inversion statistics. As demonstrated by Haglund, Haiman and Loehr \cite{HHL04}, superization is the first step in deducing a combinatorial formula of $J_{\lambda}(X;q,t)$ from $\widetilde{H}_{\lambda}(X;q,t)$. Using the same approach, Ayyer, Mandelshtam and Martin \cite{AMM23} proved a queue-inversion analogue.
	
	\begin{proposition} \cite[Equations (40) and (63)]{HHL04}\cite[Equation (12) and Theorem 5.3]{AMM23}\label{propquinvJ} 
		For a partition $\lambda$ of $n$, let $\M(\lambda')$ be the set of super fillings of the Young diagram $\dg'(\lambda)$. Then for $\vartheta\in \{\inv,\quinv\}$ we have
		\begin{align}\label{E:supJ}
			J_{\lambda}(X;q,t)
			&=t^{n(\lambda)+n} \sum_{\sigma\in \M(\lambda') } (-1)^{m(\sigma)} t^{-p(\sigma)-\vartheta(\sigma)}q^{\maj(\sigma)} x^{|\sigma|},\\
			&=t^{n(\lambda)+n} \sum_{\sigma\in \M(\lambda')\atop |\sigma|\ \vartheta\textrm{-non-attacking} } (-1)^{m(\sigma)} t^{-p(\sigma)-\vartheta(\sigma)}q^{\maj(\sigma)} x^{|\sigma|}.\label{E:supJ2}
		\end{align}
	\end{proposition}
	For a complete description of superization, we refer the reader to \cite[Section 8]{HHL04}.
	
		
		
		\section{Flip operators}\label{S:3}
		In this section, we define four flip operators that act on pairs of adjacent columns of equal height within a filling. 
		Each flip operator is paired with one of the following statistics: inversion ($\inv$), queue inversion ($\quinv$),  $\eta$ from (\ref{eqthm1}), or quadruple coinversion ($\quadcoinv$). Essentially, each operator is an involution on fillings that alters its paired statistic by at most $\pm 1$ while preserving the major index.
		
		Here we define a new flip operator for the $\quadcoinv$ statistic on super fillings. The other three flip operators were originally introduced in \cite{AMM23,JL25,LN12}.
		
		
		\begin{definition}[{\em \color{blue} Flip operators}]\label{Def3.1}
			For a  super filling $\sigma$ of $\dg'(\lambda)$ and integers $r,i$ satisfying $r\le \lambda_i=\lambda_{i+1}$, let $t_i^{r}$ denote the operator that interchanges the entries $\sigma(r,i)$ and $\sigma(r,i+1)$. For any interval $[r,s]$ with $1\le r\le s\le \lambda_i$, define the operator
			\begin{align*}
				t_i^{[r,s]}:=t_i^{r}\circ t_i^{r+1}\cdots \circ t_i^{s},
			\end{align*}
			which swaps the entries of boxes $(u,i)$ and $(u,i+1)$ for all rows $u$ from $r$ to $s$.
			\begin{enumerate}
				\item The {\color{blue} queue-inversion flip operator} $\rho_i^r$ is defined as $\rho_i^r=t_i^{[k,r]}$, where
				$k$ is the largest integer satisfying $1\le k\le r$ and
				\begin{align*}
					&\quad\,\CMcal{Q}(\sigma(k,i),\sigma(k-1,i),\sigma(k-1,i+1))\notag\\
					&=\CMcal{Q}(\sigma(k,i+1),\sigma(k-1,i),\sigma(k-1,i+1)).
				\end{align*}
				\item The {\color{blue} inversion flip operator} $\zeta_i^r$ is defined as $\zeta_i^r=t_i^{[r,k]}$, where $k$ is the smallest integer such that $k\ge r$ and
				\begin{align*}
					&\quad\,\CMcal{Q}(\sigma(k+1,i),\sigma(k,i),\sigma(k+1,i+1))\notag\\
					&=\CMcal{Q}(\sigma(k+1,i),\sigma(k,i+1),\sigma(k+1,i+1)).
				\end{align*}
				\item The {\color{blue}flip operator} $\delta_i^r$ is defined as $\delta_i^r=t_i^{[k,r]}$, where $k$ is the largest integer such that $1\le k\le r$ and $\sigma(k,i)=\sigma(k,i+1)$. If no such $k$ exists, we set $k=1$.
				\item The {\color{blue}quadruple-coinversion flip operator} $\xi_i^r$ is defined as $\xi_i^r=t_i^{[k,r]}$, where $k$ is the largest integer satisfying $1\le k\le r$ and
				\begin{align*}
					&\quad\,\CMcal{Q}(\sigma(k,i),\sigma(k,i+1),\sigma(k-1,i),\sigma(k-1,i+1)) \notag\\
					&=\CMcal{Q}(\sigma(k,i+1),\sigma(k,i),\sigma(k-1,i),\sigma(k-1,i+1)).
				\end{align*}
			\end{enumerate}
			For simplicity, we write
			\begin{align*}
				\rho_i=\rho_i^{\lambda_i},\quad \zeta_i=\zeta_i^{1} ,\quad \xi_i=\xi_i^{\lambda_i} \quad\mbox{ and }\quad \delta_i=\delta_i^{\lambda_i}.
			\end{align*}
			By construction,  each of the operators $\rho_i^r$, $\zeta_i^r$, $\xi_i^r$ and $\delta_i^r$ is an involution.
		\end{definition}
		
		
		
		The essential properties of the queue-inversion flip operators extend identically from positive fillings to super fillings.
		
		\begin{lemma}\cite[Lemma 5.3]{JL24}\label{lem3.1}
			For a partition $\lambda$ and an integer $i$ such that $\lambda_i=\lambda_{i+1}$, let $\sigma$ be a super filling of the Young diagram $\dg'(\lambda)$.
			
			Let $\raisebox{-8pt}{\young(ab,cd)}$ be part of $\sigma$ where $a=\sigma(r+1,i)$, $b=\sigma(r+1,i+1)$, $c=\sigma(r,i)$ and $d=\sigma(r,i+1)$. If $\CMcal{Q}(a,c,d)=\CMcal{Q}(b,c,d)$, then
			\begin{align}\label{E:maj11}
				\maj(\sigma)&=\maj(\rho_i^{r}(\sigma)),\\
				\label{E:quinv1}\quinv(\sigma)&=\quinv(\rho_i^{r}(\sigma))-\CMcal{Q}(a,d,c)+\CMcal{Q}(a,c,d).
			\end{align}
			That is, the number of queue inversion triples induced between column $i$ or $i+1$ and column $j$ of $\sigma$ for any $j>i+1$ is invariant under $\rho_i^r$.
		\end{lemma}
		The following lemma describes the property of the quadruple-coinversion flip operators on non-attacking super fillings; notably, these operators preserve the non-attacking condition. Throughout this paper, we use the notation $\chi(A)$ for the indicator function of a statement $A$: $\chi(A)=1$ if $A$ is true, and $\chi(A)=0$ otherwise.

		\begin{lemma}\label{lem3.2}
			For a partition $\lambda$ and an integer $i$ such that $\lambda_i=\lambda_{i+1}$, let $\sigma$ be a non-attacking super filling of the Young diagram $\dg'(\lambda)$.
			
			Let $\raisebox{-8pt}{\young(ab,cd)}$ be part of $\sigma$ where $a=\sigma(r+1,i)$, $b=\sigma(r+1,i+1)$, $c=\sigma(r,i)$ and $d=\sigma(r,i+1)$. If $\Q(a,b,c,d)=\Q(b,a,c,d)$, then
			\begin{align}
				\label{E:maj1}
				\maj(\sigma)&=\maj(\xi_i^{r}(\sigma)).\\
				\label{E:tq1}
				\quadcoinv(\sigma)&=\quadcoinv(\xi_i^{r}(\sigma))-\Q(a,d,c)+\Q(a,c,d).
			\end{align}
			That is, the number of quadruple coinversions located between column $i$ or $i+1$ and column $j$ of $\sigma$ for any $j\not \in \{i,i+1\}$ is invariant under $\xi_i^r$. Moreover, the number of (un)restricted boxes in each row is invariant, and $\xi_i^{r}(\sigma)$ remains a non-attacking filling.
		\end{lemma}
		\begin{proof}
			Let $\raisebox{-8pt}{\young(st,uv)}$ be part of $\sigma$ such that $\xi_i^r$ terminates at row $\raisebox{-2pt}{\young(st)}$, say row $\kappa$. We claim that the absolute values $|s|$, $|t|$, $|u|$ and $|v|$ are all distinct, unless $s$ and $t$ lie in the bottom row of $\sigma$. Because $\sigma$ is non-attacking, we have $|s|\ne|t|$, $|s|\ne|v|$, $|t|\ne|u|$, and $|u|\ne|v|$. The equality $|u|=|v|$ occurs only when $u=v=\infty$, which corresponds precisely to the case where $s$ and $t$ are in the bottom row.
			
			{ If $|t|=|v|$,
				then by Definition \ref{Def2.2} (1) and (2),  $\Q(s,t,u,v)\ne\Q(t,s,u,v)$ whether $|s|=|u|$ or not.
				If $|s|=|u|, |t|\ne|v|$. Then  $\Q(s,t,u,v)=\Q(s,u,v)$, and by \eqref{eq:a=b},   $(s,u,v)$ is a queue inversion triple if and only if $s\in\Z_+$. However, by Definition \ref{Def2.2} (2). $\Q(t,s,u,v)=1$ if and only if $s\in\Z_-$. Hence  $\Q(s,t,u,v)\ne \Q(t,s,u,v)$. }
			
			Since all cases contradict the terminating condition $\Q(s,t,u,v)=\Q(t,s,u,v)$, we conclude that $|s|$, $|t|$, $|u|$, $|v|$ must be distinct. By a similar argument, the absolute values $|a|$, $|b|$, $|c|$, $|d|$ are also distinct (except when $c$ and $d$ are in the top row). Given Definition \ref{Def2.2}, the condition $\Q(a,b,c,d)=\Q(b,a,c,d)$ for \eqref{E:tq1} together with the terminating condition $\Q(s,t,u,v)=\Q(t,s,u,v)$ implies
			\begin{align}
				\Q(a,c,d)=\Q(b,c,d)=\Q(a,b,c,d)=\Q(b,a,c,d), \label{eqpflem3.2a}\\
				\Q(s,u,v)=\Q(t,u,v)=\Q(s,t,u,v)=\Q(t,s,u,v).\label{eqpflem3.2b}
			\end{align}
			Consequently, the filling $\xi_{i}^r(\sigma)$ remains non-attacking, and the operator $\xi_i^r$ preserves the number of (un)restricted boxes in each row. Moreover, the first equalities in \eqref{eqpflem3.2a}--\eqref{eqpflem3.2b} and \eqref{E:maj11} ensure that $\maj(\sigma)=\maj(\xi_i^r(\sigma))$.
			
			It remains to prove (\ref{E:tq1}). Given the distinct absolute values $|a|$, $|b|$, $|c|$ and $|d|$, we have
			$\Q(a,b,d,c)\ne \Q(a,b,c,d)$. Thus, swapping $c$ and $d$ creates a new quadruple coinversion precisely when $\Q(a,b,d,c)=1$. Therefore, it suffices to show that the total number of quadruple coinversions located in all other rows and columns remains unchanged under $\xi_i^r$.
			
			First, the terminating condition $\Q(s,t,u,v)=\Q(t,s,u,v)$ guarantees that swapping $s$ and $t$ preserves the count of quadruple coinversions between rows $\kappa$ and $\kappa-1$. Second, we assert that the operator $\xi_i^r$ preserves the number of quadruple coinversions between columns $i$ and $i+1$, as well as between rows $\kappa$ and $r$.
			
			Let $\pi=\raisebox{-8pt}{\young(\alpha\beta,\gamma\delta)}$ be a square situated between rows $\kappa$ and $r$ and between columns $i$ and $i+1$. Since $\xi_i^r$ reverses both rows of $\pi$, we have
			$\Q(\alpha,\beta,\gamma,\delta)\ne \CMcal{Q}(\beta,\alpha,\gamma,\delta)$. Because the row containing $\gamma$ and $\delta$ is not the terminating row, it follows that $\Q(\beta,\alpha,\gamma,\delta)\ne \Q(\beta,\alpha,\delta,\gamma)$. Consequently,  $\CMcal{Q}(\alpha,\beta,\gamma,\delta)=\CMcal{Q}(\beta,\alpha,\delta,\gamma)$, which proves the assertion.
			
			Finally, we analyze the effect of $\xi_i^r$ on the number of quadruple coinversions arising from column $i$ or $i+1$ and any other column $j\not \in \{i,i+1\}$. It is sufficient to verify the invariance under $t_i^{(k)}$ for $k\in \{\kappa,r\}$. For the case $k=\kappa$, consider the square of entries $s$, $t$, $u$ and $v$. Let the columns to the left and right of this square be $(x,y)$ and $(e,f)$, respectively, as illustrated below.
			
			\begin{center}
				\begin{minipage}[H]{0.4\linewidth}
					\begin{ytableau}
						x& \none[\dots]&		s  & t & \none[\dots]& e
						\\
						y& \none[\dots]&		u & v & \none[\dots] & f \\
					\end{ytableau}
					\qquad\, $\longrightarrow$
				\end{minipage}
				\begin{minipage}[H]{0.2\linewidth}
					\begin{ytableau}
						x& \none[\dots]&	t  & s & \none[\dots]& e
						\\
						y& \none[\dots]&	u & v & \none[\dots] & f \\
					\end{ytableau}
				\end{minipage}
			\end{center}
			If $|e|=|f|$ and columns $(t,v)$ and $(e,f)$ have equal height, then by Definition \ref{Def2.2} (2) and the distinctness of $|s|$, $|t|$, $|u|$ and $|v|$, we obtain
			\begin{align*}
				\Q(s,e,u,f)=\Q(t,e,v,f)=\Q(t,e,u,f)=\Q(s,e,v,f)=\chi(e\in \mathbb{Z}_{+}).
			\end{align*}
			Otherwise, $|e|\ne|f|$ or column $(t,v)$ is longer than column $(e,f)$, then
			\begin{align*}
				\Q(s,e,u,f)&=\Q(s,u,f), \quad \Q(t,e,v,f)=\Q(t,v,f),\\
				\Q(t,e,u,f)&=\Q(t,u,f),\quad \Q(s,e,v,f)=\Q(s,v,f).
			\end{align*}
			Applying Lemma \ref{lem3.1} and \eqref{eqpflem3.2b} yields
			$\Q(s,u,f)+\Q(t,v,f)=\Q(t,u,f)+\Q(s,v,f)$, which implies $\Q(s,e,u,f)+\Q(t,e,v,f)=\Q(t,e,u,f)+\Q(s,e,v,f)$. A similar argument gives
			\begin{align*}
				\Q(x,s,y,u)+\Q(x,t,y,v)=\Q(x,t,y,u)+\Q(x,s,y,v).
			\end{align*}
			The remaining case $k=r$ follows by an identical argument, the details of which are omitted.
			This completes the proof.
		\end{proof}
		
		\section{Reduced expressions of permutations}\label{S:reduced}
		Let $S_n$ denote the symmetric group of permutations of $[n]=\{1,\ldots,n\}$.
		Every permutation $w\in S_n$ can be expressed as a product of simple transpositions $s_i=(i\,i+1)$ for $1\le i<n$. An expression $\w=s_{i_k}\cdots s_{i_1}$ is called {\em \color{blue} reduced} if $k$ is the minimal number of simple transpositions required to represent $w$. The length $\ell(w)$ of a permutation $w$ is the number of factors in any of its reduced expressions; it satisfies $\ell(w)=\inv(w)$ where
		\begin{align*}
			\inv(w)=|\{(i,j): 1\le i<j\le n \,\textrm{ and }\, w(i)>w(j)\}|
		\end{align*}
		is the number of inversions of $w$.
		
		Given a reduced expression $\w=s_{i_k}\cdots s_{i_1}$, a flip operator $\gamma\in \{\rho,\zeta,\delta,\xi\}$ from Section \ref{S:3} and an integer $r$, we define the flip operator for the reduced expression $\w$ by
		\begin{align}\label{E:wflip}
			\gamma_{{\bf w}}^r&=\gamma_{i_k}^r \,\cdots\,\gamma_{i_1}^r,
		\end{align}
		and we simply write $\gamma_{{\bf w}}=\gamma_{i_k}\,\cdots\,\gamma_{i_1}$.
		
		Among the flip operators from Section \ref{S:3}, only $\delta_i^r$ satisfies the braid relation (see \cite[Lemma 4.5]{JL25}); that is,
		\begin{align*}
			\delta_i^r\,\delta_{i+1}^r\,\delta_i^r(\sigma)=\delta_{i+1}^r\,\delta_{i}^r\, \delta_{i+1}^r(\sigma)
		\end{align*}
		holds for any filling $\sigma$ whose columns $i$, $i+1$ and $i+2$ have equal height. The other three flip operators $\gamma\in \{\rho,\zeta,\xi\}$ do not satisfy this relation. Consequently, for a permutation $w$, two different reduced expressions $\w$ and ${\bf v}$ may yield distinct flip operators, i.e., $\gamma_{\w}^r\ne \gamma_{{\bf v}}^r$.
		
		To obtain a well-defined pairing between a flip operator and a permutation, we must therefore select a specific reduced expression for each $w$, called a {\em positive distinguished subexpression} (PDS). This well-definedness is required for the proofs of (\ref{eqthm2.2}) and (\ref{eqlem2.1}). The technique of selecting the PDS for counting purpose appears in \cite{CHO22,O:23,O:24}.
		
		\begin{definition}[{\em \color{blue} Positive distinguished subexpression (PDS)}] \cite{MR:04}
			The (strong) Bruhat order $(S_n,\le)$ is defined by its covering relation: $u$ covers $w$ if and only if $u=w(i\,j)$ and $\ell(u)=\ell(w)+1$ for some transposition $(i\,j)$.
			
			Now, let $v\le w$ in $S_n$, and fix a reduced expression $\w=s_{i_k} \cdots s_{i_1}$ for $w$. The positive distinguished subexpression (PDS) for $v$ in $\mathbf{w}$ is a reduced expression for $v$ obtained by removing some simple transpositions from $\mathbf{w}$. It is constructed inductively: set $v^0=v$, and for $1\le j\le k$, define
			\begin{align*}
				v^j=\begin{cases}
					v^{j-1} s_{i_j} & \text{ if\ } v^{j-1} s_{i_j}< v^{j-1}; \\
					v^{j-1}     & \text{ otherwise.}
				\end{cases}
			\end{align*}
			Let $v_j=s_{i_j}$ if $v^{j-1} s_{i_j}< v^{j-1}$ and $v_j=\mathbf{1}$ (the identity) otherwise. Then the sequence $\mathbf{v}=v_{k}\cdots v_{1}$ is the PDS for $v$ in $\w$.	
		\end{definition}
		\begin{example}
			For $n=3$, consider the permutations $w=321$ and $v=231$ in one-line notation. Fix the reduced expression $\w=s_1s_2s_1$ for $w$. Following the inductive construction, we compute the PDS for $v$ in $\w$:
			\begin{itemize}
				\item $v^0=v=231$.
				\item $v^1=v^0=231$, since $v^{0}s_{1}=321>231=v^{0}$.
				\item $v^2=v^1s_2=213$, since $v^1s_2=213<v^1$.
				\item $v^3=v^2s_1=123$, since $v^2s_1=123<v^2$.
			\end{itemize}
			Thus, the corresponding factors are $v_1={\bf 1}$, $v_2=s_2$ and $v_3=s_1$, yielding the PDS for $v$ in $\w$ is ${\bf v}=s_1s_2$.
		\end{example}
		\begin{definition}[{\em \color{blue} Flip operator paired with a permutation}]
			Let $\w^*=(s_1)(s_2s_1)\cdots (s_{n-1}\cdots s_1)$ be a reduced expression of $w_0=n\cdots 21$ in one-line notation. For any permutation $w\in S_n$, let ${\bf w}=s_{i_k}\cdots s_{i_1}$ be the positive distinguished subexpression (PDS) for $w$ in $\w^*$, and set $\textrm{PDS}(n)=\{{\bf w}: w\in S_n\}$. Then, for each flip operator $\gamma\in \{\rho,\zeta,\xi\}$ defined in Section \ref{S:3} and for an integer $r$, we define the flip operator paired with the permutation $w$ by
			\begin{align}\label{E:wflip2}
				\gamma_{w}^r=\gamma_{{\bf w}}^r=\gamma_{i_k}^r \circ\cdots\circ\gamma_{i_1}^r.
			\end{align}
			Moreover, $\delta_w^r=\delta_{\w}^r$ for any reduced expression $\w$ of $w$.
		\end{definition}
		
		\section{Roadmap for the proof of Theorem \ref{T:main1}}\label{S:roadmap}
		The proof of Theorem \ref{T:main1} is divided into three steps. In the first step, we establish a relationship between the two pairs $(\quadcoinv,\maj)$ and $(\quinv,\maj)$ on super fillings via the following theorem.
		\begin{theorem}\label{thm11}
			For a partition $\lambda$, let $\M(\lambda)$ be the set of super fillings of the Young diagram $\dg(\lambda)$ of $\lambda$. Then,
			\begin{align}\label{eq1.1}
				&\quad\,\sum_{\sigma\in\M(\lambda) \atop |\sigma| \, \textrm{non-attacking} } (-1)^{m(\sigma)} q^{\maj(\sigma)} t^{p(\sigma)+\quadcoinv(\sigma)}x^{|\sigma|}\notag\\
				&=\sum_{\sigma\in\M(\lambda) \atop |\sigma| \,\,  \quinv\textrm{-non-attacking} } (-1)^{m(\sigma)} q^{\maj(\sigma)} t^{p(\sigma)+\quinv(\sigma)}x^{|\sigma|}.
			\end{align}
		\end{theorem}
		In the second step, we derive combinatorial formulas for the integral and symmetric Macdonald polynomials, $J_{\lambda}(X;q,t)$ and $P_{\lambda}(X;q,t)$ (Theorem \ref{thm1}), by applying Theorem \ref{thm11} and the superization formula (\ref{E:supJ2}) where $\vartheta=\quinv$.
		
		Given a partition $\lambda$, let $m_i$ denote the multiplicity of part $i$ in the partition $\lambda$ for $i\ge 1$. Let $[m]_t=1+\cdots+t^{m-1}$ for $m\ge 1$ and $[0]_t=1$. Define $[m]_t!=[1]_t\cdots [m]_t$ for $m\ge 1$ and $[0]_t!=1$. 
		\begin{theorem}\label{thm1}
			For a partition $\lambda$, let $\T(\lambda)$ be the set of positive fillings of the Young diagram $\dg(\lambda)$ of $\lambda$.
			Then,
			\begin{align}
				J_{\lambda}(X;q,t)
				&=\sum_{\tau\in \T(\lambda') \atop \tau \textrm{ non-attacking} } q^{\maj(\tau)} t^{\overline{\quadcoinv}(\tau)} x^{\tau}\prod_{\tau(u)\ne \tau(\South(u)) \atop \text{ or }u\,\in \,\text{row }1}(1-t)\notag\\
				&\qquad \qquad\qquad\times\prod_{\tau(u)=\tau(\South(u))\atop \text{ and }u\,\not\in \,\text{row }1} (1-q^{\leg(u)+1} t^{\widehat{\arm}(u)+1} ).\label{eqthm1.2}\\
				P_{\lambda}(X;q,t)&=\prod_{i\ge 1}([m_i]_t!)^{-1}\sum_{\sigma\in\T(\lambda') \atop \sigma\, \text{ non-attacking} }
				q^{\maj(\sigma)}t^{\overline{\quadcoinv}(\sigma)}c_{\sigma}(q,t)\,x^{\sigma}, \label{eqthm2.1}\\
				&=\sum_{\sigma}
				q^{\maj(\sigma)}t^{\overline{\quadcoinv}(\sigma)}c_{\sigma}(q,t)\,x^{\sigma},\label{eqthm2.2} 
			\end{align}
			which is summed over non-attacking and top-row increasing tableaux $\sigma$ of the Young diagram $\dg'(\lambda)$ of $\lambda$.
		\end{theorem}
		Equation \eqref{eqthm2.2} was first found by Corteel, Mandelshtam and Williams \cite[Theorem 5.9 and Remark 5.17]{CMW22}. Their proof utilizes a weighted generating function for multiline queues, the $q$-deformed Knizhnik--Zamolodchkov (qKZ) family introduced by Kasatani and Takeyama \cite{KT07} and its relation to Macdonald polynomials developed by Macdonald \cite{Mac94} (see also Cantini, de Gier and Wheeler \cite[Lemma 3]{CdGW:15}).
		
		In the final step, we demonstrate that for any non-attacking, top-row increasing filling $\sigma$ and for each $\eta\in \A^+$,
		the statistic pairs $(\eta^{\circ},\maj)$ and $(\overline{\quadcoinv},\maj)$ are equidistributed over the row-equivalence class $[\sigma]$. This completes the proof of Theorem \ref{T:main1}.
		\begin{theorem}\label{thm3}
			{Let $N^+(\lambda')$ be the set of positive fillings of $\dg'(\lambda)$ that are non-attacking and top-row increasing. Then for every $\eta\in\A^+$, there is a bijection $\varphi:N^+(\lambda')\rightarrow N^+(\lambda')$ such that  $\varphi(\tau)\sim \tau$,}
			\begin{align}
				\label{eqthm3.1}(\maj,\eta^{\circ})(\varphi(\tau))&=(\maj,\overline{\quadcoinv})(\tau),
			\end{align}
			and $c_{\tau}(q,t)=c_{\varphi(\tau)}(q,t)$ where $c_{\tau}(q,t)$ is defined in (\ref{E:cqt}). 
		\end{theorem}
		Following the same line of argument, we can also prove an inversion analogue of Theorem \ref{T:main1} using the superization formulas by Haglund, Haiman and Loehr (see Proposition \ref{propquinvJ}). However, in this analogue (Theorem \ref{T:Ptin}), the polynomial factor $b_{\lambda}(q,t)$ defined in (\ref{E:bqt}) cannot be cancelled out.
		
		\section{Proof of Theorem \ref{thm11}}\label{S:5}
		This section is devoted to proving Theorem \ref{thm11}, which constitutes the first step in the proof of Theorem \ref{T:main1}. The strategy is to partition the set of all $\quinv$-non-attacking super fillings into two disjoint subsets. We then show that the first subset is in bijection with the set of non-attacking super fillings, while the second subset makes no contribution to the right-hand side of (\ref{eq1.1}).
		
		\begin{definition}[{\em \color{blue} Separable $\quinv$-non-attacking super fillings}]\label{Def:separable}
			Let $\lambda'=(n,n)$ and let $\sigma$ be a $\quinv$-non-attacking super filling of $\dg(\lambda')$. If $\sigma$ is also non-attacking and all restricted boxes lie to the left of all unrestricted boxes, then $\sigma$ is defined to be separable.
			
			Otherwise, there exists a pair of attacking boxes $u$ and $v$ satisfying condition (III) of Definition \ref{Def:non-attacking} with $|\sigma(u)|=|\sigma(v)|$, or there exists a restricted box $u$ to the right of an unrestricted box $v$ in the top row.
			For such a super filling $\sigma$, we construct a new super filling $\Omega(\sigma)$ by applying a sequence of queue-inversion flip operators $\rho_i$.
			
			The super filling $\Omega(\sigma)$ is defined as follows: We start with the first (leftmost) box in the top row of $\sigma$ that is an attacking box or a restricted box with some unrestricted boxes to its left,
			and move its entry to the first column by a sequence of queue-inversion operators $\rho_i$. We repeat this process for the second such box in the top row, moving its entry to the second column, and continue until all entries from these boxes in the top row have been moved to the front columns.
			
			If the resulting filling $\Omega(\sigma)$ is non-attacking, we say $\sigma$ is {\color{blue} separable}; otherwise, $\sigma$ is {\color{blue} inseparable}. Equivalently, $\sigma$ is separable if and only if in $\Omega(\sigma)$, any two entries with the same absolute value lie within a single column.
			
			Let $\sigma\vert_{i}^{j}$ denote the segment of a super filling $\sigma$ from row $i$ through row $j$. For an arbitrary partition $\lambda$, consider a super filling $\sigma=\sigma_{\lambda_1} \sqcup \cdots \sqcup \sigma_1$ of the Young diagram $\dg'(\lambda)$. We say $\sigma$ is {\color{blue} separable} if, for every $\sigma_k$ and every pair of consecutive rows $j$, $j+1$ of $\sigma_k$, the two-row super filling $(\sigma_k)\vert_{j}^{j+1}$ is separable. Otherwise, $\sigma$ is {\color{blue} inseparable}.
			
		\end{definition}
		
		\begin{example}\label{Eg:separable}
			Consider the $\quinv$-non-attacking super filling $\sigma$ shown below. This super filling is separable because $\Omega(\sigma)=\rho_2\, \rho_3\, \rho_1\, \rho_2(\sigma)$ is non-attacking and all restricted boxes stand before unrestricted ones. Specifically,\\
			\begin{align*}
				\sigma&=\begin{ytableau}[] 7&5& *(green)\overline{4} & *(green)3 \\ *(green)4& *(green)3& 6& 1
				\end{ytableau} \xrightarrow{\rho_2}
				\begin{ytableau}[] 7&*(green)\overline{4} &5& *(green)3 \\ *(green) 4& *(green)3& 6& 1
				\end{ytableau}\xrightarrow{\rho_1}
				\begin{ytableau}[] *(green)\overline{4}&7& 5 & *(green)3 \\ *(green)4& *(green)3 & 6& 1
				\end{ytableau} \\[5pt]
				&\xrightarrow{\rho_3}
				\begin{ytableau}[] *(green)\overline{4} &7& *(green)3& 5 \\  *(green)4&*(green) 3& 6& 1
				\end{ytableau}
				\xrightarrow{\rho_2}
				\begin{ytableau}[] *(green)\overline{4} & *(green)3& 7& 5 \\  *(green)4& *(green)3& 6&1
				\end{ytableau}=\Omega(\sigma).
			\end{align*}
			Furthermore, the filling $\tau$ below is inseparable: with $\Omega=\rho_1$, the resulting filling $\Omega(\tau)$ is attacking. All pairs of attacking boxes are highlighted in green.\\
			\vspace{-2mm}
			\begin{align*}
				\tau=\begin{ytableau}[] 2& *(green)\overline{3}  \\  *(green)3& 5\end{ytableau} \xrightarrow{\rho_1} \begin{ytableau}[] *(green)\overline{3} &2 \\  5& *(green)3\end{ytableau}=\Omega(\tau).
			\end{align*}	
		\end{example}
		Let $G(\lambda')$ and $B(\lambda')$ denote the sets of separable and inseparable $\quinv$-non-attacking super fillings of $\dg'(\lambda)$, respectively. Let  $N(\lambda')$ be the set of all non-attacking super fillings of $\dg'(\lambda)$. Theorem \ref{thm11} follows directly from the two equalities:
		\begin{align}\label{eq5.1}
			&\quad\sum_{\sigma\in N(\lambda')  } (-1)^{m(\sigma)} q^{\maj(\sigma)} t^{p(\sigma)+\quadcoinv(\sigma)}x^{|\sigma|}\notag\\
			&=\sum_{\sigma\in G(\lambda')  } (-1)^{m(\sigma)} q^{\maj(\sigma)} t^{p(\sigma)+\quinv(\sigma)}x^{|\sigma|}.\\
			\label{eq5.2}
			&\quad\sum_{\sigma\in B(\lambda') } (-1)^{m(\sigma)} q^{\maj(\sigma)} t^{p(\sigma)+\quinv(\sigma)}x^{|\sigma|}=0.
		\end{align}
		We will prove \eqref{eq5.1} and \eqref{eq5.2} bijectively in Theorems \ref{thm5.1} and \ref{thm5.2} respectively.
		\begin{theorem}\label{thm5.1}
			For a partition $\lambda$, there exists a bijection
			$\psi: G(\lambda')\rightarrow N(\lambda')$ with the following properties for all $\sigma\in G(\lambda')$: {
				\begin{enumerate}
					\item $|\sigma|\sim|\psi(\sigma)|$ (row equivalence),
					\item $(\maj,\quinv,p)\sigma=(\maj,\quadcoinv,p)\psi(\sigma)$,
					\item the top rows of $|\psi(\sigma)|$ and $|\sigma|$ are identical.
			\end{enumerate} }
			Consequently, (\ref{eq5.1}) holds.
		\end{theorem}

		\begin{proof}
			We begin by establishing the bijection $\psi$ in the case that $\dg'(\lambda)$ contains at most two rows; we then prove that $\psi$ can be extended to an arbitrary diagram $\dg'(\lambda)$.
			
			If the diagram $\dg'(\lambda)$ has only one row, then by definition $\quadcoinv(\sigma)=\quinv(\sigma)$ and $\maj(\sigma)=0$. In this base case, we simply set $\psi(\sigma)=\sigma$.
			{Now let $\sigma=\sigma_2\sqcup\sigma_1$ be a two-row  $\quinv$-non-attacking  separable filling. Then $\Omega(\sigma_2)$ is a non-attacking filling and $\rho_{w}(\sigma_2)=\Omega(\sigma_2)$ for some permutation $w$. If there are $n$ restricted boxes in $\Omega(\sigma_2)$, i.e., the first $n$ columns of $\Omega(\sigma_2)$ are
				\[ \begin{ytableau}[] a_1\\ b_1
				\end{ytableau}
				\begin{ytableau}[] a_2\\ b_2
				\end{ytableau}
				\begin{ytableau} \cdots\\ \cdots
				\end{ytableau}
				\begin{ytableau}[] a_n\\ b_n
				\end{ytableau}
				\]
				with $|a_i|=|b_i|$ for $1\le i\le n$. Suppose there are $k$ entries $a_i\in\Z_+$, namely $a_{i_1},a_{i_2},\ldots,a_{i_k}$. We then assign positive signs to the $i_1$th, $i_2$th,\ldots,$i_k$-th smallest entries among $\{a_1,\ldots,a_n \}$  and negative signs to the remaining $n-k$ entries. Denote by $\tau$ the filling formed by combining the $n$ current columns with the right unrestricted boxes and with $\sigma_1$. Define
				\begin{align}\label{E:psi}
					\psi(\sigma)=\xi_{w}^{-1}(\tau).
				\end{align}
				In the restricted columns, the upper entries contribute $n-k$ to the major index before and after the sign reassignment, i.e., the number of negative entries. All unrestricted columns are unchanged. Therefore  $\maj(\Omega(\sigma_2))=\maj(\tau)$ and $p(\sigma)=p(\psi(\sigma))$.
				Equations \eqref{E:maj11} and \eqref{E:maj1} give that
				$\maj(\sigma)=\maj(\Omega(\sigma_2))$ and $\maj(\tau)=\maj(\psi(\sigma))$. Consequently, $\maj(\sigma)=\maj(\psi(\sigma))$.
				
				Since the top-row entries  have distinct absolute values, changing an entry's sign does not affect the action of $\xi_i$ on the top row. Thus, $\psi$ satisfies condition $(3)$, and
				\begin{align}\label{eqprop5.2a}
					\quinv(\sigma|_2^2)=\quadcoinv(\psi(\sigma|_2^2)).
				\end{align}
				
				It remains to compare $\quinv$ and $\quadcoinv$ between row $1$ and $2$. In  $\Omega(\sigma_2)$, by \eqref{eq:a=b}, the restricted pairs in the first $n$ columns contribute
				\[ \sum_{j=1}^n (n-j)\chi(a_j\in\Z_+)=\sum_{j=1}^k (n-i_j) \]
				to $\quinv$.
				On the other hand,  for the first $n$ columns of $\tau$, the $i_j$th smallest entry in the top row is positive. According to Definition \ref{Def2.2}, it can form a quadruple coinversion with some larger entry $u$ in the same row, together with the two entries below, and there are $n-i_j$ such $u$. Hence, the number of quadruple coinversions arising in the first $n$ columns of $\tau$ is also $\sum_{j=1}^{k} (n-i_j)$.

				For quadruples $(z,w,u,v)$ arising between restricted columns and unrestricted columns, that is  $|z|=|u|, |w|\ne|v|$.  By \eqref{eq:a=b} and Definition \ref{Def2.2}, $\Q(z,w,u,v)=\Q(z,u,v)=\chi(z\in\Z_+)$.
				Since the sign reassignment preserves the number of positives entries, the number of queue inversion triples between restricted and unrestricted columns in $\Omega(\sigma_2)$ equals the number of quadruple coinversions between  between restricted and unrestricted columns in $\tau$.

				Finally, any quadruple  $(z,w,u,v)$ lying to the right of column $n$ in $\tau$ and  $\Omega(\sigma_2)$ is unchanged and satisfies $\Q(z,w,u,v)=\Q(z,u,v)$. Hence,
				\begin{align} \label{eqprop5.2b}
					\quinv( \Omega(\sigma) )-\quinv( \Omega(\sigma)|_2^2 )=\quadcoinv( \tau )-\quadcoinv( \tau|_2^2 ).
				\end{align}
				Lemmas \ref{lem3.1} and \ref{lem3.2} say
				\begin{align}
					\quinv( \sigma )-\quinv( \sigma|_2^2 )&=\quinv( \Omega(\sigma) )-\quinv( \Omega(\sigma)|_2^2 ) \label{eqprop5.2c} \\
					\quadcoinv( \tau )-\quadcoinv( \tau|_2^2 )&=\quadcoinv( \psi(\sigma )-\quadcoinv( \psi(\sigma)|_2^2 ). \label{eqprop5.2d}
				\end{align}
				Combining  equations \eqref{eqprop5.2a}, \eqref{eqprop5.2b}-\eqref{eqprop5.2d}, we conclude that $\quinv( \sigma )=\quadcoinv( \psi(\sigma ))$.
				.
				
				The bijection $\psi:G(\lambda')\rightarrow N(\lambda')$ for an arbitrary partition $\lambda$ can be defined inductively from the bottom row upward. More precisely,
				suppose that the bottom $r$ rows have already been transformed into a non-attacking filling $U$. Apply the two-row map $\psi$ from
				\eqref{E:psi} to the top row of $U$ and row $r+1$ of the original filling, and denote the result by $T$. Let $w$ be the unique
				permutation taking the  top row of $U$ to the  bottom row of $T$. Replace $U$ by $\xi_w^r(U)$, and adjoin the top row of $T$.
				From the $\psi$ for the two-row case, we see that the bijection $\psi$ for an arbitrary diagram has properties (1) and (3).
				When $\psi$ is applied between rows $r$ and $r+1$, since $I(a,b)=I(a,\overline{b})$ for every $a,b\in\A$, and $\Q(a,b,c)=\Q(a,\overline{b},c)$, $\Q(a,b,c)=\Q(a,b,\overline{c})$ for every $|b|\ne|c|$, the major index and  the number of queue inversion triples between rows $r+1$ and $r+2$ are unchanged.
				Together with  Lemmas \ref{lem3.1} and \ref{lem3.2}, which imply that the change of statistics occurs only between the starting row $r$ and row $r+1$, the bijection $\psi$ also has property (2). This completes the proof. }
		\end{proof}	
		
		\begin{example}
			We continue with the separable super filling $\sigma$ from Example \ref{Eg:separable}.
			{Since there are two restricted boxes in $\Omega(\sigma)$, and only the entry in column $2$ is positive, we only set the $2$th smallest entry among $\{3,\overline{4} \}$ to be positive, as shown on the second filling below.
				Since $\Omega=\rho_2\, \rho_3\, \rho_1\, \rho_2$, we continue to apply  $\xi_2\, \xi_1\, \xi_3\,\xi_2$ to form  $\psi(\sigma)$ by (\ref{E:psi}).
				\begin{align*}
					\Omega(\sigma)&=\begin{ytableau}[] *(green)\overline{4} & *(green)3& 7& 5 \\  *(green)4& *(green)3& 6& 1
					\end{ytableau}\rightarrow
					\begin{ytableau}[] *(green)4 & *(green)\overline{3}& 7& 5 \\  *(green)4& *(green)3& 6& 1
					\end{ytableau} \xrightarrow{\xi_2}
					\begin{ytableau}[] *(green)4 & 7&*(green)\overline{3} & 5 \\  *(green)4&  6& *(green)3& 1
					\end{ytableau}\xrightarrow{\xi_3}
					\begin{ytableau}[]  *(green)4 & 7& 5& *(green)\overline{3} \\  *(green)4&  6 & 1 & *(green)3
					\end{ytableau} \\
					&\quad\xrightarrow{\xi_1}
					\begin{ytableau}[]  7 & *(green)4 & 5& *(green)\overline{3} \\    6& *(green)4 & 1 & *(green)3
					\end{ytableau} \xrightarrow{\xi_2}
					\begin{ytableau}[] 7 & 5& *(green)4 & *(green)\overline{3} \\  6 & 1 & *(green)4& *(green)3
					\end{ytableau}=\psi(\sigma).
				\end{align*}
				Clearly, $\psi(\sigma)$ is a non-attacking filling and  $|\psi(\sigma)|\sim|\sigma|$. Moreover, $\quinv(\sigma)=\mathrm{quadcoinv}(\psi(\sigma))=2$ and $\maj(\sigma)=\maj(\psi(\sigma))=3$.  }
		\end{example}
		We now construct an involution $\varphi$ to prove \eqref{eq5.2} in Theorem \ref{thm5.2}.
		\begin{theorem}\label{thm5.2}
			For a partition $\lambda$, there is an involution
			$\varphi: B(\lambda')\rightarrow B(\lambda')$ with the following properties for
			$\sigma\in B(\lambda')$:
			\begin{enumerate}
				\item $|\sigma|\sim |\varphi(\sigma)|$,
				\item $p(\sigma)-p(\varphi(\sigma))=\quinv(\varphi(\sigma))-\quinv(\sigma)=\pm 1$,
				\item $\maj(\sigma)=\maj(\varphi(\sigma))$,
				\item the top rows of $|\varphi(\sigma)|$ and $|\sigma|$ are identical.
			\end{enumerate}
			Consequently, (\ref{eq5.2}) holds.
		\end{theorem}
		\begin{proof}
			As in the proof of Theorem \ref{thm5.1}, we first construct the involution $\varphi: B(\lambda')\rightarrow B(\lambda')$ for the case where $\lambda'$ has two parts, and then extend it to act on super fillings of an arbitrary diagram $\dg'(\lambda)$.
			
			Since a single-row filling is trivially separable, any partition $\lambda'$ for which $B(\lambda')\ne \varnothing$ must contain at least two parts. Suppose
			$\sigma=\sigma_2\sqcup\sigma_1$ is inseparable. Then $\sigma_2$ must be inseparable.
			Define $\varphi(\sigma_2)$ via the following procedure: Apply $\Omega$ to $\sigma_2$. Locate the leftmost box $u$ in the top row of $\Omega(\sigma_2)$ that is part of an attacking pair of $\Omega(\sigma_2)$.  Change the sign of the entry $\Omega(\sigma_2)(u)$.  Apply the inverse map $\Omega^{-1}$ to the resulting filling. The outcome is defined to be $\varphi(\sigma_2)$. We then set $\varphi(\sigma)=\varphi(\sigma_2)\sqcup \sigma_1$.
			
			It is evident that $\varphi$ has property $(1)$. {Since the top-row entries of $\sigma$ have distinct absolute values, changing an entry's sign does not affect the action of $\rho_i^r$ on the top row}. Thus, $\varphi$ satisfies condition $(4)$. We will show that $\varphi$ is an involution with properties $(2)$ and $(3)$. Let $v$ be any box in the bottom row that attacks $u$ in $\Omega(\sigma_2)$. Let $a=\sigma(u)$ and $b=\sigma(v)$, then $|a|=|b|$.
			The process of moving entry $a$ to create
			$\Omega(\sigma_2)$ results in one of the two configurations (i) or (ii) depicted in Figure \ref{F:ab}. Without loss of generality, we assume $a\in \mathbb{Z}_{+}$.
			\begin{figure}
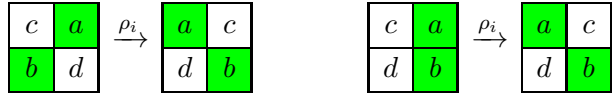

				\centering
				\begin{align*}
					\begin{ytableau}[] c& *(green) a \\ *(green) b & d
					\end{ytableau} \,\xrightarrow{\rho_i}\,
					\begin{ytableau}[] *(green) a & c \\ d& *(green) b
					\end{ytableau}
					\qquad \qquad
					\begin{ytableau}[] c& *(green) a  \\ d & *(green) b
					\end{ytableau}\, \xrightarrow{\rho_i}\,
					\begin{ytableau}[] *(green) a & c \\ d& *(green) b
					\end{ytableau}
				\end{align*}
				\caption{Subcases (i) and (ii) for inseparable super fillings.}\label{F:ab}
			\end{figure}
			For subcase (i), since $\rho_i$ swaps $b$ and $d$, we have $\Q(c,b,d)\ne\Q(a,b,d)$ by definition. We note that $\Q(a,b,d)=1$ because $I(a,b)=0$ and $I(d,b)\ne I(a,d)$. Consequently, $\Q(a,d,b)=0$ and $\Q(c,d,b)=1$. Now, consider changing the entry $a$ to its negative $\overline{a}$. Since $|a|\ne|d|$, the sign change leaves the major index of the filling invariant. We now discuss the change on the number of queue inversion triples.
			
			Because $I(\overline{a},b)=1$, we obtain $\Q(\overline{a},d,b)=1$. For any entry $e$ with $|e|\ne |a|$, we find $\Q(a,d,e)=\Q(\overline{a},d,e)$. It follows that the number of queue inversion triples between these two rows increases by exactly one. Since all entries in the top row are distinct, the number of queue inversions induced solely by the top row remains unchanged. Finally, applying $\rho_i$ swaps only the entries $\bar{a}$ and $c$, as the condition $\Q(\overline{a},d,b)=\Q(c,d,b)=1$ is satisfied (see the diagram below).
			\begin{align*}
				\sigma_2\,\rightarrow\,\begin{ytableau}[] c& *(green) a  \\ *(green) b & d
				\end{ytableau} \,\xrightarrow{\rho_i}\,
				\begin{ytableau}[] *(green) a & c \\ d& *(green) b
				\end{ytableau}\,\rightarrow\,\Omega(\sigma_2)\,\rightarrow\,
				\begin{ytableau}[] *(green) \overline{a} & c \\ d& *(green) b
				\end{ytableau}\,\xrightarrow{\rho_i}\,
				\begin{ytableau}[] c&  *(green) \overline{a}  \\ d& *(green) b
				\end{ytableau}\,\rightarrow\,\varphi(\sigma_2)
			\end{align*}
			Lemma \ref{lem3.1} states that the inverse map $\Omega^{-1}$, being a product of queue-inversion flip operators, preserves both the major index and the number of queue inversion triples between these two rows. Therefore,
			\begin{align*}
				\quinv(\sigma)&=\quinv(\varphi(\sigma))-1,\\
				\maj(\sigma)&=\maj(\varphi(\sigma)).
			\end{align*}
			The resulting filling $\varphi(\sigma_2)$ remains inseparable, because the reverse construction (illustrated below) confirms that $\varphi(\sigma_2)$ falls into subcase (ii); namely, the entries $\overline{a}$ and $b$ lie in the same column. Furthermore, the top rows of the positive fillings $|\sigma_2|$ and $|\varphi(\sigma_2)|$ are identical.
			\begin{align*}
				\varphi(\sigma_2)\,\rightarrow\,\begin{ytableau}[] c&  *(green) \overline{a}  \\ d& *(green) b
				\end{ytableau}\,\xrightarrow{\rho_i}\,
				\begin{ytableau}[] *(green) \overline{a} & c \\ d& *(green) b
				\end{ytableau}\,\rightarrow\,\Omega(\varphi(\sigma_2))\,\rightarrow\,
				\begin{ytableau}[] *(green) a & c \\ d& *(green) b
				\end{ytableau}\,\xrightarrow{\rho_i}\,
				\begin{ytableau}[] c& *(green) a  \\ *(green) b & d
				\end{ytableau}\,\rightarrow\,\sigma_2
			\end{align*}
			Thus, $\varphi$ is a sign-reversing and weight-preserving involution on the set of two-row inseparable super fillings satisfying conditions $(1)$--$(4)$.
			
			For $\sigma=\sigma_{\lambda_1} \sqcup \cdots \sqcup \sigma_1\in B(\lambda')$. Find the largest index $i$ for which $\sigma_i$ is inseparable; and choose the smallest index $r$ for which $\sigma_i|_r^{r+1}$ is inseparable. This implies that $\sigma_i|_1^{r}$ is separable. By Theorem \ref{thm5.1} and the involution $\varphi$ for two-row super fillings, {the top row of $|\psi(\sigma_i|_1^{r})|$ and the bottom row of  $|\varphi(\sigma_i\vert_{r}^{r+1})|$ are row-equivalent. That is, one can be transformed into the other by a permutation. Let $w$ be a permutation such that the top row of $|\xi_{w}\psi(\sigma_i|_1^{r})|$ matches the bottom row of $|\varphi(\sigma_i\vert_{r}^{r+1})|$}. We then define $\varphi(\sigma_i)$ to be the super filling satisfying
			\begin{align*}
				\varphi(\sigma_i)\vert_{1}^{r}=\psi^{-1}\,\xi_{w}\,\psi(\sigma_i|_1^{r}),\quad
				\varphi(\sigma_i)\vert_{r+1}^{r+1}=\varphi(\sigma_i\vert_{r}^{r+1})|_{r+1}^{r+1}\quad \mbox{ and }\quad
				\varphi(\sigma_i)\vert_{r+2}^{i}=\sigma_i\vert_{r+2}^{i}.
			\end{align*}
			Finally, we define
			\begin{align*}
				\varphi(\sigma)=\sigma_{\lambda_1} \sqcup \cdots \sqcup \sigma_{i+1}\sqcup \varphi(\sigma_i)\sqcup \sigma_{i-1}\sqcup \cdots \sqcup \sigma_1.
			\end{align*}
			By construction, the involution $\varphi$ clearly has properties $(1)$ and $(4)$ for an arbitrary partition $\lambda$. It remains to verify conditions $(2)$ and $(3)$. By Lemma  \ref{lem3.2} and Theorem \ref{thm5.1}, we have
			\begin{align*}
				(\maj,\quinv)(\varphi(\sigma_i)\vert_{1}^{r})&=(\maj,\quinv)(\sigma_i|_1^{r}), \\
				(\maj,\quinv)(\varphi(\sigma_i)\vert_{r+2}^{i})&=(\maj,\quinv)(\sigma_i\vert_{r+2}^{i}).
			\end{align*}
			{Here, both $\quinv$ counts exclude the queue inversion triples arising from the top row  of the truncated fillings.} 
			Moreover, if $|x|\ne |y|$ and $|a|\ne |y|$, then $\Q(x,a,y)=\Q(x,\overline{a},y)$ and $\Q(x,y,a)=\Q(x,y,\overline{a})$. That is, changing the sign of $a$ does not affect the number of queue inversion triples above row $r$. Therefore, the only change of statistics $(\maj,\quinv)$ occurs between rows $r$ and $r+1$. By the previous argument on two-row fillings, $\varphi$ fulfills conditions $(2)$ and $(3)$, as desired.
		\end{proof}
		
		\begin{example}\label{example5.2}
			Let $\lambda'=(5,5)$ and let $\sigma$ be the inseparable $\quinv$-non-attacking super filling of $\dg(\lambda')$ shown below. According to Theorem \ref{thm5.2}, its image $\varphi(\sigma)$ is constructed as follows.\\
			\begin{align*}
				\sigma&=\,\begin{ytableau}[] 8& *(green)6 & 4 & \overline{2} & *(green)3 \\ *(green)\overline{6}& 7& *(green)3& 5& 1
				\end{ytableau} \xrightarrow{\rho_1}
				\begin{ytableau}[] *(green)6& 8 & 4 & \overline{2} & *(green)3 \\ *(green)\overline{6}& 7& *(green)3& 5& 1
				\end{ytableau}\xrightarrow{\rho_4}
				\begin{ytableau}[] *(green)6& 8 & 4 & *(green)3& \overline{2}  \\ *(green)\overline{6}& 7& *(green)3& 5& 1
				\end{ytableau} \\[5pt]
				&\,\,\xrightarrow{\rho_3}
				\begin{ytableau}[]  *(green)6& 8 & *(green)3 & 4& \overline{2}  \\ *(green)\overline{6}& 7& 5& *(green)3& 1
				\end{ytableau}
				\rightarrow\begin{ytableau}[]  *(green)6& 8 & *(green)\overline{3} & 4& \overline{2}  \\ *(green)\overline{6}& 7& 5& *(green)3& 1
				\end{ytableau}\xrightarrow{\rho_3}
				\begin{ytableau}[]  *(green)6& 8 & 4& *(green)\overline{3} & \overline{2}  \\ *(green)\overline{6}& 7& 5& *(green)3& 1
				\end{ytableau}\\[5pt]
				&\,\,\xrightarrow{\rho_4}
				\begin{ytableau}[]  *(green)6& 8 & 4& \overline{2} & *(green)\overline{3}  \\ *(green)\overline{6}& 7& 5& 1& *(green)3
				\end{ytableau}\xrightarrow{\rho_1}
				\begin{ytableau}[]  8& *(green)6 & 4& \overline{2} & *(green)\overline{3}  \\ *(green)\overline{6}& 7& 5& 1& *(green)3
				\end{ytableau}\,=\varphi(\sigma).
			\end{align*}
		\end{example}
		
		\section{Proof of Theorem \ref{thm1}}\label{S:4}
		The purpose of this section is to prove combinatorial formulas (\ref{eqthm1.2}) -- (\ref{eqthm2.2}) for symmetric Macdonald polynomials and integral Macdonald polynomials, as the second step towards the proof of Theorem \ref{T:main1}.
		
		We prove (\ref{eqthm1.2}) by compressing the weighted contributions of super fillings into those of positive fillings. This follows the superization--compression approach used in \cite[Proposition 8.1]{HHL04} and \cite[Theorem 5.3]{O:23}.
		
		{\em Proof of (\ref{eqthm1.2})}.
		Starting from (\ref{E:supJ2}) and (\ref{eq1.1}), we obtain
		\begin{align}
			J_{\lambda}(X;q,t)
			&=t^{n(\lambda)+n} \sum_{\sigma\in\M(\lambda') \atop |\sigma| \ \textrm{non-attacking} } (-1)^{m(\sigma)} q^{\maj(\sigma)}
			t^{-p(\sigma)-\quadcoinv(\sigma)}x^{|\sigma|} \label{eqthm1.1}
		\end{align}
		Now, fix a positive non-attacking filling $\tau\in \T(\lambda')$. We analyze the total contribution to the right-hand side of (\ref{eqthm1.1}) from all non-attacking super fillings $\sigma\in \M(\lambda')$ satisfying $\tau=|\sigma|$.
		
		Consider a restricted box $s$ of $\tau$, i.e., one satisfying $\tau(s)=\tau(\South(s))$. Let $u=\sigma(s)$ and $v=\sigma(\South(s))$. If $u=\tau(s)\in\Z_+$, then
		$I(u,v)=I(u,\tau(\South(s)))=0$. This implies that the positive entry $\sigma(s)$  contributes zero to the difference $\maj(\sigma)-\maj(\tau)$. Furthermore, by Definition \ref{Def2.2} and the non-attacking property of $\tau$, the positive entry $u$ also contributes nothing to the difference $\quadcoinv(\sigma)-\quadcoinv(\tau)$.
		
		If $u=\overline{\tau(s)}\in\Z_-$, then $I(u,v)=1$ and $I(\bar{u},\tau(\South(s)))=0$. Since $I(a,b)=I(a,\overline{b})$ for all $a,b\in\A$, we conclude that box $s$ contributes $\leg(s)+1$ to the difference $\maj(\sigma)-\maj(\tau)$.
		{Assume that box $s$ lies in a rectangle $\tau_j$ of width  $m$, and there are $h$ boxes   to the right of $\South(s)$ contained in columns shorter than the column containing $s$. If there are $k$ restricted boxes in the same row of $s$ in $\tau_j$, then the $\widehat{\arm}$-lengths of these $k$ boxes from left to right are $m+h-1,m+h-2,\ldots,m+h-k$.
			
			We claimed that if  $u$ is the $i$th smallest entry among these $k$ restricted boxes, then  the contribution of the entry $u$ (compare to $|u|$) to the difference $\quadcoinv(\sigma)-\quadcoinv(\tau)$ is $-(m+h-i)$. If this holds, then together with the $\widehat{\arm}$-lengths of the $k$ restricted boxes, we obtained that the total contribution of all super fillings $\sigma$ (with $|\sigma|=\tau$) to the right-hand side of (\ref{eqthm1.1}) for all restricted box $s$ is
			\begin{align*}
				q^{\maj(\tau)}t^{-\quadcoinv(\tau)}x^{\tau} \prod_{\tau(s)=\tau(\South(s))}   (t^{-1}-q^{\leg(s)+1} t^{\widehat{\arm}(s)}).
			\end{align*}
			
			We now examine potential quadruples that include the entry $u$. Consider a quadruple in which box $s$ occupies either the upper-left or upper-right corner (as illustrated below). In such a configuration, the columns containing the pairs $(a,b)$ and $(u,v)$ have equal height. (The entry $c$ may be absent in some cases.)
			\begin{center}
				\begin{ytableau}
					a& \none[\dots]&		*(green)u& \none[\dots]& c
					\\
					b& \none[\dots]&		v &  \none[\dots] & d \\
				\end{ytableau}
			\end{center}
			
			If the columns containing $(u,v)$ and $(c,d)$ have equal height with $|c|\ne|d|$, or the former is taller than the latter,
			Definition \ref{Def2.2} and \eqref{eq:a=b} give $\Q(u,c,v,d)=\Q(u,v,d)=0, \Q(|u|,c,v,d)=\Q(|u|,v,d)=1$.
			
			For $(a,u,b,v)$, non-attackingness gives $|a|\ne|u|$ and $|b|\ne|u|$. If $|a|\ne|b|$, then $\Q(a,u,b,v)=0$ and $\Q(a,|u|,b,v)=1$.
			If $|a|=|b|$, the contribution depends on the sign of
			$\min\{a,u\}$, so replacing $|u|$ by $u$ decreases it by one precisely when $|a|>|u|$.
			Similarly, when $(c,d)$ has the same height as $(u,v)$ and $|c|=|d|$, the contribution decreases by one precisely when $|c|>|u|$.
			
			Since there are $i-1$ entries among these $k$ restricted boxes that are smaller than $u$, the contribution of the entry $u$ (compare to $|u|$) to the difference $\quadcoinv(\sigma)-\quadcoinv(\tau)$ is $-(m+h-1-(i-1))=-(m+h-i)$.

			Now suppose box $s$ occupies the lower-left or lower-right corner of the quadruple square (as illustrated below). We claim that in these configurations, the negative entry $u$ contributes zero to the difference $\quadcoinv(\sigma)-\quadcoinv(\tau)$.
			\begin{center}
				\begin{ytableau}
					e& \none[\dots]&		z & \none[\dots]& w
					\\
					f& \none[\dots]&		*(green)u &  \none[\dots] & y \\
				\end{ytableau}
			\end{center}
			To establish the claim, it suffices to prove that
			\begin{align}
				\Q(e,z,f,u)&=\Q(e,z,f,|u|), \quad \Q(e,f,u)=\Q(e,f,|u|), \label{E:Q1a}\\
				\Q(z,w,u,y)&=\Q(z,w,|u|,y),\quad \Q(z,u,y)=\Q(z,|u|,y)\label{E:Q2a}.
			\end{align}
			We proceed case by case, leveraging the non-attacking property of $\tau$.
			
			First, we prove (\ref{E:Q1a}). Suppose $|z|=|u|$ and the columns of $(z,u)$ and $(e,f)$ have equal height. Then whether $(e,z,f,u)$ is a  quadruple coinversion depends only on the upper entries $e$ and $z$, and is independent of $u$. Hence, $\Q(e,z,f,u)=\Q(e,z,f,|u|)$.
			Otherwise, $\Q(e,z,f,u)=\Q(e,f,u)$ and $\Q(e,z,f,|u|)=\Q(e,f,|u|)$.  Since $\tau$ is non-attacking, we have $|e|\ne |u|$ and $|f|\ne|u|$, which implies $I(u,f)=I(|u|,f)$ and $I(e,u)=I(e,|u|)$. Consequently, $\Q(e,f,u)=\Q(e,f,|u|)$. This proves (\ref{E:Q1a}).

			Second, we verify (\ref{E:Q2a}). Suppose $|w|=|y|$ and the columns containing $(z,u)$ and $(w,y)$ have equal height. Then whether $(z,w,u,y)$ is a quadruple coinversion depends only on the
			upper entries $z,w$, and not on $u$. Hence $\Q(z,w,u,y)=\Q(z,w,|u|,y)$. Otherwise, we have $\Q(z,w,u,y)=\Q(z,u,y)$ and $\Q(z,w,|u|,y)=\Q(z,|u|,y)$. The non-attacking property of $\tau$ yields $|y|\ne |u|$ and $|y|\ne|z|$, from which it follows that $I(z,u)=I(z,\overline{u})$ and $I(y,u)=I(y,\overline{u})$. Therefore, $\Q(z,u,y)=\Q(z,|u|,y)$.
		}

		Now consider an unrestricted box $s$ of $\tau$, satisfying $\tau(s)\ne\tau(\South(s))$. This includes the case where $s$ lies in the bottom row. Since $\tau(s)\ne\tau(\South(s))$, the relation $I(\sigma(s),\sigma(\South(s)))=I(\tau(s), \tau(\South(s)))$ holds for any choice of $\sigma(s)\in \A$. Hence, box $s$ contributes nothing to the difference $\maj(\sigma)-\maj(\tau)$. A similar analysis of all quadruples containing $\sigma(s)$ shows that it also contributes zero to $\quadcoinv(\sigma)-\quadcoinv(\tau)$. Therefore, the total contribution from all super fillings of $\sigma$ (with $|\sigma|=\tau$) to the right-hand side of (\ref{eqthm1.1}) for each unrestricted box $s$ is
		\begin{align*}
			q^{\maj(\tau)}t^{-\quadcoinv(\tau)}x^{\tau}(t^{-1}-1).
		\end{align*}
		Combining the contributions from all restricted and unrestricted boxes, (\ref{eqthm1.1}) simplifies to
		\begin{align*}
			J_{\lambda}(X;q,t)
			&=t^{n(\lambda)+n} \sum_{\tau\in \T(\lambda') \atop  \tau\ \textrm{non-attacking} } q^{\maj(\tau)} t^{-\quadcoinv(\tau)} x^{\tau}\\
			&\qquad\qquad\times\prod_{\tau(s)\ne \tau(\South(s))}(t^{-1}-1)
			\prod_{\tau(s)=\tau(\South(s))} (t^{-1}-q^{\leg(s)+1} t^{\widehat{\arm}(s)}),
		\end{align*}
		which is precisely \eqref{eqthm1.2}. This completes the proof.
		\qed
		
		{\em Proof of (\ref{eqthm2.1})}. We begin by reformulating the polynomial $b_{\lambda}(q,t)$ from (\ref{E:bqt}). An equivalent product expression is
		\begin{align}\label{E:bla2}
			b_{\lambda}(q,t)=\prod_{s\in \dg'(\lambda)}(1-q^{\leg(s)}t^{\arm(s)+1}).
		\end{align}
		For any integer $n\ge 1$, define the $t$-factorial bracket $(t;t)_{n}=\prod_{1\le i\le n}(1-t^i)=(1-t)^n [n]_t!$,  with $(t;t)_0=1$. Observe that $\leg(s)=0$ for every topmost box $s$ in a column of $\dg'(\lambda)$. These boxes collectively contribute the factor
		\begin{align*}
			\prod_{i\ge 1}(t;t)_{m_i}=(1-t)^{\ell(\lambda)} \prod_{i\ge 1}[m_i]_t!
		\end{align*}
		to $b_{\lambda}(q,t)$, where $m_i$ is the multiplicity of part $i$ in $\lambda$. Therefore, we obtain the factorization
		\begin{align}\label{eqJ.2}
			b_{\lambda}(q,t)=(1-t)^{\ell(\lambda)} \prod_{i\ge 1}[m_i]_t!\prod_{s\in\dg'(\lambda) \atop s \not\in\, \text{row }1}(1-q^{1+\mathrm{leg}(s)}t^{1+\mathrm{arm}(\South(s))}).
		\end{align}
		
		For any filling $\sigma\in \T(\lambda')$, we prove that the second product of (\ref{eqJ.2}) equals
		\begin{align}\label{E:b1}
			\prod_{s\in\dg'(\lambda) \atop s \not\in\, \text{row }1}(1-q^{1+\mathrm{leg}(s)}t^{1+\mathrm{arm}(\South(s))}) &=\prod_{\sigma(s)=\sigma(\South(s))}(1-q^{1+\mathrm{leg}(s)}t^{1+\widehat{\arm}(s)}) \notag\\
			&\qquad \times
			\prod_{\sigma(s)\ne\sigma(\South(s)) \atop s \not\in\, \text{row }1}(1-q^{1+\mathrm{leg}(s)}t^{1+\overline{\arm}(s)}).
		\end{align}
		Fix row $i+1$ in a rectangle $\sigma_j$ of width
		$m=\lambda_j'-\lambda_{j+1}'$, and put $h=\lambda_i'-\lambda_j'$.
		If that row has $k$ restricted boxes, its three arm multisets are
		\begin{align*}
			\{\arm(\South(s)):s\text{ in this row}\}&=\{h,\ldots,h+m-1\},\\
			\{\widehat{\arm}(s):s\text{ restricted in this row}\}
			&=\{h+m-k,\ldots,h+m-1\},\\
			\{\overline{\arm}(s):s\text{ unrestricted in this row}\}
			&=\{h,\ldots,h+m-k-1\}.
		\end{align*}
		The last two multisets partition the first. Since the leg-length is constant
		on this row of the rectangle, the factors $1-q^{\leg(s)+1}t^{a+1}$ pair as
		products. Multiplying this rowwise identity over all rows and rectangles proves
		\eqref{E:b1}.
		Finally, combining (\ref{E:b1}) with (\ref{eqthm1.2}) and (\ref{eqJ.2}) yields
		\eqref{eqthm2.1}, completing its proof.
		\qed
		
		The factor $\prod_{i\ge 1}[m_i]_t!^{-1}$ will be cancelled by accounting for the multiplicities that arise from permuting the top-row entries of top-row increasing tableaux. 

		{\em Proof of (\ref{eqthm2.2})}. Let $S(\lambda')$ be the set of top-row increasing tableaux of shape $\dg'(\lambda)$. Equation (\ref{eqthm2.2}) is equivalent to
		\begin{align}\label{eqpfthm2.3}
			&\sum_{\sigma\in\T(\lambda') \atop \sigma \ \text{non-attacking} }
			q^{\maj(\sigma)}t^{\overline{\quadcoinv(\sigma)}}c_{\sigma}(q,t)x^{\sigma} \notag\\
			&=\prod_{i\ge 1}[m_i]_t!\sum_{\sigma\in S(\lambda') \atop \sigma \ \text{non-attacking}}
			q^{\maj(\sigma)}t^{\overline{\quadcoinv}(\sigma)}c_{\sigma}(q,t)x^{\sigma},
		\end{align}
		where $c_{\sigma}(q,t)$ is defined in (\ref{E:cqt}). For $\pi=\pi_1\cdots \pi_n\in S_n$ in one-line notation, let $\inv(\pi)$ be the number of inversions of $\pi$. The coinversion number of $\pi$ is
		\begin{align*}
			\overline{\inv}(\pi)=|\{(i,j): 1\le i<j\le n \,\textrm{ and }\, \pi_i<\pi_j\}|.
		\end{align*}
		Recall from \cite{LN12}  that
		\begin{align*}
			\sum_{\pi\in S_{n}} t^{\inv(\pi)}=\sum_{\pi\in S_{n}} t^{\overline{\inv}(\pi)}=[n]_t!,
		\end{align*}
		which implies the product identity
		\begin{align}\label{E:mit}
			\prod_{i\ge 1} \sum_{\pi\in S_{m_i}} t^{\inv(\pi)}=\prod_{i\ge 1}[m_i]_t!.
		\end{align}
		We will use (\ref{E:mit}) to give a combinatorial proof of (\ref{eqpfthm2.3}). Let $\sigma\in S(\lambda')$ be an arbitrary non-attacking and top-row increasing tableau. Consider its factorization $\sigma=\sigma_{\lambda_1} \sqcup \cdots \sqcup \sigma_1$, where each $\sigma_j$ is a rectangular filling of height $j$. By Definitions \ref{Def:non-attacking} and \ref{Def:quad_sorted}, the entries in the top row of every $\sigma_j$ are strictly increasing. 
		
		Recall the flip operator $\xi_{w}^r$ paired with a permutation $w$ as defined by (\ref{E:wflip2}). 
		{Fix a rectangle $\sigma_j$ with top row $a_1<\cdots<a_n$.
			For $w\in S_n$, the operator $\xi_w$ sends this row to
			$(a_{w^{-1}(1)},\ldots,a_{w^{-1}(n)})$. Along its fixed reduced
			expression, Lemma \ref{lem3.2} preserves the major index and all
			quadruple-coinversion contributions below the top row. Hence
			\begin{align}
				\maj(\xi_w(\sigma))&=\maj(\sigma),\label{E:sort1}\\
				\overline{\quadcoinv}(\xi_{w}(\sigma))-\overline{\quadcoinv}(\sigma)
				&=\quadcoinv(\sigma)-\quadcoinv(\xi_w(\sigma)),\notag\\
				&=\overline{\inv}(a)-\overline{\inv}(w)\notag\\
				&=\inv(w)=\ell(w).\label{E:sort2}
			\end{align}
			Conversely, the distinct top-row entries of an arbitrary
			non-attacking filling $\tau$ determine a unique position permutation
			$w\in S_{m_{\lambda_1}}\times\cdots\times S_{m_1}$ relative to their
			increasing orders. Set $\sigma=\xi_w^{-1}(\tau)$, using the reverse
			of the fixed operator word for $w$. Then $\sigma$ is non-attacking
			and top-row increasing, and $\tau=\xi_w(\sigma)$. Thus the pairs
			$(\sigma,w)$ parametrize all non-attacking fillings bijectively.
			Consequently,
			
			\begin{align}
				\sum_{w\in S_{m_{\lambda_1}}\times \cdots \times S_{m_1}} q^{\maj(\xi_w(\sigma))} t^{\overline{\quadcoinv}(\xi_w(\sigma))}x^{\sigma}
				&=q^{\maj(\sigma)}t^{\overline{\quadcoinv}(\sigma)}x^{\sigma} \sum_{w\in S_{m_{\lambda_1}}\times \cdots \times S_{m_1}}t^{\inv(w)}\notag\\
				&=\prod_{i\ge 1}[m_i]_t!\,q^{\maj(\sigma)}t^{\overline{\quadcoinv}(\sigma)}x^{\sigma},\label{E:sigma1}
		\end{align} }
		where the last equality uses (\ref{E:mit}). Finally, Lemma \ref{lem3.2} ensures that the number of (un)restricted boxes in each row is invariant under $\xi_{w}$, and Remark \ref{R:1} states that the set of $\overline{\arm}$-lengths of unrestricted boxes is independent of their positions. Therefore, we obtain $c_{\sigma}(q,t)=c_{\tau}(q,t)$ for $\tau=\xi_{w}(\sigma)$. By summing over all non-attacking $\sigma\in S(\lambda')$ in (\ref{E:sigma1}) , we completes the proof of (\ref{eqpfthm2.3}) and (\ref{eqthm2.2}).
		\qed
		
		\section{Proofs of Theorems \ref{thm3}, \ref{T:main1} and \ref{cor:1}}\label{S:7}
		This section focuses on the equidistribution of pairs $(\maj,\overline{\quadcoinv})$ and $(\maj,\eta^{\circ})$ over non-attacking and top-row increasing tableaux, as the final step in proving Theorem \ref{T:main1}. Subsequently, we prove Theorem \ref{cor:1} using the compact formula given in \cite{JL25}.
		
		We first recall the family $\A^+$ of statistics $\eta$ introduced in \cite[Definition 4]{JL25}. We then construct the desired bijection $\varphi$ as a composition of the quadruple-coinversion flip operators $\xi_i^r$ and the flip operators $\delta_i^r$. {
			\begin{definition}[{\em \color{blue} $\quinv$-quadruple sets}]\label{Def:setS}
				Let $\S$ be a set of ordering patterns for quadruples $(z,w,u,v)$, whose elements are determined by the following rules.
				\begin{enumerate}
					\item If $z\ne w$, then $(z,w,u,v)\in\S$ if and only if $(w,z,v,u)\in\S$. Otherwise $z=w$, then $(z,w,u,v)\in\S$ if and only if $\Q(z,u,v)=1$. That is, $z=w>v>u$, $u\ge z=w>v$, and $v>u\ge z=w$.
					\item No quadruple with $u=v$ belongs to $\S$. If $u\ne v$, then exactly one of $(z,w,u,v)$ and $(z,w,v,u)$ belongs to $\S$.
					\item The patterns $z>u>v\ge w$ and $v\ge z>w>u$ are either both included or both excluded.
					\item Both patterns $z>v\ge w>u$ and $u\ge z>v\ge w$ are included.
				\end{enumerate}
				A set $\S$ satisfying these requirements is called a  $\quinv$-quadruple set.
			\end{definition}
			According to Definition \ref{Def:setS}, we now enumerate all possible $\quinv$-quadruple sets $\S$ in the following lemma.
			
			\begin{lemma}\cite[Lemma 5.1]{JL25}\label{L:setSi}
				Let $\S^{*}=\{(z,w,u,v)\mid (z=w>v>u), (u\ge z=w>v),(v>u\ge z=w)\}$ be the set of patterns with $z=w$ in Definition ~\ref{Def:setS} (1), and let
				\[
				B=\{(z,w,u,v)\mid z>v\ge w>u,\ u\ge z>v\ge w\}
				\]
				be the set of patterns in Definition \ref{Def:setS} (4).
				For $z>w$, define the following pairs of patterns:
				\begin{align}\label{E:seta}
					a_1&=(z>w>v>u,\ z>w>u>v),\\
					a_2&=(u>v\ge z>w,\ v>u\ge z>w),\notag\\
					a_3&=(z>u>v\ge w,\ z>v>u\ge w),\notag\\
					a_4&=(v\ge z>w>u,\ u\ge z>w>v).\notag
				\end{align}
				Choose one pattern from each pair $a_i$ for $1\le i\le 4$. We use $\alpha_i=0$ and $\alpha_i=1$ to denote the first and second choices, respectively. In particular, $\alpha_3=\alpha_4$. 
				Together with $B$, their reflected patterns $(w,z,v,u)$, and $\S^*$, all eight $\quinv$-quadruple sets are numbered as follows
				\[
				\begin{array}{c|cccccccc}
					\S_i&\S_1&\S_2&\S_3&\S_4&\S_5&\S_6&\S_7&\S_8\\\hline
					\alpha_1\,\alpha_2\,\alpha_3&000&010&100&110&001&101&011&111
				\end{array}
				\]
			\end{lemma}
		}
		
		We next associate each $\quinv$-quadruple set $\S$ with a statistic $\eta$ defined on the set of positive fillings.
		\begin{definition}[{\em \color{blue} Family $\A^+$ of statistics}]\label{Def:eta1}
			For a positive filling $\sigma$, consider the quadruple $(z,w,u,v)$ of entries in the augmented super filling $\hat{\sigma}$ (Definition \ref{Def:inv}),
			positioned as shown in the left diagram of Figure \ref{F:f1}, with the columns containing $(z,u)$ and $(w,v)$ having equal height in $\sigma$. 
			
			For such quadruple $(z,w,u,v)$, define $\Q_{\S}(z,w,u,v)=1$ and call it an {\em  $\S$-quadruple} if $(z,w,u,v)$ satisfies one of the total orderings in $\S$; otherwise $\Q_{\S}(z,w,u,v)=0$. This yields a statistic $\eta_{\S}:\T(\lambda)\rightarrow \mathbb{N}$ given by
			\begin{align*}
				\eta_{\S}(\sigma)=\sum_{(z,w,u,v)\in \sigma}\Q_{\S}(z,w,u,v),
			\end{align*}
			which counts the number of $\S$-quadruples in $\sigma$.
			
			An {\em $\S$-triple} $(a,b,c)$ is a queue inversion triple for which the column containing $(a,b)$ is longer than the column containing $c$ in $\sigma$. Define $\Q_{\S}(a,b,c)=1$ for an $\S$-triple and $0$ otherwise. The statistic $\eta$ associated with $\S$ is then
			\begin{align}\label{E:etai}
				\eta(\sigma)=\eta_{\S}(\sigma)+\sum_{(a,b,c)\in\sigma}\Q_{\S}(a,b,c).
			\end{align}
			Finally, let $\A^+$ be the set of all statistics $\eta$ arising from the $\quinv$-quadruple sets $\S_i$ in Lemma \ref{L:setSi}.
		\end{definition}
		
		Now we are in a position to provide an example of Theorem \ref{T:main1}.
		\begin{example}\label{Example:mainT}
			Given the partition $\lambda=(2,2,2)$ and a filling $\sigma=\raisebox{-8pt}{\young(456,123)}$ of the Young diagram $\dg'(\lambda)$, there are six non-attacking and top-row increasing tableaux in $[\sigma]$. We compute their contributions to the coefficient $[x_1 x_2 x_3 x_4 x_5 x_6]P_{\lambda}(X;q,t)$ using (\ref{eqthm3.3}) of Theorem \ref{T:main1}. We take $\S=\S_8$ from Lemma \ref{L:setSi} and define the statistics $\eta$ and $\eta^{\circ}$ according to (\ref{E:etai}) and (\ref{E:eta2}).
			\begin{table}[H]
				\begin{adjustwidth}{-0cm}{0cm}
					\renewcommand\arraystretch{1.8}
					\begin{center}
						\begin{tabular}{c| c c c c c c}%
							$\tau$   & $\young(456,123)$ & $\young(456,213)$  & $\young(456,321)$ & $\young(456,312)$ & $\young(456,132)$ & $\young(456,231)$ \\ \hline			
							$q^{\maj(\tau)} t^{\eta^{\circ}(\tau)}$   & $q^3$        & $q^3 t$      &     $q^3 t^3$       &   $q^3 t^2$         &  $q^3 t$  &  $q^3 t^2$
						\end{tabular}
					\end{center}
				\end{adjustwidth}
			\end{table}
			\vspace{-5mm}
			For each of these fillings, we have $c_{\tau}(q,t)=(1-t)^3(1-qt)^{-1}(1-qt^2)^{-1}(1-qt^3)^{-1}$. Therefore, (\ref{eqthm3.3}) gives the total contribution to $P_{\lambda}(X;q,t)$ as
			\begin{align*}
				q^3 (1+2t+2t^2+t^3) \frac{(1-t)^3x_1 x_2 x_3 x_4 x_5 x_6}{(1-qt)(1-qt^2)(1-qt^3)}.
			\end{align*}
		\end{example}
		

			The flip operator $\delta_i^r$ (Definition \ref{Def3.1}) is paired with the family of statistics $\eta\in \A^+$. We now recall some relevant results from \cite{JL25}.
			{
				\begin{lemma}\cite[Lemma 4.2]{JL25}\label{L:delta}
					Given a partition $\lambda$ with $\lambda'_i=\lambda'_{i+1}=h$ for a fixed $i$, let $\sigma\in\T(\lambda)$. Then
					\begin{align*}
						\maj(\sigma)-\maj(\delta_i^r(\sigma))&=(h-r)(\des(\sigma\vert_{r}^{r+1})-\des(\delta_i^r(\sigma)\vert_{r}^{r+1})),\\
						\eta(\sigma)-\eta(\delta_i^r(\sigma))&=\eta(\sigma\vert_{r}^{r+1})-\eta(\delta_i^r(\sigma)\vert_{r}^{r+1}).
					\end{align*}
					That is, the major index, and the number of $\S$-quadruples and $\S$-triples of $\sigma$ between two consecutive rows that lie entirely above row $r$ or entirely below row $r+1$ are invariant under $\delta_i^r$.
			\end{lemma} }
			
			\begin{theorem}\cite[Theorem 2 and Section 5]{JL25} \label{T:eta1}
				For any statistic $\eta\in \A^+$, there is a bijection $\gamma:\T(\lambda')\rightarrow \T(\lambda')$ such that for all $\sigma\in \T(\lambda')$, we have $\gamma(\sigma)\sim \sigma$,
				\begin{align}
					\label{E:gamma1}(\maj,\eta)\gamma(\sigma)&=(\maj,\quinv)\sigma,
				\end{align}
				where the top rows of $\gamma(\sigma)$ and $\sigma$ are identical and $\des(\sigma\vert_{r}^{r+1})=\des(\gamma(\sigma)\vert_r^{r+1})$ for all $r$.
			\end{theorem}
			By Definition \ref{Def2.2}, for a non-attacking positive filling $\sigma$, the value $\overline{\quadcoinv}(\sigma)$ enumerates
			\begin{itemize}
				\item
				quadruples $(z,w,u,v)$ (as shown on the left of Figure \ref{F:f1}) that lie in two columns of equal height in $\sigma$ such that $z,w,u,v$ are distinct and $(z,u,v)$ is not a queue inversion triple; \vspace{2mm}
				\item triples $(z,u,v)$ (as shown on the right of Figure \ref{F:f1}) that lie in two columns of unequal height for which $(z,u,v)$ is not a queue inversion triple.
			\end{itemize}
			
			Let $w=(w_{\lambda_1},\ldots,w_{1})$ be the sequence of top-row entries $w_j$ of all rectangles $\sigma_j$ in the filling $\sigma=\sigma_{\lambda_1}\sqcup\cdots\sqcup \sigma_1$, read from left to right. The top-row entries $w_j$ contribute $\inv(w_j)$ to $\overline{\quadcoinv}(\sigma)$. Define
			\begin{align*}
				\overline{\quadcoinv}^{\star}(\sigma)=\overline{\quadcoinv}(\sigma)-\sum_{j=1}^{\lambda_1}\inv(w_j).
			\end{align*}
			{
				Recall from (\ref{E:eta2}) that $\eta^{\circ}(\sigma)=\bar{\eta}(\sigma)-\alpha(\sigma)$, where $\alpha(\sigma)$ counts quadruples $(z,w,u,v)$ with repetitions for which $\Q_{\S}(z,w,u,v)=0$. Thus $\eta^{\circ}(\sigma)$ counts $\S$-quadruples $(z,w,u,v)$ with distinct entries such that $\Q_{\S}(z,w,u,v)=0$, together with the non-queue-inversion triples arising from columns of unequal height. Using some properties of the bijection $\gamma:\T(\lambda')\rightarrow \T(\lambda')$ from Theorem \ref{T:eta1}, we construct the bijection $\varphi: N^+(\lambda')\rightarrow N^+(\lambda')$ for Theorem \ref{thm6.1}.
			}

			\begin{theorem}\label{thm6.1}
				For any statistic $\eta\in \A^+$, there is a bijection $\varphi: N^+(\lambda')\rightarrow N^+(\lambda')$ with the following properties for all $\sigma\in N^{+}(\lambda')$,
				\begin{align}
					\varphi(\sigma)&\sim \sigma,\notag\\
					\label{eqthm6.1a}(\maj,\eta^{\circ})(\varphi(\sigma))&=(\maj,\overline{\quadcoinv}^{\star})(\sigma),\\
					\label{eqthm6.1b}  c_{\sigma}(q,t)
					&=c_{\varphi(\sigma)}(q,t),
				\end{align}
				where the top rows of $\varphi(\sigma)$ and $\sigma$ are identical, and the number of (un)restricted boxes in each row is invariant under $\varphi$. Consequently, Theorem \ref{thm3} holds.
			\end{theorem}
			
			\begin{proof}
				Much like in Theorem \ref{thm5.1}, it suffices to prove Theorem \ref{thm6.1} for the case where $\lambda'$ has two parts.
				
				Case I: If $\sigma$ has only one row, then let $\varphi(\sigma)=\sigma$. By definition we have $\maj(\sigma)=\eta^{\circ}(\sigma)=\overline{\quadcoinv}^{\star}(\sigma)=0$ and trivially $ c_{\sigma}(q,t)=c_{\varphi(\sigma)}(q,t)=1$.
				
				Case II (1): If $\sigma$ is a two-row rectangular non-attacking filling with exactly $n$ columns and $\maj(\sigma)=n$.
				By Definition \ref{Def:setS} $(2)$, exactly one of the quadruples $(z>w>u>v)$ and $(z>w>v>u)$ belongs to $\S$.
				\begin{itemize}
					\item If $(z>w>u>v)\in\S$, let $\pi$ be the permutation of minimal length that sorts the top row of $\sigma$ into weakly increasing order;
					\item Otherwise $(z>w>v>u)\in\S$ and let $\pi$ be the permutation of minimal length that sorts its top row into weakly decreasing order.
				\end{itemize}
				Define $\varphi(\sigma)=\delta_{\pi}^{-1}(\xi_{\pi}(\sigma))$. By the construction of the operators $\delta_i^r$ and $\xi_i^r$ in Definition \ref{Def3.1}, the top rows of $\varphi(\sigma)$ and $\sigma$ are the same. {
					Since every quadruple $(z,w,u,v)$ lying entirely in $\xi_{\pi}(\sigma)$ contains four distinct entries,
				}, it follows from Definitions \ref{Def:eta1} and \ref{Def2.2} that $(z,w,u,v)$ is an $\S$-quadruple, if and only if $(z,u,v)$ is a queue inversion triple, if and only if $(z,w,u,v)$ is a quadruple coinversion. Therefore,
				\begin{align*}
					\eta^{\circ}(\xi_{\pi}(\sigma))=\overline{\quadcoinv}^{\star}(\xi_{\pi}(\sigma)).
				\end{align*}
				Together with Lemmas \ref{lem3.2} and \ref{L:delta}, we obtain
				\begin{align}\label{xi}
					(\maj,\overline{\quadcoinv}^{\star})\sigma&= (\maj,\overline{\quadcoinv}^{\star})\xi_{\pi}(\sigma)\notag\\
					&= (\maj,\eta^{\circ})\xi_{\pi}(\sigma) =(\maj,\eta^{\circ})\varphi(\sigma).
				\end{align}
				
				Case II (2): If $\sigma$ is a two-row rectangular non-attacking filling with $\maj(\sigma)=0$. Then we define $\varphi(\sigma)=\delta_{\pi}^{-1}(\xi_{\pi}(\sigma))$, where $\pi$ is chosen by the following rule:
				\begin{itemize}
					\item If $(u>v\ge z>w)\in\S$, let $\pi$ be the permutation of minimal length that sorts the top row of $\sigma$ into weakly increasing order;
					\item Otherwise $(v>u\ge z>w)\in\S$ and let $\pi$ be the permutation of minimal length that sorts its top row into weakly decreasing order.
				\end{itemize}
				The argument to discuss the properties of $\varphi$ is exactly the same as in Case II (1), so we omit the details.
				
				Case III: If $\sigma$ is a two-row rectangular non-attacking filling with $\maj(\sigma)\not\in \{0,n\}$. Then $\varphi(\sigma)$ is defined as follows:
				
				If $(z>u>v\ge w)\in\S$, let
				$\pi$ be the permutation of minimal length that sorts its top row into a weakly increasing. Otherwise, if $(z>v>u\ge w)\in\S$, let $\pi$ be the permutation of minimal length that sorts its top row into a weakly decreasing.
				
				Let $\sigma_{>}$, $\sigma_{\le}$ and $\sigma_{<}$ denote the subfillings formed by its descent columns, non-descent columns, and ascent columns of $\sigma$, respectively.
				By \eqref{xi} of Case II, we have
				\begin{align}
					(\maj,\overline{\quadcoinv}^{\star})\xi_{\pi}(\sigma)_{>}=(\maj,\eta^{\circ})\varphi(\xi_{\pi}(\sigma)_{>}), \label{xi1}\\
					(\maj,\overline{\quadcoinv}^{\star})\xi_{\pi}(\sigma)_{<}=(\maj,\eta^{\circ})\varphi(\xi_{\pi}(\sigma)_{<}),\label{xi2}
				\end{align}
				where the top rows of $\xi_{\pi}(\sigma)_{>}$  and  $\varphi(\xi_{\pi}(\sigma)_{>})$ are identical, as are those of  $\xi_{\pi}(\sigma)_{<}$  and  $\varphi(\xi_{\pi}(\sigma)_{<})$. Let $\tau$ to be the filling determined by letting $\tau_{>}=\varphi(\xi_{\pi}(\sigma)_{>})$, $\tau_{<}=\varphi(\xi_{\pi}(\sigma)_{<})$, and requiring that for each $i$, column $i$ of $\tau$ is a descent (resp. ascent) column if and only if column $i$ of $\xi_{\pi}(\sigma)$ is a descent (resp. ascent) column. {The columns of $\xi_{\pi}(\sigma)$ whose upper and lower entries are equal are copied unchanged into $\tau$. Define $\varphi(\sigma)=\delta_{\pi}^{-1}(\tau)$.
				}.
				
				We now examine the change of statistics. Note that the top rows of $\tau$ and $\xi_{\pi}(\sigma)$ are the same, and $\tau$ is non-attacking with $\maj(\tau)=\maj(\xi_{\pi}(\sigma))$. By Lemmas \ref{lem3.1} and \ref{lem3.2}, the top rows of $\sigma$ and $\varphi(\sigma)$ are the same. Moreover,
				$\maj(\varphi(\sigma))=\maj(\tau)=\maj(\xi_{\pi}(\sigma))=\maj(\sigma)$.
				It remains to prove that
				\begin{align}
					\overline{\quadcoinv}^{\star}(\sigma)=\eta^{\circ}(\varphi(\sigma)).
				\end{align}
				Lemmas \ref{lem3.1} and \ref{lem3.2} guarantee that  $\overline{\quadcoinv}^{\star}(\sigma)=\overline{\quadcoinv}^{\star}(\xi_{\pi}(\sigma))$ and $\eta^{\circ}(\varphi(\sigma))=\eta^{\circ}(\tau)$. Therefore, it boils down to showing that
				\begin{align}\label{E:case41}
					\overline{\quadcoinv}^{\star}(\xi_{\pi}(\sigma))=\eta^{\circ}(\tau).
				\end{align}
				
				Let $\eta^{\circ}(\sigma_{>},\sigma_{<})$ be the number of $\S$-quadruples $(z,w,u,v)$ of distinct entries in $\sigma$ such that $\chi(z>u)\ne\chi(w>v)$ and $\Q_{\S}(z,w,u,v)=0$. Similarly,
				let $\overline{\quadcoinv}^{\star}(\sigma_{>},\sigma_{<})$ be the number of quadruples $(z,w,u,v)$ of distinct entries in $\sigma$ such that $\chi(z>u)\ne\chi(w>v)$ and $\Q(z,w,u,v)=0$. Then by definition,
				\begin{align*}
					\eta^{\circ}(\tau)&=\eta^{\circ}(\tau_{>})+\eta^{\circ}(\tau_{<})+\eta^{\circ}(\tau_{>},\tau_{<}), \\
					\overline{\quadcoinv}^{\star}(\xi_{\pi}(\sigma))&=\overline{\quadcoinv}^{\star}(\xi_{\pi}(\sigma)_{>})+\overline{\quadcoinv}^{\star}(\xi_{\pi}(\sigma)_{<})\\
					&\qquad+\overline{\quadcoinv}^{\star}(\xi_{\pi}(\sigma)_{>},\xi_{\pi}(\sigma)_{<}).
				\end{align*}
				Since $\eta^{\circ}(\tau_{>})=\overline{\quadcoinv}^{\star}(\xi_{\pi}(\sigma)_{>})$ and $\eta^{\circ}(\tau_{<})=\overline{\quadcoinv}^{\star}(\xi_{\pi}(\sigma)_{<})$ by $\eqref{xi1}$ and $\eqref{xi2}$, the proof of (\ref{E:case41}) reduces to showing that
				\begin{align}\label{E:case42}
					\bar{\eta}(\tau_{>},\tau_{<})= \eta^{\circ}(\tau_{>},\tau_{<})=\overline{\quadcoinv}^{\star}(\xi_{\pi}(\sigma)_{>},\xi_{\pi}(\sigma)_{<}).
				\end{align}
				
				When $\sigma$ is a two-row rectangular filling with a weakly increasing (resp. decreasing) top row, then, provided that $(z>u>v\ge w)\in\S$  (resp. $(z>v>u\ge w)\in\S$),
				the number of $\S$-quadruples equals the number of queue inversion triples induced between descent and non-descent columns \cite[Equation (5.11)]{JL25}. That is,
				\begin{align*}
					\quinv(\sigma_{>},\sigma_{\le})=\eta(\sigma_{>},\sigma_{\le}).
				\end{align*}
				{Consequently, we have
					\begin{align*}
						\overline{\quadcoinv}^{\star}(\xi_{\pi}(\sigma)_{>},\xi_{\pi}(\sigma)_{<})&=	\overline{\quinv}(\xi_{\pi}(\sigma)_{>},\xi_{\pi}(\sigma)_{<})=\overline{\eta}(\xi_{\pi}(\sigma)_{>},\xi_{\pi}(\sigma)_{<}),\\
						\overline{\quadcoinv}^{\star}(\tau_{>},\tau_{<})&= \overline{\quinv}(\tau_{>},\tau_{<})=\overline{\eta}(\tau_{>},\tau_{<})
					\end{align*}
					and therefore (\ref{E:case42}) is equivalent to
					\begin{align}\label{E:eta_phi1}
						\overline{\quinv}(\xi_{\pi}(\sigma)_{>},\xi_{\pi}(\sigma)_{<})=\overline{\quinv}(\tau_{>},\tau_{<}),
					\end{align}
					which is proved in \cite[Equation (5.12)]{JL25}.}

				Case IV: If  $\sigma=\sigma_2\sqcup\sigma_1$ is a two-row non-rectangular non-attacking filling, then we define $\varphi(\sigma)=\varphi(\sigma_2)\sqcup\sigma_1$. It is clear that the top rows of $\sigma$ and $\varphi(\sigma)$ are identical, and
				\begin{align*}
					(\maj,\overline{\quadcoinv}^{\star})\sigma_2=(\maj,\eta^{\circ})\varphi(\sigma_2)
				\end{align*}
				according to Cases II and III. Since $\varphi(\sigma_2)$ is generated from $\sigma_2$ by successively swapping  two entries $u,v$ in the bottom row, such that the columns pairs $(z,u),(w,v)$ and $(z,v),(w,u)$ are all either descents or ascents. That is, $\min(z,w)>\max(u,v)$ or $\max(z,w)<\min(u,v)$. It follows that $\CMcal{Q}(z,u,v)=\CMcal{Q}(w,u,v)$, thus implying that the number of non-queue inversion triples between $\sigma_2$ and $\sigma_1$, say $a$, is invariant. Consequently,
				\begin{align*}
					(\maj,\overline{\quadcoinv}^{\star})\sigma&=(\maj,\overline{\quadcoinv}^{\star})\sigma_2+(\maj,\overline{\quadcoinv}^{\star})\sigma_1+a \\
					&=(\maj,\eta^{\circ})\varphi(\sigma_2)+(\maj,\eta^{\circ})\sigma_1+a\\
					&=(\maj,\eta^{\circ})\varphi(\sigma).
				\end{align*}
				Therefore,  the bijection $\varphi$ satisfies  \eqref{eqthm6.1a}.
				
				For Case II to IV, since $\varphi$ (as a composition of the operators $\xi_{i}$ and $\delta_i$) preserves the number of (un)restricted boxes by Lemma \ref{lem3.2} and Definition \ref{Def3.1},  (\ref{eqthm6.1b}) follows from Remark \ref{R:1}. This completes the proof.
			\end{proof}
			Finally, let us introduce sorted tableaux to establish (\ref{E:main_comp}).
			\begin{definition}[{\em \color{blue} Sorted tableaux}]\cite[Definition 6]{JL25}\label{Def:sorted}
				Let $\sigma=\sigma_{\lambda_1} \sqcup \cdots \sqcup \sigma_1$ be a non-attacking tableau, where each $\sigma_j$ is a rectangular filling of height $j$. Then $\sigma$ is called sorted if each $\sigma_j$ is sorted in the sense defined below.
				
				For a rectangular partition $\lambda=(n^m)$ and any tableau $\sigma\in \dg'(\lambda)$, we say that
				$\sigma$ is sorted if it is a top-row increasing tableau satisfying the following condition:
				\begin{itemize}
					\item For any $1\le r<n$ and for any two descent columns $\young(a,c)$ and $\young(b,d)$ in  $(\sigma\vert_{r}^{r+1})_{>}$ with $a$ to the left of $b$, we require that $a>b$ implies $c>d$, and $a<b$ implies $c<d$. In other words, if the upper entries $a,b$ are arranged in decreasing order from left to right, then the corresponding lower entries $c,d$ must also be decreasing.
				\end{itemize}
				For example, the tableau $\sigma$ below is a sorted tableau of shape $\lambda=(3^6)$.
				\begin{align*}
					\sigma=\begin{ytableau}[] 3 &5& 7& 8& 10& 13\\  2 &5& 6& 11& 9& 12 \\ 2 &1 & 8& 4& 3& 7
					\end{ytableau}
				\end{align*}
				
			\end{definition}
			{\em Proofs of Theorem \ref{T:main1} and Theorem \ref{cor:1}}. Since (\ref{eqthm3.3}) follows from (\ref{eqthm2.2}) and Theorem \ref{thm3}, it suffices to prove (\ref{E:main_comp}). 
			The coefficient $d_{\sigma}(t)$ for a sorted non-attacking tableau $\sigma$ is defined as follows. Let $\sigma=\sigma_{\lambda_1} \sqcup \cdots \sqcup \sigma_1$, where $\sigma_j$ is a rectangular filling of height $j$ and $n$ is the largest entry of $\sigma$. We then define a sequence $\nu_{i,j}=(\nu_{i,j}^{1}\le \cdots\le \nu_{i,j}^{n})$ of non-negative integers by
			\begin{align}
				\nu_{i,j}^{1}&=\# \, 1's \textrm{ in row } i \textrm{ of } \sigma_j,\label{E:v1}\\
				\nu_{i,j}^{k}-\nu_{i,j}^{k-1}&=\# \, k's \textrm{ in row } i \textrm{ of } \sigma_j,\label{E:v2}
			\end{align}
			for any $1<k\le n$. Next, define the sequence $s=\{s_h^k\}_{0\leq k<n,\,k\leq h\le n}$ from  $\sigma_j\vert_{i}^{i+1}$ as follows:
			\begin{align*}
				s_{k}^{k}&=\#\,\textrm{ descent columns } (a,b) \textrm{ in } \sigma_j\vert_i^{i+1}
				\textrm{ such that } k+1=a>b,\\
				s_{h}^k-s_{h-1}^k&=\#\,\textrm{ non-descent columns } (a,b) \textrm{ in } \sigma_j\vert_i^{i+1}
				\textrm{ such that } k+1=a\le b=h.
			\end{align*}
			Consequently, $s_{n}^k=\nu_{i+1,j}^{k+1}-\nu_{i+1,j}^{k}$ counts the occurrences of $k+1$ in $\sigma_j\vert_{i+1}^{i+1}$. Finally, define
			\begin{align}\label{E:dsigma}
				d_{\sigma}(t)&=d_{\sigma_{\lambda_1}}(t)\cdots d_{\sigma_1}(t),\,\,\mbox{ where }\,\\
				d_{\sigma_j}(t)&=\prod_{i<j}\prod_{0\le k<n}{\nu_{i,j}^{k+1}-s_{k+1}^{0,k}\brack \nu_{i,j}^{k}-s_{k}^{0,k}}_t \notag
			\end{align}
			and we set $d_{\sigma_j}(t)=1$ whenever $j$ is not a part of $\lambda$ (i.e., $\sigma_j=\varnothing$).
			We claim that
			\begin{align}\label{E:compact1}
				\sum_{\tau \textrm{ non-attacking}\atop \textrm{top-row increasing}}q^{\maj(\tau)}t^{{\eta}^{\circ}(\tau)} x^{\tau}
				=\sum_{\sigma\,\textrm{sorted}\atop \textrm{non-attacking}}d_{\sigma}(t)\,q^{\maj(\sigma)}t^{{\eta}^{\circ}(\sigma)}x^{\sigma}.
			\end{align}
			Following Definition \ref{Def:eta1}, let $\overline{\S}$ denote the set obtained from $\S$ by replacing each quadruple $(z,w,u,v)$ with $(z,w,v,u)$.
			If $\sigma$ is a rectangular filling, then $\eta_{{S}}(\sigma)=\eta(\sigma)$, thus $\eta_{\overline{\S}}(\sigma)=\overline{\eta}(\sigma)$.

			{If $(z>w>u>v)\in{\S}$, then $(z>w>v>u)\in\overline{\S}$. The compact formula for rectangular non-attacking and top-row increasing tableaux $\tau$  in \cite[Equation (7.4)]{JL25} yields
				\begin{align*}
					\sum_{\tau \textrm{ non-attacking}\atop \textrm{top-row increasing}}q^{\maj(\tau)}t^{\overline{\eta}(\tau)}
					=\sum_{\sigma\,\textrm{sorted}\atop \textrm{non-attacking}}d_{\sigma}(t)\,q^{\maj(\sigma)}t^{\overline{\eta}(\sigma)}.
				\end{align*}
				Since the number of restricted boxes is invariant under $\delta_i^r$, Remark \ref{R:1} gives $c_{\sigma}(q,t)=c_{\tau}(q,t)$. Moreover, by (1) of  Definition \ref{Def:setS}, we have $\alpha(\sigma)=\alpha(\tau)$. Consequently, (\ref{E:compact1}) holds for rectangular fillings $\sigma$.
				Observing that the flip operator $\delta_i^r$ preserves the number of queue inversion triples located in two columns of unequal height, we see that (\ref{E:compact1}) extends to non-rectangular fillings $\sigma$ as well. Combining this with $c_{\sigma}(q,t)=c_{\tau}(q,t)$ and Remark \ref{R:1}, we conclude the proof of (\ref{E:main_comp}). }
			
			Following the line of argument developed in \cite[Section 8]{JL25}, we can translate (\ref{E:main_comp}) into an explicit $(q,t)$-formula (\ref{E:vsP}) for the coefficient $c_{\lambda\mu}(q,t)=[m_{\mu}]P_{\lambda}$. Here we only define the relevant notations in (\ref{E:vsP}). Let
			\begin{align*}
				(a;q)_{n}=\prod_{k=1}^{n}(1-aq^{k-1}), \quad (a;q)_\infty=\prod_{k=1}^\infty (1-aq^{k-1}).
			\end{align*}
			{For nondecreasing sequences of nonnegative integers $\nu=(\nu^1\le\cdots\le\nu^n)$ and $\tilde\nu=(\tilde\nu^1\le\cdots\le\tilde\nu^n)$ such that $\nu^n=\tilde\nu^n$, let $\nu^0=\tilde\nu^0=0$ and \({A\brack B}_{b^{-1}}=0\) unless $0\le B\le A$.
			} We define
			\begin{align}\label{E:phi}
				\phi_{\nu| \widetilde{\nu}}(a,b,c)=\sum_{s}b^{\xi(s,\widetilde{\nu},\nu)} &\prod_{0\le k<n} a^{ s_{n}^{k}-s_{k}^{k} }
				{\widetilde{\nu}^{k+1}-s_{k+1}^{0,k} \brack \widetilde{\nu}^{k}-s_{k}^{0,k} }_{b^{-1}}
				b^{( \widetilde{\nu}^{k+1}-s_{k+1}^{0,k} -\widetilde{\nu}^{k}+s_{k}^{0,k}) (\widetilde{\nu}^{k}-s_{k}^{0,k})}    \notag\\
				&\qquad\times\frac{(1-b^{-1})^{c-\sum_{h}(s_{h+1}^h-s_h^h)}}{(a^{-1}b^{-1};b^{-1})_{c-\sum_{h}(s_{h+1}^h-s_h^h)}},
			\end{align}
			which is summed over all sequences $\{s_k^i \}_{0\le i<n,i\leq k\leq n}$ such that
			\begin{align}\label{E:defs2}
				0=s_0^0\leq s_{k}^{k}\leq s_{k+1}^{k}\leq \cdots \leq s_{n-1}^{k}\leq s_{n}^{k}=\nu^{k+1}-\nu^{k},
			\end{align}
			{In addition, the summation variables satisfy
				\begin{align}\label{eq9}
					s_{k+1}^{k}-s_k^{k}=1 \quad\text{whenever}\quad \nu^{k+1}-\nu^k=\widetilde{\nu}^{k+1}-\widetilde{\nu}^k=1, \qquad 0\le k<n.
				\end{align}
				The sum is zero if no admissible array $s$ exists.
				Define $s_{i}^{j,k}=\sum_{a=j}^{k}s_{i}^{a}$, and }
			\begin{align*}
				\xi(s,\widetilde{\nu},\nu)&=\sum_{0\leq h\leq k<n} (s_{k+1}^{h}-s_{k}^{h})  (\widetilde{\nu}^{k}-s_{k}^{0,k}+s_h^{0,h-1}+\nu^h-s_{k+1}^{0,h-1})\\
				&\qquad +\sum_{0\leq h<k<n} (s_{h+1}^{h}-s_{h}^{h})  (\widetilde{\nu}^{k+1}-s_{k+1}^{0,k}-\widetilde{\nu}^{k}+s_k^{0,k}).
			\end{align*}
			The sum in (\ref{E:vsP}) is taken over sequences $\nu_{i,j}=(\nu_{i,j}^{1} \le \cdots \le \nu_{i,j}^{n})$ of nonnegative integers, $1\le i\le j\le \lambda_1$, {with $\nu_{i,j}^{0}=0$}, satisfying
			\begin{align}
				\label{eq5}\nu_{i,j}^{n}=\lambda_{j}'-\lambda_{j+1}', &\qquad \textrm{ for all }\,\,1\le i\le j\le \lambda_1,   \\
				\label{eq6}\sum_{1\le i\le j\le \lambda_1} \nu_{i,j}^{k}=\mu_1+\cdots+\mu_k, &\qquad \textrm{ for all }\,\, 1\le k\le n,\\
				\label{eq7}\sum_{i\le j\le \lambda_1}\nu_{i,j}^{k+1}-\nu_{i,j}^k=0 \ \text{or} \ 1, &\qquad \textrm{ for all }\,\, 0\le k\le n-1, 1\le i\le \lambda_1,\\
				\label{eq8}   (\nu_{i+1,j}^{k+1}-\nu_{i+1,j}^k)+(\nu_{i,\ell}^{k+1}-\nu_{i,\ell}^k)<2, &\qquad \textrm{ for all }\,\, j>\ell,
			\end{align}
			{For each factor $\phi_{\nu_{i+1,j}\mid\nu_{i,j}}$, condition \eqref{eq9} is imposed on its internal summation
				variables, for $1\le i<j\le\lambda_1$ and $0\le k<n$.  \eqref{eq9}, \eqref{eq7} and \eqref{eq8}  describe the non-attacking condition. }
			\qed
			

			\begin{example}\label{Example:main2}
				We continue with Example \ref{Example:mainT}. Since there is only one sorted tableau $\sigma$ in this example, we have $d_{\sigma}(t)={2\brack 1}_t {3\brack 2}_t$. The compact formula (\ref{E:main_comp}) then shows that the contribution of the six non-attacking and top-row increasing tableaux in $[\sigma]$ is
				\begin{align*}
					q^3 {2\brack 1}_t {3\brack 2}_t\frac{(1-t)^3x_1 x_2 x_3 x_4 x_5 x_6}{(1-qt)(1-qt^2)(1-qt^3)}.
				\end{align*}
				In other words, the polynomial $1+2t+2t^2+t^3$ is captured by the multiplicity $d_{\sigma}(t)$.
			\end{example}

			\begin{example}
				Take $n=6$ and $\lambda=(2,2,2)$, then $n(\lambda)=6$, $n(\lambda')=3$ and $\sum_j\binom{\lambda_j'-\lambda_{j+1}'}{2}=3$. The sum in \eqref{E:vsP} involves only two sequences: $\nu_{1,2}$ and $\nu_{2,2}$. Continuing with $\sigma$ in Example \ref{Example:mainT},  we define $\nu_{1,2}=(1,2,3,3,3,3)$ and $\nu_{2,2}=(0,0,0,1,2,3)$ by (\ref{E:v1}) and (\ref{E:v2}).
				
				This sequence $\{\nu\}$ satisfies \eqref{eq7}--\eqref{eq8} and by \eqref{eq6} we have $x^{\mu}=x_1 x_2 x_3 x_4 x_5 x_6$.
				Setting $i=1,j=2$,  then $a=q^{-(j-i)}t^{-(\lambda_i'-\lambda_j')}=q^{-1}, b=t^{-1}, c=\lambda_j'-\lambda_{j+1}'=3$ in (\ref{E:phi}) and we obtain
				\begin{align}\label{vsp1}
					\phi_{\nu_{2,2}| \nu_{1,2}}(q^{-1},t^{-1},3)
					&=\sum_{s}t^{-\xi(s,\nu_{1,2},\nu_{2,2})} \prod_{0\le k<6} q^{ -(s_{6}^{k}-s_{k}^{k}) }
					t^{-(\nu_{1,2}^{k+1}-s_{k+1}^{0,k} - \nu_{1,2}^{k}+s_{k}^{0,k}) (\nu_{1,2}^{k}-s_{k}^{0,k})} \notag\\
					&\qquad\times { \nu_{1,2}^{k+1}-s_{k+1}^{0,k} \brack \nu_{1,2}^{k}-s_{k}^{0,k} }_{t} \frac{(1-t)^{m}}{(qt;t)_{m} },
				\end{align}	
				where $m=3-\sum_{h}(s_{h+1}^h-s_h^h)$. By (\ref{E:defs2}), (\ref{eq9}) and (\ref{vsp1}), only one sequence $\{s_k^i \}_{0\le i<6,i\leq k\leq 6}$ makes a non-zero contribution to the sum in $\phi$; that is,
				\[\left[
				\begin{array}{ccccccc}
					s_0^0 & s_1^0 & s_2^0  &s_3^0 &  s_4^0 & s_5^0 &  s_6^0\\
					& s_1^1 & s_2^1  &s_3^1 &  s_4^1 & s_5^1 &  s_6^1\\
					&       & s_2^2  &s_3^2 &  s_4^2 & s_5^2 &  s_6^2\\
					&       &        &s_3^3 &  s_4^3 & s_5^3 &  s_6^3\\
					&       &        &      &  s_4^4 & s_5^4 &  s_6^4\\
					&       &        &      &        & s_5^5 &  s_6^5\\
				\end{array} \right]
				=
				\left[
				\begin{array}{ccccccc}
					0    &   0& 0  &0 &  0 & 0 &  0\\
					&   0 & 0  &0 &  0 & 0 &  0\\
					&       & 0  & 0 &  0 & 0 &  0\\
					&       &        & 1 &  1 & 1 &  1\\
					&       &        &      &  1 & 1 &  1\\
					&       &        &      &        & 1 &  1\\
				\end{array} \right]. \]
			It follows that the contribution of $\nu_{1,2}$ and $\nu_{2,2}$ to the sum in $c_{\lambda\mu}(q,t)$ is
			\begin{align*}
				&x_1 x_2 x_3 x_4 x_5 x_6  q^{n(\lambda')}t^{n(\lambda)-\chi(\nu)-\sum_j\binom{\lambda_j'-\lambda_{j+1}'}{2}}\phi_{\nu_{2,2}| \nu_{1,2}}(q^{-1},t^{-1},3) \\
				&= q^3 { 2 \brack 1 }_{t} { 3 \brack 2 }_{t} \frac{(1-t)^{3}x_1 x_2 x_3 x_4 x_5 x_6}{(1-qt)(1-qt^2)(1-qt^3)},
			\end{align*}
			which agrees with  Example \ref{Example:main2}.
		\end{example}

		\section{Inversion analogue of Theorem \ref{T:main1}}\label{S:remark}
		
		This section presents an analogue of Theorem \ref{T:main1} in Theorem \ref{T:Ptin} via two superization formulas in Proposition \ref{propquinvJ}.

		Recall that Theorem \ref{T:eta} is established for sixteen statistics in the set $\A$. Other than the eight statistics in $\A^+$ (see Lemma \ref{L:setSi} and Definition \ref{Def:eta1}), the remaining eight statistics in $\A$ are to be defined. Here we give new combinatorial formulas of $P_{\lambda}(X;q,t)$ with respect to the eight statistics $\A-\A^+$.
		
		\begin{definition}[{\em \color{blue} Family $\A$ of statistics}]\cite[Definition 4]{JL25}\label{Def:eta3}
			For any filling $\sigma=\sigma_{\lambda_1} \sqcup \cdots \sqcup \sigma_1\in\T(\lambda')$, let $N$ denote the largest entry of $\sigma$. For any $1\le j\le \lambda_1$, let $\sigma_j'$ be obtained from $\sigma_j$ by replacing every entry $x$ with $x^c=N+1-x$, and then flipping it vertically. Define $\sigma'=\sigma_{\lambda_1}' \sqcup \cdots \sqcup \sigma_1'$ and
			\begin{align}\label{E:eta*}
				\eta^*(\sigma')=\eta_{\S}(\sigma)+\sum_{(a,b,c)\in\sigma'}\Q^*_{\S}(a,b,c),
			\end{align}
			where $\S$ is a $\quinv$-quadruple set from Lemma \ref{L:setSi}, and $\Q^*_{\S}(a,b,c)=1$ if $(a,b,c)$ is an inversion triple of $\sigma'$ in which the column containing $(a,b)$ is longer than the column containing $c$; otherwise $\Q^*_{\S}(a,b,c)=0$. Finally, define
			\begin{align*}
				\A=\A^+\,\dot\cup\,\{\eta^*:\eta\in\A^+\}.
			\end{align*}
		\end{definition}
		Let $\eta^{\bullet}(\sigma)=\overline{\eta^{*}}(\sigma)-\alpha(\sigma')$. From (\ref{E:eta*}), (\ref{E:etai}) and (\ref{E:eta2}), we obtain
		\begin{align}\label{E:eta_bullet}
			\eta^{\bullet}(\sigma')&=\eta^{\circ}(\sigma)+\sum_{(a,b,c)\in\sigma'}(1-\Q^*_{\S}(a,b,c))-\sum_{(a,b,c)\in\sigma}(1-\Q_{\S}(a,b,c)).
		\end{align}
		For all non-attacking rectangular tableaux $\sigma$, it follows from (\ref{E:eta_bullet}) that $\eta^{\bullet}(\sigma')=\eta^{\circ}(\sigma)$.
		
		Similarly, the statistic $\quadcoinv$ on super fillings (see Definition \ref{Def2.2}) has an inversion-counterpart $\quadinv$; the corresponding non-attacking tableaux are defined as follows.
		
		\begin{definition}[{\em \color{blue} Statistic $\quadinv$}]\label{Def2.2a}
			For a super filling $\sigma$, consider any quadruple $(z,w,u,v)$ of entries in $\hat{\sigma}$, located as shown in Figure \ref{F:f1}. Assume the columns containing $(z,u)$ and $(w,v)$ are of equal height in $\sigma$. {
				\begin{enumerate}
					\item If  $|z|=|u|$, $|w|=|v|$, and $|z|\ne|w|$, then $(z,w,u,v)$(resp. $(w,z,u,v)$) is a quadruple inversion if and only if $\max\{z,w\}\in \Z_+$(resp. $\max\{z,w\}\in\Z_-$).
					\item If  $|z|=|u|$, and $|z|,|w|,|v|$ are distinct, then $(z,w,u,v)$(resp. $(w,z,u,v)$) is a quadruple inversion if and only if $z\in\Z_+$(resp. $z\in\Z_-$).
					\item If neither  (1) nor  (2) holds,  then $(z,w,u,v)$ is a quadruple inversion  if and only if  $\Q(w,v,z)=1$.
			\end{enumerate} }
			Otherwise, the column containing $(z,u)$ is longer than the column containing $(w,v)$. In this case, $(z,w,u,v)$ is a quadruple inversion if and only if $(z,u,w)$ is an inversion triple. Equivalently, this holds if and only if exactly one of the following conditions is satisfied:
			$$\{I(z,u)=1,  I(w,u)=0,   I(z,w)=0  \}.$$
			Define $\Q(z,w,u,v)=1$ if $(z,w,u,v)$ is a quadruple inversion; and $\Q(z,w,u,v)=0$ otherwise.
			Let $\quadinv(\sigma)$ be the number of quadruple inversions in $\sigma$.
		\end{definition}
		
		For all positive non-attacking rectangular tableaux $\sigma$, the triple $(z,u,w)$ is an inversion triple in $\sigma$, if and only if $(u^c,z^c,w^c)$ is a queue inversion triple in $\sigma'$. Hence,  $\overline{\quadinv}(\rev(\sigma))=\overline{\quadcoinv}(\sigma')$.
		
		\begin{definition}[{\em \color{blue} Dual non-attacking fillings}]
			Given the Young diagram $\dg'(\lambda)$, two boxes $u,v\in \dg'(\lambda)$ are said to be {dual attacking} each other if either
			\begin{enumerate}
				\item[$\mathrm{(I)}$]  $u=(i,j)$, $v=(i,k)$, where $j\ne k$; or
				\item[$\mathrm{(II)}$]  $u=(i,j)$, $v=(i-1,k)$, where $k<j$; or
				\item[$\mathrm{(III)}$] $u=(i,j)$, $v=(i-1,k)$ where $k>j$ such that $\lambda_j=\lambda_{k}$.
			\end{enumerate}
			A {\color{blue} dual non-attacking filling} $\sigma$ is a filling that contains no two dual attacking boxes $u,v$ satisfying $|\sigma(u)|=|\sigma(v)|$. When the diagram $\dg'(\lambda)$ is a rectangle, dual-non-attacking fillings coincide with non-attacking fillings.
		\end{definition}
		
		Following the line of argument developed in Sections \ref{S:roadmap} -- \ref{S:7}, we can prove the corresponding formulas for the statistics in $\A - \A^+$.
		\begin{theorem}\label{T:Ptin}
			For a partition $\lambda$, let $\vartheta\in \{\overline{\quadinv},\eta^{\bullet}\}$. Then,
			\begin{align} \label{eqT:Ptin}
				P_{\lambda}(X;q,t)&=\frac{\prod_{i\ge 1}[m_i]_t!}{b_{\lambda}(q,t)}
				\sum_{ \tau\, \textrm{ dual non-attacking} \atop \textrm{and bottom-row  increasing} } q^{\maj(\tau)} t^{\vartheta(\tau)} x^{\tau}  \notag\\
				&\qquad \times\prod_{\tau(u)\ne \tau(\South(u))\atop \text{ or}\, u \in\, \text{row }1}(1-t) \prod_{\tau(u)=\tau(\South(u)) \atop \text{ and}\,u \not\in\, \text{row }1} (1-q^{\leg(u)+1} t^{1+\arm'(u)} ),
			\end{align}
			where $b_{\lambda}(q,t)$ is defined in (\ref{E:bla2}) and for $u=(i,j)$,
			\begin{align*}
				\arm'(u)&=|\{(i,k)\in \dg'(\lambda): k<j  \textrm{ and } \lambda_{k}=\lambda_j   \}| \\
				&\qquad + |\{(i,k)\in\dg'(\lambda): k>j, \, \lambda_{j}=\lambda_k \textrm{ and } (i,k)\  \textrm{is unrestricted} \textrm{ or } \lambda_{j}>\lambda_k   \}|.
			\end{align*}
		\end{theorem}
		\begin{corollary}[A compact version of (\ref{eqT:Ptin})]
			For a partition $\lambda$, {and $\eta\in\A^+$ such that $(z>w>u>v)\in\S$},
			\begin{align*}
				\sum_{ \tau \ \textrm{dual non-attacking}  \atop \textrm{and bottom-row increasing}} q^{\maj(\tau)} t^{\eta^{\bullet}(\tau)}x^{\tau}=\sum_{ \sigma\, \textrm{dual non-attacking} \atop\rev(\sigma)'\textrm {sorted non-attacking}  }  d_{\sigma'}(t)q^{\maj(\sigma)} t^{\eta^{\bullet}(\sigma)}x^{\sigma},
			\end{align*}
			where $\sigma'$ is as defined in Definition \ref{Def:eta3}.
			Consequently,  	
			\begin{align*}
				P_{\lambda}(X;q,t)&=\frac{\prod_{i\ge 1}[m_i]_t!}{b_{\lambda}(q,t)}
				\sum_{\sigma\, \textrm{dual non-attacking} \atop\rev(\sigma)'\textrm {sorted non-attacking} }  d_{\sigma'}(t)q^{\maj(\sigma)} t^{\eta^{\bullet}(\sigma)} x^{\sigma}  \notag\\
				&\qquad \times\prod_{\sigma(u)\ne \sigma(\South(u))\atop \text{ or}\, u \in\, \text{row }1}(1-t) \prod_{\sigma(u)=\sigma(\South(u)) \atop \text{ and}\,u \not\in\, \text{row }1} (1-q^{\leg(u)+1} t^{1+\arm'(u)} ).
			\end{align*}
			
		\end{corollary}
		Unlike (\ref{eqthm2.2}), the second product in (\ref{eqT:Ptin}) can not be cancelled out by the factor $b_{\lambda}(q,t)$. This also happens in the combinatorial formula of $P_{\lambda}(X;q,t)$ using the inversion statistic by Haglund, Haiman and Loehr \cite{HHL04}.
		\begin{theorem}\cite[Proposition 8.1]{HHL04}
			For a partition $\lambda$,
			\begin{align*}
				P_{\lambda}(X;q,t)&=\frac{1}{b_{\lambda}(q,t)}
				\sum_{\tau \ \textrm{$\inv$-non-attacking} } q^{\maj(\tau)} t^{\overline{\inv}(\tau)} x^{\tau}  \notag\\
				&\qquad \times\prod_{\tau(u)\ne \tau(\South(u))\atop \text{ or}\, u \in\, \text{row }1}(1-t) \prod_{\tau(u)=\tau(\South(u)) \atop \text{ and}\,u \not\in\, \text{row }1} (1-q^{\leg(u)+1} t^{1+\arm(u)}),
			\end{align*}
			where $b_{\lambda}(q,t)$ is defined in (\ref{E:bla2}).
		\end{theorem}
		
		\section{Proof of Theorem \ref{thmE}}\label{S:10}
		This section gives a compressed combinatorial formula for non-symmetric Macdonald polynomials $E_{\gamma}(X;q,t)$. Throughout this section, all weak compositions $\gamma$ are assumed to have length $n$ and $X=(x_1,\ldots,x_n)$.
		
		For a weak composition $\gamma=(\gamma_1,\ldots,\gamma_n)$, let $\dg'(\gamma)$ denote the array of boxes with $\gamma_i$ boxes in the $i$th column from left to right, with the bottom row justified.  Let $\widehat{\dg}'(\gamma)$  be the composition diagram obtained by  adjoining one box to the bottom of each column; that is, there are exactly $n$ boxes in row $0$.
		For a filling $\sigma$ of the weak composition diagram $\dg'(\gamma)$, we define $\widehat{\sigma}$ to be the function  $\widehat{\sigma}:\widehat{\dg}'(\gamma)\rightarrow[n]$  such that $\widehat{\sigma}$ agrees with $\sigma$ on $\dg'(\gamma)$, and $\widehat{\sigma}((0,j))=j$ for $1\le j\le n$.  That is, row $0$ of $\widehat{\sigma}$ is filled with entries $1$ to $n$ from left to right.
		\begin{definition}[{\em \color{blue} Statistic $\coinv'$}]\label{Def:coinv}
			For a filling $\widehat{\sigma}$ of $\widehat{\dg}'(\gamma)$, we consider the triples $(a,b,c)$ of entries in the filling $\widehat{\sigma}$.
			
			A {type $A$ triple} of $\sigma$ is a triple $(a,b,c)=\begin{ytableau}
				c &\none[\dots] & a
				\\
				\none[] & \none[] &   b
			\end{ytableau}$
			of entries in $\widehat{\sigma}$ such that
			\begin{enumerate}
				\item the column-length of $c$ is less than or equal to the column-length of $a$.
			\end{enumerate}
			A {\em type $B$  triple} of $\sigma$ is a triple $(a,b,c)=\begin{ytableau}
				a  &\none[]
				\\
				b &\none[\dots] & c\\
			\end{ytableau}$ of entries in $\widehat{\sigma}$ satisfying
			\begin{enumerate}
				\setcounter{enumi}{1}
				\item the column-length of $b$ is larger than the column-length of $c$.
			\end{enumerate}
			Let $\coinv'(\widehat{\sigma})$ count the number of type $A$ and $B$ triples in $\widehat{\sigma}$ such that $\CMcal{Q}(a,b,c)=1$.
		\end{definition}
		The $\quinv$-attacking condition for the weak composition diagram $\dg'(\gamma)$ is identical to the one for partition diagram $\dg'(\lambda)$ in Definition \ref{Def:non-attacking}, while the attacking condition in $\dg'(\gamma)$ is modified by replacing $\mathrm{(III)}$ by the following $\mathrm{(III')}$:
		\begin{enumerate}
			\item[$\mathrm{(III')}$] $u=(i,j)$, $v=(i-1,k)$ where $k<j$ such that $\gamma_m=\gamma_k$ for all $k\le m\le j$.
		\end{enumerate}
		When $\gamma$ is a partition, condition $\mathrm{(III')}$ reduces to condition $\mathrm{(III)}$ in Definition  \ref{Def:non-attacking}. We are now ready to state the first combinatorial formula for the non-symmetric  Macdonald polynomial $E_{\gamma}(X;q,t)$ by Haglund, Haiman and Loehr \cite{HHL08}.
		\begin{theorem}\cite[Theorem 3.5.1]{HHL08}
			The non-symmetric Macdonald polynomials $E_{\gamma}(X;q,t)$ are given by
			\begin{align}\label{E:HHL08}
				E_{\gamma}(X;q,t)=\sum_{\sigma:\gamma\rightarrow [n] \atop \widehat{\sigma}\, \quinv  \text{-non-attacking}}x^{\sigma} q^{\maj(\widehat{\sigma})}  t^{\overline{\coinv'}(\widehat{\sigma})}  \prod_{u\in\dg'(\gamma)\atop \widehat{\sigma}(u)\ne\widehat{\sigma}(\South(u))}  \frac{1-t}{1-q^{\leg(u)+1} t^{\widetilde{\arm}(u)+1}},
			\end{align}
			where $\leg(u),\widetilde{\arm}(u)$ {are defined as in Section~\ref{S:2}, with $\lambda$ replaced
				by $\gamma$ and $\dg'(\lambda)$ by $\widehat{\dg}'(\gamma)$.}
		\end{theorem}
		Before we prove the compact formula (Theorem \ref{thmE}) for $E_{\gamma}(X;q,t)$, we introduce some notations and definitions for the weak composition diagrams. First, the weak composition diagram can also be regarded as a concatenation of maximal rectangles from left to right. That is, two columns $i$ and $j$ such that $i<j$ belong to same rectangle if $\gamma_i=\gamma_j$ and $\gamma_k=\gamma_i$ for all $i\le k\le j$. We write
		\begin{align}\label{E:dec_wide}
			\widehat{\sigma}=\widehat{\sigma}_{1} \sqcup \cdots \sqcup \widehat{\sigma}_{k},
		\end{align}
		where $\widehat{\sigma}_i$ is the $i$th rectangle of $\widehat{\sigma}$ from left to right.
		
		\begin{example}
			For $n=6$ and $\gamma=(2,2,1,2,0,2)$, the filling $\widehat{\sigma}$ below is $\quinv$-non-attacking, where the bottom row is row $0$.
			\begin{align*}
				\widehat{ \sigma}=\begin{ytableau}
					*(green)6	 & *(green)1 &\none & *(green)2  & \none  & *(green)4 \\
					*(green)1	 & *(green)2 & 3  &  *(green)4  & \none  &*(green) 5\\
					*(green)1	 & *(green)2 & 3  &  *(green)4  & 5  & *(green)6\\
				\end{ytableau}
			\end{align*}
			We have $\widehat{\sigma}=\widehat{\sigma}_{1} \sqcup \widehat{\sigma}_{2}  \sqcup \widehat{\sigma}_{3} \sqcup \widehat{\sigma}_{4}\sqcup \widehat{\sigma}_{5}$, where $\widehat{\sigma}_1$, $\widehat{\sigma}_3$ and $\widehat{\sigma}_5$ are the rectangles of height $2$ (highlighted in green).
		\end{example}
		Second, we generalize the statistics $\quadinv$ and $\eta$ to the weak composition diagram $\widehat{\dg}'(\gamma)$.
		\begin{definition}[{\em \color{blue} Statistics $\quadinv$ and $\eta$ on composition diagrams}]
			Consider a super filling $\widehat{\sigma}$ of $\widehat{\dg}'(\gamma)$, where the bottom row is always filled with positive integers $1$ to $n$ from left to right. Let $\widehat{\sigma}_i$ be the $i$th rectangle of $\widehat{\sigma}$ for $1\le i\le k$ in (\ref{E:dec_wide}). For each $\vartheta\in \{\quadinv,\eta\}$, we define
			\begin{align*}
				\vartheta(\widehat{\sigma})&=\sum_{i=1}^{k} \vartheta(\widehat{\sigma}_{i})+\sum_{(a,b,c)\in\widehat{\sigma}}\Q(a,b,c),
			\end{align*}
			which is summed over triples $(a,b,c)$ of types $A$ and $B$ located  in different rectangles of  $\widehat{\sigma}$.
		\end{definition}
		
		\begin{remark} {For the filling of $\widehat{\dg}'(\gamma)$, we count only quadruples and triples contained
				in $\widehat{\dg}'(\gamma)$, including row $0$, without adding  $0$ or $\infty$.}
			Let $\widehat{\sigma}$ be a non-attacking positive filling of $\widehat{\dg}'(\gamma)$. We explain how the statistics $\quadinv$ and $\eta$ reduce to those for fillings of a partition diagram.
			If $\gamma$ is weakly increasing, then there are no triples of type $B$ in $\widehat{\sigma}$. We reverse the order of the rectangles in $\widehat{\sigma}$, obtaining $\widehat{\sigma}_{k} \sqcup \cdots \sqcup \widehat{\sigma}_{1}$, call this $\tau$. Then all  triples of type $A$ in $\widehat{\sigma}$ are inversion triples in $\tau$, $\quadinv(\widehat{\sigma})=\quadinv(\tau)$ and
			$\eta(\widehat{\sigma})=\eta^*(\tau)$.
			
			If $\gamma$ is weakly decreasing, then there are no triples of type $A$ located in two columns of unequal height in $\widehat{\sigma}$ and all triples of type $B$ are queue inversion triples. That is, $\eta$ reduces to (\ref{E:etai}).
		\end{remark}
		For every box $u=(i,j)\in \dg'(\gamma)$, we define
		\begin{align*}
			\arm''(u)&=|\{(i,k)\in \dg'(\gamma): k<j, \,\gamma_k\le\gamma_j \}| +  |\{(i-1,k)\in \widehat{\dg}'(\gamma): k>j, \,\gamma_j>\gamma_k \}| \\
			&\qquad +|\{(i,k)\in \dg'(\gamma): k>j, \, (i,k) \ \textrm{is unrestricted},  \textrm{ and } \,\gamma_m=\gamma_k \textrm{ for } j\le m\le k  \}|.
		\end{align*}
		Throughout this section, for a statistic $\vartheta$ on fillings of $\widehat{\dg}'(\gamma)$, define its complement by
		\[
		\overline{\vartheta}(\widehat{\sigma})=\sum_{u\in\dg'(\gamma)}\widetilde{\arm}(u)-\vartheta(\widehat{\sigma}).
		\]
		Finally, we sketch the proof of Theorem \ref{thmE}, as it follows from analogous argument developed in Sections \ref{S:5} -- \ref{S:7}.
		
		{\em Proof of Theorem \ref{thmE}}.
		We begin by proving
		\begin{align}
			&\quad\sum_{\sigma:\gamma\rightarrow [n] \atop  \widehat{\sigma} \ \quinv\text{-non-attacking}}x^{\sigma} q^{\maj(\widehat{\sigma})}  t^{\overline{\coinv'}(\widehat{\sigma})}
			\prod_{u\in\dg'(\gamma)\atop \widehat{\sigma}(u)\ne\widehat{\sigma}(\South(u))}(1-t)
			\prod_{u\in\dg'(\gamma)\atop \widehat{\sigma}(u)=\widehat{\sigma}(\South(u))} (1-q^{\leg(u)+1} t^{\widetilde{\arm}(u)+1})\notag\\
			&=\sum_{\sigma:\gamma\rightarrow \A^n \atop  \widehat{\sigma} \ \quinv\text{-non-attacking}} x^{|\sigma|} q^{\maj(\widehat{\sigma})}  t^{\overline{\coinv'}(\widehat{\sigma})} (-t)^{m(\sigma)}, \label{E:2}
		\end{align}
		similar to the proof of (\ref{eqthm1.2}) in Theorem \ref{thm1}. Then we continue by showing
		\begin{align*}
			\sum_{\sigma:\gamma\rightarrow \A^n \atop  \widehat{\sigma} \ \quinv\text{-non-attacking}} x^{|\sigma|} q^{\maj(\widehat{\sigma})}  t^{\overline{\coinv'}(\widehat{\sigma})} (-t)^{m(\sigma)}
			=\sum_{\sigma:\gamma\rightarrow \A^n \atop \widehat{\sigma} \ \text{non-attacking}} x^{|\sigma|} q^{\maj(\widehat{\sigma})}  t^{\overline{\quadinv}(\widehat{\sigma})} (-t)^{m(\sigma)}
		\end{align*}
		much in the spirt of (\ref{eq1.1}) in Theorem \ref{thm11}. Finally, we complete the proof by establishing
		\begin{align*}
			&\sum_{\sigma:\gamma\rightarrow \A^n \atop \widehat{\sigma} \ \text{non-attacking}} x^{|\sigma|} q^{\maj(\widehat{\sigma})}  t^{\overline{\quadinv}(\widehat{\sigma})} (-t)^{m(\sigma)} \notag \\
			&=\sum_{\tau:\gamma\rightarrow [n] \atop \widehat{\tau} \ \text{non-attacking}} x^{\tau} q^{\maj(\widehat{\tau})}  t^{\overline{\quadinv}(\widehat{\tau})}
			\prod_{u\in\dg'(\gamma)\atop \widehat{\tau}(u)\ne\widehat{\tau}(\South(u))}(1-t)
			\prod_{u\in\dg'(\gamma)\atop \widehat{\tau}(u)=\widehat{\tau}(\South(u))} (1-q^{\leg(u)+1} t^{\arm''(u)+1}),
		\end{align*}
		following the strategy used to prove (\ref{E:2}).
		{By an argument analogous to the proof of Theorem~\ref{thm6.1}, the statistic $\overline{\quadinv}$ on the right-hand side
			of the preceding identity can be replaced by $\eta^\circ$. The corresponding bijection fixes the row $0$ and preserves the arm-length factors.}
		\qed
		
		When $\gamma$ is a partition, we obtain the following consequence. For the case $\vartheta=
		\eta^{\circ}$ where $\S=\S_7$, Corollary \ref{cor:10.1} is equivalent to the formula in \cite[Proposition 1.10 and Theorem 5.9]{CMW22}.
		
		\begin{corollary}\label{cor:10.1}
			For a partition $\lambda$, and $\vartheta\in \{\overline{\quadinv},\eta^{\circ}\}$,
			\begin{align*}
				E_{\lambda}(X;q,t)&=
				\sum_{\sigma:\lambda\rightarrow [n] \atop \widehat{\sigma} \ \text{non-attacking}}x^{\sigma} q^{\maj(\widehat{\sigma})}  t^{\vartheta(\widehat{\sigma})}
				\prod_{u\in\dg'(\lambda)\atop \widehat{\sigma}(u)\ne\widehat{\sigma}(\South(u))} \frac{(1-t)}{  1-q^{\leg(u)+1} t^{\overline{\arm}(u)+1}   }.
			\end{align*}
		\end{corollary}

		\begin{proof}
			It is a direct consequence of (\ref{E:quadinv}) and
			\begin{align*}
				\prod_{u\in\dg'(\lambda)}(1-q^{\leg(u)+1} t^{\widetilde{\arm}(u)+1})&= \prod_{u\in\dg'(\lambda)\atop \widehat{\sigma}(u)=\widehat{\sigma}(\South(u))} (1-q^{\leg(u)+1} t^{\arm''(u)+1}) \\
				&\qquad \times\prod_{u\in\dg'(\lambda)\atop \widehat{\sigma}(u)\ne\widehat{\sigma}(\South(u))} (1-q^{\leg(u)+1} t^{\overline{\arm}(u)+1}).
			\end{align*}
			The latter one is proved by a similar argument as for (\ref{E:b1}).
		\end{proof}
		
		\section{Proof of Theorem \ref{thm:twoP}}\label{S:8}
		This section provides alternative proofs of Theorem \ref{thm:twoP}.
		The combinatorial formula (\ref{macdonald3}) for $P_{\lambda}(X;q,t)$ is established in \cite[Equation (4.10)]{CHO22}, and it follows from \cite[Theorem 1.10 and Proposition 4.1]{CMW22}. Specifically,
		let $\textrm{inc}(\gamma)$ denote the weak composition obtained by sorting the parts of $\gamma$ in weakly increasing order. Then (\ref{macdonald3}) follows from
		\begin{align*}
			P_{\lambda}(X;q,t)=\sum_{\gamma^+=\lambda} E_{\textrm{inc}(\gamma)}^{\sigma}(X;q,t),
		\end{align*}
		where the sum is over all weak compositions $\gamma=(\gamma_1,\ldots,\gamma_n)\in \mathbb{N}^n$
		satisfying $\gamma^+=\lambda$, i.e., the partition determined by the positive parts of $\gamma$ is $\lambda$. Here, $E_{\textrm{inc}(\gamma)}^{\sigma}(X;q,t)$ is a permuted-basement nonsymmetric Macdonald polynomial in \cite{Alex:16,Fer:11}, and $\sigma$ is the longest permutation for which $\gamma_{\sigma(1)}\le \cdots\le \gamma_{\sigma(n)}$.
		
		To state (\ref{macdonald3}) precisely, we first define the $\coinv^*$ statistic.
		\begin{definition}[{\em \color{blue} Statistic $\coinv^*$}]
			For a filling $\sigma\in \T(\lambda')$, we consider the triples of entries in the augmented filling $\hat{\sigma}$ (in Definition \ref{Def:inv}). Let $\coinv^*(\sigma)$ count the number of triples $(a,b,c)$ of types $A$ and $B$ (see Definition \ref{Def:coinv}) in $\hat{\sigma}$ such that $\CMcal{Q}(a,b,c)=0$ and $c\ne 0$.
		\end{definition}
		\begin{remark}
			Equation (\ref{macdonald3}) in \cite[Equation (4.10)]{CHO22} is formulated using the diagram of $\textrm{inc}(\lambda)$, in which the column-lengths are $\lambda_n,\ldots,\lambda_1$ from left to right for a partition $\lambda=(\lambda_1\ge\cdots\ge\lambda_n)$. Here (\ref{macdonald3}) is stated using the diagram $\dg'(\lambda)$, with column-lengths $\lambda_1,\ldots,\lambda_n$ from left to right.
		\end{remark}
		The proof of (\ref{macdonald3}) begins with a new combinatorial formula for the modified Macdonald polynomials $\widetilde{H}_{\lambda}(X;q,t)$ given in Theorem \ref{thm1a}. For any super filling $\sigma$, let $\minv(\sigma)$ denote the number of inversion triples arising from two columns of equal height, and queue inversion triples from two columns of unequal height. 
		\begin{theorem}\label{thm1a}
			Let $[\sigma]$ denote the row-equivalency class of $\sigma$. Then
			\begin{align}\label{E:mixinv1}
				\sum_{\tau\in[\sigma]}t^{\maj(\tau)}q^{\minv(\tau)}=\sum_{\tau\in[\sigma]}t^{\maj(\tau)}q^{\quinv(\tau)}.
			\end{align}
			Consequently, the modified Macdonald polynomial is given by
			\begin{align}\label{E:tildeH1}
				\widetilde{H}_{\lambda}(X;q,t)=\sum_{\tau\in \T(\lambda)}x^{\tau}t^{\maj(\tau)}q^{\minv(\tau)}.
			\end{align}
		\end{theorem}
		\begin{proof}
			If the diagram $\dg(\lambda)$ is a rectangle, then a bijection $\theta:\T(\lambda)\rightarrow \T(\lambda)$ exists that satisfies $\sigma\sim \theta(\sigma)$ and
			\begin{align}\label{E:sym1}
				(\inv,\quinv,\maj)\sigma=(\quinv,\inv,\maj)\theta(\sigma)
			\end{align}
			as defined in \cite[Theorem 6.1]{JL24}. For a non-rectangular filling
			$\sigma=\sigma_n\sqcup \cdots \sqcup \sigma_1$ of the Young diagram $\dg(\lambda)$, we define
			\begin{align*}
				\theta(\sigma)=\theta(\sigma_n)\sqcup \cdots \sqcup \theta(\sigma_1).
			\end{align*}
			Since $\theta$ is a product of queue-inversion flip operators, the number, denoted by $a$, of queue inversion triples arising from two columns of unequal height is invariant under $\theta$ by Lemma \ref{lem3.1}. This implies
			\begin{align*}
				\minv(\sigma)=\sum_{1\le i\le n}\inv(\sigma_i)+a =\sum_{1\le i\le n}\quinv(\theta(\sigma_i))+a
				=\quinv(\theta(\sigma)),
			\end{align*}
			and $\maj(\sigma)=\maj(\theta(\sigma))$ by (\ref{E:sym1}), as desired.
		\end{proof}
		We find the superization of $\widetilde{H}_{\lambda}(X;q,t)$ given in (\ref{E:supJ1}) as an immediate consequence of  (\ref{E:supJ}) and (\ref{E:mixinv1}) in the super-filling setting.
		\begin{corollary}
			For a partition $\lambda$, let $\M(\lambda')$ be the set of super fillings of $\dg'(\lambda)$. Then
			\begin{align}\label{E:supJ1}
				J_{\lambda}(X;q,t)
				&=t^{n(\lambda)+n}\sum_{\sigma\in\M(\lambda')} (-1)^{m(\sigma)} t^{-p(\sigma)-\minv(\sigma)}q^{\maj(\sigma)} x^{|\sigma|}.
			\end{align}
		\end{corollary}
		\begin{proof}
			The bijection $\theta$ can be applied to the set $\M(\lambda')$ of super fillings as well. Consequently, (\ref{E:sym1}) holds for $\sigma\in \M(\lambda')$, which implies $\minv(\sigma)=\quinv(\theta(\sigma))$ and $\maj(\sigma)=\maj(\theta(\sigma))$. Therefore, (\ref{E:supJ1}) follows from (\ref{E:supJ}), (\ref{E:HtoJ}) and the identity $\widetilde{H}_{\lambda}(X;q,t)=\widetilde{H}_{\lambda'}(X;t,q)$ in \cite[Proposition 2.5]{Haiman:99}.
		\end{proof}
		Having established (\ref{E:supJ1}), we now prove (\ref{macdonald3}). First, we introduce the necessary notation and definitions.
		\begin{definition}[{\em \color{blue} $\minv$-non-attacking super fillings}]\label{Def:minv_attack}
			For the Young diagram $\dg'(\lambda)$, two boxes $u,v\in \dg'(\lambda)$ are said to be {\em $\minv$-attacking} each other if either
			\begin{enumerate}
				\item[$\mathrm{(I)}$]  $u=(i,j)$, $v=(i,k)$, where $j\ne k$; or
				\item[$\mathrm{(II)}$] $u=(i+1,j)$, $v=(i,k)$, where $k>j$ and $\lambda_j>\lambda_{k}$; or
				\item[$\mathrm{(III)}$] $u=(i+1,j)$, $v=(i,k)$, where $k<j$ and $\lambda_j=\lambda_k$.
			\end{enumerate}
			A $\minv$-non-attacking super filling $\sigma$ is a super filling without two $\minv$-attacking boxes $u,v$ such that $|\sigma(u)|=|\sigma(v)|$.
		\end{definition}
		\begin{definition}[{\em \color{blue} Reading order}]\label{Def:reading}
			The {\em reading order} is the total ordering on the boxes of Young diagram given by reading them row by row, top to bottom, and left to right within each row of the rectangle, starting from the rightmost rectangle. See Figure \ref{F:reading} for an example.
			\begin{figure}[H]
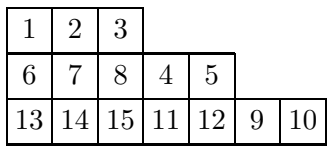

				\centering
				\begin{align*}
					\begin{ytableau}[]1 & 2& 3  \\  6&7&8&4& 5 \\ 13 &14 &15 & 11 &12& 9 &10
					\end{ytableau}
				\end{align*}
				\caption{A diagram with boxes labelled by its reading order.}\label{F:reading}
			\end{figure}
		\end{definition}
		
		\begin{lemma}\label{lem1}
			For any partition $\lambda$,
			let $\M(\lambda')$ be the set of super fillings of $\dg(\lambda')$ and let $\M_1(\lambda')\subseteq \M(\lambda')$ be the set of super fillings $\sigma=\sigma_{\lambda_1}\sqcup \cdots \sqcup \sigma_1$ whose bottom-row entries of each rectangle $\sigma_j$ {have distinct absolute values} for $1\le j\le \lambda_1$. Then,
			\begin{align}\label{eqlem1.1}
				&\quad\sum_{\sigma\in\M(\lambda')}  (-1)^{m(\sigma)} t^{-p(\sigma)-\minv(\sigma)}q^{\maj(\sigma)} x^{|\sigma|}\notag\\
				&=\sum_{\sigma\in\M_1(\lambda')} (-1)^{m(\sigma)} t^{-p(\sigma)-\minv(\sigma)}q^{\maj(\sigma)} x^{|\sigma|}.
			\end{align}
			Let $\M_2(\lambda')\subseteq \M_1(\lambda_1)$ be the set of super fillings $\sigma=\sigma_{\lambda_1}\sqcup \cdots \sqcup \sigma_{1}$ whose bottom-row {absolute values} of each rectangle $\sigma_j$ are strictly decreasing. Then,
			\begin{align}\label{eqprop1.1}
				&\sum_{\sigma\in\M_2(\lambda')} (-1)^{m(\sigma)} t^{-p(\sigma)-\minv(\sigma)}q^{\maj(\sigma)} x^{|\sigma|} \notag\\
				&=\sum_{\sigma\in\M_2(\lambda') \atop \sigma \ \minv\text{-non-attacking}} (-1)^{m(\sigma)} t^{-p(\sigma)-\minv(\sigma)}q^{\maj(\sigma)} x^{|\sigma|}.
			\end{align}
		\end{lemma}
		\begin{proof}
			Both (\ref{eqlem1.1}) and (\ref{eqprop1.1}) are proved by a sign-reversing and weight-preserving involution. This involution first appeared in \cite[Lemma 5.1]{HHL04} and was also applied in \cite[Theorem 5.3]{AMM23}.
			
			We begin by proving (\ref{eqlem1.1}). The involution $\phi:\M(\lambda')\rightarrow \M(\lambda')$ is defined as follows. For any $\sigma\in \M_1(\lambda')$, we define $\phi(\sigma)=\sigma$. {Otherwise, choose $v$ to be the last bottom-row box in the reading order that $\minv$-attacks a bottom-row box with the same absolute entry. Choose $u$ to be the last bottom-row box in the reading order that $\minv$-attacks $v$ and satisfies $|\sigma(u)|=|\sigma(v)|$. Define $\phi(\sigma)$ to be the super filling obtained by flipping the sign of $\sigma(u)$.}
			This is clearly a sign-reversing and weight-preserving involution on $\M(\lambda')-\M_1(\lambda')$. It remains to examine the statistics $\minv$ and $\maj$.
			
			Since $I(a,b)=I(a,\overline{b})$ for all $a,b\in\A$, the major index remains invariant. Without loss of generality, we assume that $\sigma(u)=k\in\Z_+$. It suffices to prove that flipping the sign from $k$ to $\bar{k}\in \mathbb{Z}_{-}$ increases the number of $\minv$-triples by exactly one, thus preserving the quantity $p(\sigma)+\minv(\sigma)$ under $\phi$.
			
			Now assume that the columns of $u$ and $v$ are of equal height. First we consider the inversion triple involving $\sigma(v)$. Since $\Q(k,\infty,\sigma(v))=0$ and $\Q(\overline{k},\infty,\sigma(v))=1$, the number of inversion triples increases by one after flipping the sign. We claim that the change from $k$ to $\bar{k}$ does not affect any other triple containing box $u$.
			
			For any triple $\begin{ytableau} a \\ b  & \none[\dots]& k \\\end{ytableau}$ where the column of $(a,b)$ is longer than the column of $k$, we observe that $|b|\ne k$, because otherwise it contradicts the choice of $u$. Since $|b|\ne k$, we have $\Q(a,b,k)=\Q(a,b,\overline{k})$ for $a,b\in \A$. Similarly,  $\Q(c,k,d)=\Q(c,\overline{k},d)$ holds for any $c,d\in\A$.
			It follows that $\phi$ keeps the number of queue inversion triples located in two columns of unequal height invariant.
			
			For any triple $\begin{ytableau} a & \none[\dots]& c\\ b   \\\end{ytableau}$
			including $k$ but excluding $\sigma(v)$, such that the column lengths of $(a,b)$ and $c$ are equal,
			we have three scenarios.
			If $b=k$, then $\Q(a,k,c)=\Q(a,\overline{k},c)$ by noting $I(a,k)=I(a,\overline{k})$ and $I(c,k)=I(c,\overline{k})$. If $c=k$, then $b=\infty$ as box $u$ is in the bottom row. In view of $I(a,k)=I(a,\overline{k})$, we are led to $\Q(a,b,k)=\Q(a,b,\overline{k})$. If $a=k$, then again $b=\infty$. The choice of $u$ determines $|c|\ne k$, thus yielding $I(k,c)=I(\overline{k},c)$ and $\Q(k,\infty,c)=\Q(\overline{k},\infty,c)$. This proves the case when the columns of $u$ and $v$ are of equal height.

			

			Similarly, if $u$ and $v$ are in two columns of unequal height, then the number of queue inversion triples from two columns of distinct heights increases by one,  while the number of inversion triples from two columns of equal height remains invariant under $\phi$. This completes the proof of (\ref{eqlem1.1}).
			
			The involution for proving (\ref{eqprop1.1}) is slightly different. Define the map $\Phi:\M_2(\lambda')\rightarrow \M_2(\lambda')$ as follows. For any $\sigma\in \M_2(\lambda')$ that is $\minv$-non-attacking, let $\Phi(\sigma)=\sigma$. 
			{Otherwise, let $a$ be the smallest positive integer for which there are two $\minv$-attacking boxes with absolute entry $a$. Choose $v$ to be the last box in the reading order that $\minv$-attacks another box with absolute entry $a$ and satisfies $|\sigma(v)|=a$. Choose $u$ to be the last box in the reading order that $\minv$-attacks $v$ and satisfies $|\sigma(u)|=a$.} Define $\Phi(\sigma)$ as the super filling obtained by changing the sign of $\sigma(u)$. The remaining argument is analogous to the proof of (\ref{eqlem1.1}); hence, we omit the details.
		\end{proof}
		We now relate (\ref{eqlem1.1}) to (\ref{eqprop1.1}) by the following lemma.
		\begin{lemma}\label{lem2}
			For a partition $\lambda$ of $n$, and $\M_1(\lambda'), \M_2(\lambda')$ defined in Lemma~\ref{lem1}. Then,
			\begin{align} \label{eqlem2.1}
				J_{\lambda}(X;q,t)
				&=t^{n(\lambda)+n} \prod_{i\ge 1}[m_i]_t!\notag\\
				&\qquad\times\sum_{\sigma\in\M_2(\lambda')} (-1)^{m(\sigma)} t^{-p(\sigma)-\minv(\sigma)}q^{\maj(\sigma)} x^{|\sigma|}.
			\end{align}
			where $m_i$ is the multiplicity of part $i$ in $\lambda$ for $i\ge 1$.
		\end{lemma}
		\begin{proof}
			From \eqref{E:supJ1} and \eqref{eqlem1.1}, it is equivalent to establish that
			\begin{align}\label{E:M12}
				&\quad \sum_{\tau\in\M_1(\lambda')} (-1)^{m(\tau)} t^{-p(\tau)-\minv(\tau)}q^{\maj(\tau)} x^{|\tau|}\notag\\
				&=\prod_{i\ge 1}[m_i]_t! \sum_{\sigma\in\M_2(\lambda')} (-1)^{m(\sigma)} t^{-p(\sigma)-\minv(\sigma)}q^{\maj(\sigma)} x^{|\sigma|}.
			\end{align}
			The argument is analogous to that for (\ref{eqthm2.2}). Recall the flip operators $\zeta_{w}=\zeta_w^1$, which are indexed by a permutation $w$ as defined in (\ref{E:wflip2}).
			We note that
			\begin{align}\label{E:zeta1}
				\zeta_i(\sigma)=\rho_i^k(\sigma),
			\end{align}
			where $\zeta_i(\sigma)=t_i^{[1,k]}(\sigma)$. Consequently, $\zeta_w$ is a product of queue-inversion flip operators and, by Lemma \ref{lem3.1}, preserves the number of queue inversion triples arising from two columns of unequal height.
			
			{For any super filling $\sigma=\sigma_{\lambda_1}\sqcup\cdots\sqcup\sigma_1\in\M_2(\lambda')$, let $a=(a_1>\cdots>a_n)$ be the bottom row of a rectangle $\sigma_j$. For $w\in S_n$, put $w(a)=(a_{w^{-1}(1)},\ldots,a_{w^{-1}(n)})$.
				Lemma \ref{lem3.1} and (\ref{E:zeta1}) imply that
				\begin{align}\label{E:sort3}
					\maj(\zeta_w(\sigma))&=\maj(\sigma),\\
					\minv(\zeta_{w}(\sigma))-\minv(\sigma)
					&=\inv(w(a))-\inv(a)=-\inv(w).\label{E:sort4}
				\end{align}
				For any super filling $\tau\in\M_1(\lambda')$, the distinct bottom-row absolute values determine a unique $w\in S_{m_{\lambda_1}}\times\cdots\times S_{m_1}$ for which the bottom rows of $\tau$ are $w(a)$, where $a$ denotes their decreasing rearrangements. This gives a unique super filling $\sigma=\zeta_w^{-1}(\tau)\in\M_2(\lambda')$.
				It follows from (\ref{E:sort3}) and (\ref{E:sort4}) that
				\begin{align*}
					\sum_{w\in S_{m_{\lambda_1}}\times \cdots \times S_{m_1}} q^{\maj(\zeta_w(\sigma))}t^{-\minv(\zeta_w(\sigma))}x^{|\zeta_w(\sigma)|}
					&=q^{\maj(\sigma)}t^{-\minv(\sigma)}x^{|\sigma|} \sum_{w\in S_{m_{\lambda_1}}\times \cdots \times S_{m_1}}t^{\inv(w)}\\
					&=\prod_{i\ge 1}[m_i]_t!\,q^{\maj(\sigma)}t^{-\minv(\sigma)}x^{|\sigma|},
				\end{align*}
			}
			and we conclude (\ref{E:M12}) by noting $p(\tau)=p(\sigma)$ and $m(\tau)=m(\sigma)$.
		\end{proof}
		Finally, we compress the terms on the right-hand side of (\ref{eqprop1.1}) to restrict the sum to $\minv$-non-attacking positive fillings.
		
		\begin{lemma}\label{prop2}
			Let $\lambda$ be a partition of $n$, then
			\begin{align}\label{eqprop2.1}
				&\sum_{\sigma\in\M_2(\lambda') \atop \sigma \ \minv\text{-non-attacking}} (-1)^{m(\sigma)} t^{n-p(\sigma)-\minv(\sigma)}q^{\maj(\sigma)} x^{|\sigma|} \notag\\
				&= \sum_{\tau\in \T(\lambda')\cap\M_2(\lambda')  \atop \tau \ \minv\textrm{-non-attacking} } q^{\maj(\tau)} t^{-\minv(\tau)} x^{\tau} \notag\\
				&\,\qquad\times\prod_{\tau(z)=\tau(\South(z))\atop \text{ and }z \not\in\, \text{row } 1} (1-q^{\leg(z)+1} t^{1+\arm(\South(z))}) \prod_{\tau(z)\ne \tau(\South(z))\,\atop\text{ or }\,z \in\, \text{row } 1}(1-t).
			\end{align}
		\end{lemma}
		\begin{proof}
			For any $\minv$-non-attacking positive filling $\tau\in\T(\lambda')$, we sum the contributions from all $\minv$-non-attacking super fillings $\sigma\in\M(\lambda')$ satisfying $|\sigma|=\tau$.
			
			Each super filling $\sigma$ with $|\sigma|=\tau$ is generated by choosing,  for every box $z$ (processed from top to bottom, and left to right within each row), either $\sigma(z)=\tau(z)\in \mathbb{Z}_{+}$ or $\sigma(z)=\overline{\tau(z)}\in \mathbb{Z}_{-}$.
			
			For each restricted box $z$, set $u=\tau(z)=\tau(\South(z))$. If $\sigma(z)=u\in\Z_+$, then we have $I(u,\sigma(\South(z)))=I(u,\tau(\South(z)))=I(u,u)=0$; hence, a positive entry $\tau(z)$ in box $z$ makes no contribution to $\maj(\sigma)-\maj(\tau)$. If $\sigma(z)=\overline{u}\in \Z_{-}$, then $I(\overline{u},\sigma(\South(z)))=1$ while $I(u,\tau(\South(z)))=0$, so that $\maj(\sigma)-\maj(\tau)$ increases by $\leg(z)+1$. Next, we compute the contribution of $\sigma(z)$ to $\minv(\sigma)-\minv(\tau)$.
			
			If $\sigma(z)=u\in\Z_+$, we consider the triples $(u,v,x)=\begin{ytableau} u & \none[\dots]& x\\ v   \\\end{ytableau}$
			of $\sigma$ in two columns of equal height, and the triples $(u,v,x)=\begin{ytableau} u \\ v  & \none[\dots]& x \\\end{ytableau}$
			in two columns of unequal height.
			In both types, $u=|v|$ and the $\minv$-non-attacking condition implies that $|v|\ne |x|$. Consequently, $\Q(u,v,x)=\Q(u,|v|,|x|)$ for all $v,x\in \A$. Therefore, fixing $\sigma(z)=u$ does not alter the number of inversion and queue inversion triples; hence, a positive entry $\sigma(z)$ contributes zero to $\minv(\sigma)-\minv(\tau)$.
			
			{If $\sigma(z)=\overline{u}\in\Z_-$, then $|u|=|v|\ne |x|$ by the $\minv$-non-attacking condition. From \eqref{eq:a=b}, we have $\Q(\overline{u},v,x)=0$ and $\Q(|u|,v,x)=1$. Thus, for each possible $x$, the number of such triples decreases by one. Since there are $\arm(\South(z))$ choices for $x$, a negative entry $\sigma(z)$ contributes $-\arm(\South(z))$ to $\minv(\sigma)-\minv(\tau)$.}
			
			For each unrestricted box $z$, we have $\tau(z)\ne \tau(\South(z))$. Consequently, $I(\sigma(z),\sigma(\South(z)))$ $=I(\tau(z), \tau(\South(z)))$ for all $\sigma(z)\in \A$, and thus box $z$ makes no contribution to $\maj(\sigma)-\maj(\tau)$.
			Let $u=\sigma(z)$ and $v=\sigma(\South(z))$, then $|u|\ne |v|$. Since $\tau$ is $\minv$-non-attacking, the absolute values $|u|$, $|v|$ and $|x|$ are distinct for any triple $(u,v,x)$ in $\sigma$. Hence, $\Q(u,v,x)=\Q(|u|,|v|,|x|)$, which means that $\sigma(z)$ contributes zero to $\minv(\sigma)-\minv(\tau)$.
			
			In terms of generating function, the preceding analysis yields
			\begin{align*}
				&t^{n}\sum_{\sigma\in\M_2(\lambda') \atop \sigma \ \minv\text{-non-attacking}} (-1)^{m(\sigma)} t^{-p(\sigma)-\minv(\sigma)}q^{\maj(\sigma)} x^{|\sigma|} \\
				&=t^{n} \sum_{\tau\in \T(\lambda')\bigcap\M_2(\lambda') \atop  \tau\  \minv\textrm{-non-attacking} } q^{\maj(\tau)} t^{-\minv(\tau)} x^{\tau} \\
				&\qquad\times
				\prod_{\tau(z)\ne \tau(\South(z))\atop \text{ or } z \in\, \text{row }1 }(t^{-1}-1)  \prod_{\tau(z)=\tau(\South(z)) \atop \text{ and }z \not\in\, \text{row }1 } (t^{-1}-q^{\leg(z)+1} t^{\arm(\South(z))} ),
			\end{align*}
			which is equivalent to \eqref{eqprop2.1}.
		\end{proof}
		We are now ready to prove (\ref{macdonald3}) and (\ref{macdonald2}) of Theorem \ref{thm:twoP}.
		
		{\em Proof of (\ref{macdonald3}) and (\ref{macdonald2})}. Equations \eqref{eqlem2.1}, \eqref{eqprop1.1} and \eqref{eqprop2.1} prove that
		\begin{align}\label{eqpfthm}
			J_{\lambda}(X;q,t)&=\prod_{i\ge 1}[m_i]_t!\, t^{n(\lambda)} \sum_{\tau\in \T(\lambda')\bigcap\M_2(\lambda')  \atop \tau \ \minv\textrm{-non-attacking} } q^{\maj(\tau)} t^{-\minv(\tau)} x^{\tau} \notag\\
			&\qquad \times\prod_{\tau(z)=\tau(\South(z)) \atop \text{ and }z \not\in \,\text{row} 1} (1-q^{\leg(z)+1} t^{1+\arm(\South(z))}) \prod_{\tau(z)\ne \tau(\South(z))\atop \text{ or }z \in \,\text{row} 1}(1-t).
		\end{align}
		In combination with (\ref{E:JP}) and (\ref{eqJ.2}), we find
		\begin{align} \label{eqP.2}
			P_{\lambda}(X;q,t)&=\sum_{\sigma\in \T(\lambda')\cap\M_2(\lambda')  \atop \sigma \ \minv\textrm{-non-attacking} } q^{\maj(\sigma)} t^{\overline{\minv}(\sigma)} x^{\sigma}\notag\\
			&\qquad\times\prod_{\sigma(z)\ne\sigma(\South(z)) \atop z \not\in \,\textrm{row }1}  \frac{1-t}{1-q^{\leg(z)+1} t^{\arm(\South(z))+1}}.
		\end{align}
		We shall prove the equivalence of \eqref{macdonald3} and \eqref{eqP.2}.
		By Definitions \ref{Def:minv_attack} and \ref{Def:non-attacking},  a positive filling $\sigma$ is $\minv$-non-attacking, if and only if $\rev(\sigma)$ is $\quinv$-non-attacking. Moreover, the bottom-row entries of each rectangle $\sigma_j$ of $\sigma$ are strictly decreasing, if and only if the bottom-row entries of each rectangle $\rev(\sigma_j)$ are strictly increasing.
		
		Since each type $A$ triple of $\rev(\sigma)$ corresponds to a non-inversion triple
		of $\sigma$, while each type $B$ triple of $\rev(\sigma)$ corresponds to a non-queue-inversion triple of $\sigma$, it follows that $\overline{\minv}(\sigma)=\coinv^*(\rev(\sigma))$. Finally, consider a box $u$ of $\sigma$ above the bottom row, and let $v$ be its corresponding box under the reverse map. Then, by Definition \ref{Def:arm-lengths}, $\arm(\South(u))=\widetilde{\arm}(v)$. Consequently, (\ref{macdonald3}) and (\ref{eqP.2}) are equal, completing the proof of (\ref{macdonald3}).
		
		The proof of (\ref{macdonald2}) proceeds analogously; we merely indicate the necessary modifications to Lemmas \ref{lem1} -- \ref{prop2}. For the statistic $\quinv$, define $\M_1(\lambda')$ (resp. $\M_2(\lambda')$) as the set of super fillings $\sigma=\sigma_{\lambda_1}\sqcup\cdots\sqcup \sigma_1$ in which the {top-row absolute values of each rectangle $\sigma_j$ are distinct (resp. strictly increasing).} Then Lemmas \ref{lem1} -- \ref{prop2} remain valid upon replacing $\minv$ with $\quinv$. Specifically, in the proof of Lemma \ref{lem2}, the inversion flip operator $\zeta_w$ is replaced by the queue inversion flip operator $\rho_w$. This establishes (\ref{macdonald2}).
		\qed
		\section{Final remarks}\label{S:finalremark}
		
		Corteel, Mandelshtam and Williams posed an open problem \cite[Remark 1.14]{CMW22} on a multiline queue formula for nonsymmetric Macdonald polynomials $E_{\gamma}(X;q,t)$, as a generalization of their formula for $E_{\lambda}(X;q,t)$ indexed by a partition $\lambda$. Their multiline queue formula is equivalent to the one in Corollary \ref{cor:10.1} for  $\vartheta=
		\eta^{\circ}$ where $\S=\S_7$. This problem remains open, though our tableau formula (\ref{E:quadinv}) for $E_{\gamma}(X;q,t)$ extends theirs.
		
		Our $(q,t)$-formula (\ref{E:vsP}) for the coefficient $c_{\lambda\mu}(q,t)=[m_{\mu}]P_{\lambda}$ is different from the summation formula of de Gier and Wheeler \cite{dGW:15}. Moreover, as in Theorem \ref{cor:1}, we can also reduce the sum to sorted non-attacking tableaux in the combinatorial formula (\ref{E:quadinv}) for $E_{\lambda}(X;q,t)$ when $\lambda$ is a partition. The reason we cannot provide such a formula for an arbitrary weak composition $\gamma$ is that the coefficient $d_{\sigma}(t)$ would become complicated if the diagram of $\gamma$ is not a partition shape.
		
		\section*{Acknowledgements}
		We would like to thank Alexandr Garbali for bringing the paper \cite{dGW:15} to our attention. Both authors are supported by the National Nature Science Foundation of China (NSFC), Projects No. 12201529 and No. 12571358.

	\end{document}